\documentstyle[11pt,amssymb,amsmath,amscd,amsbsy,amsfonts,theorem,hyperref,accents,fontenc]{article}

\input xy
\xyoption{all}

\newcommand{\zz}{{\Bbb Z}}

\newcommand{\nn}{\Bbb N}

\newcommand{\pp}{{\Bbb P}}
\newcommand{\aaa}{{\Bbb A}}
\newcommand{\ff}{{\Bbb F}}

\newcommand{\ddim}{\operatorname{dim}}
\newcommand{\ddeg}{\operatorname{deg}}
\newcommand{\kker}{\operatorname{Ker}}

\newcommand{\spec}{\operatorname{Spec}}

\newcommand{\Hom}{\operatorname{Hom}}
\newcommand{\HH}{\operatorname{H}}
\newcommand{\op}[1]{\operatorname{#1}}
\newcommand{\kbar}{\overline{k}}

\newcommand{\ffi}{\varphi}

\newcommand{\eps}{\varepsilon}

\newcommand{\row}{\rightarrow}
\newcommand{\llow}{\longleftarrow}
\newcommand{\low}{\leftarrow}
\newcommand{\lrow}{\longrightarrow}
\renewcommand{\leq}{\leqslant}
\renewcommand{\geq}{\geqslant}

\newcommand{\calm}{{\cal M}}
\newcommand{\hm}{\operatorname{H}_{\calm}}
\newcommand{\nichego}[1]{}

\newcommand{\ov}[1]{\overline{#1}}
\newcommand{\un}[1]{\underline{#1}}
\newcommand{\wt}[1]{\widetilde{#1}}

\newcommand{\smk}{{\mathbf{Sm}}_k}

\newcommand{\laz}{{\Bbb L}}

\newcommand{\cp}{{\cal P}}
\newcommand{\Ch}{\operatorname{Ch}}

\newcommand{\CH}{\operatorname{CH}}

\newcommand{\bub}{*=0{\bullet}}
\newcommand{\dub}{*=0{.}}

\newcommand{\dmk}{\op{DM}(k)}

\newcommand{\dmgmkF}[1]{\op{DM}_{gm}(k;#1)}

\newcommand{\dmkF}[1]{\op{DM}(k;#1)}
\newcommand{\dmEF}[2]{\op{DM}(#1;#2)}
\newcommand{\dmgmEF}[2]{\op{DM}_{gm}(#1;#2)}
\newcommand{\hii}{{\cal X}}
\newcommand{\whii}{\widetilde{{\cal X}}}
\newcommand{\DQMgm}{\op{DQM}^{gm}}

\newcommand{\dmELF}[3]{\op{DM}({#1}/{#2};{#3})}
\newcommand{\dmgmELF}[3]{\op{DM}_{gm}({#1}/{#2};{#3})}

\newcommand{\nump}{\,\stackrel{\scriptscriptstyle{Num(p)}}{\sim}\,}

\newcommand{\numeq}{\,\stackrel{\scriptscriptstyle{Num}}{\sim}\,}

\newcommand{\Qn}[1]{{\cal Q}^{-1}(#1)}
\newcommand{\Ia}{R_{\alpha}}
\newcommand{\Ib}{R_{\beta}}
\newcommand{\eend}{{\Bbb{E}}nd}

\newcommand{\bor}[1]{\stackrel{#1}{>}}

\newcommand{\subq}{\twoheadleftarrow\hspace{-0.5mm}\hookrightarrow}

\newcommand{\dmELFr}[4]{\op{DM}(\{{#1}/{#2}\}^{#4};{#3})}
\newcommand{\dmeffkF}[1]{\op{DM}^{-}_{eff}(k,F)}

\newcommand{\Qed}{\hfill$\square$\smallskip}
\newcommand{\Red}{\hfill$\triangle$\smallskip}

\newenvironment{proof}{\noindent{\it Proof}:}{\vskip 5mm}

\newtheorem{proposition}{Proposition}[section]{\bf}{\it}
\newtheorem{theorem}[proposition]{Theorem}{\bf}{\it}
{\bf}{\it}
{\bf}{\it}
\newtheorem{definition}[proposition]{Definition}{\bf}{\rm}
\newtheorem{conj}[proposition]{Conjecture}{\bf}{\it}
{\bf}{\it}
\newtheorem{example}[proposition]{Example}{\bf}{\it}
\newtheorem{remark}[proposition]{Remark}{\bf}{\rm}
\newtheorem{question}[proposition]{Question}{\bf}{\rm}
\newtheorem{lem}{Lemma}[proposition]{\bf}{\it}
{\bf}{\it}
\newtheorem{sublem}{Sublemma}[lem]{\bf}{\it}

\newtheorem{statement}[proposition]{Statement}{\bf}{\it}
\newtheorem{corollary}[proposition]{Corollary}{\bf}{\it}
{\bf}{\it}

\begin{document}

\title{Isotropic motives}
\author{Alexander Vishik}
\date{}

\maketitle

\begin{abstract}
In this article we introduce the local versions of the Voevodsky category of motives with $\ff_p$-coefficients
over a field $k$, parameterized by finitely-generated
extensions of $k$. We introduce the, so-called, {\it flexible} fields, passage to which is conservative on motives.
We demonstrate that, over flexible fields, the constructed {\it local} motivic categories are much simpler than the {\it global} one
and more reminiscent of a topological counterpart. This provides handy "local" invariants from which one can read
motivic information.
We compute the local motivic cohomology of a point, for $p=2$, and study the {\it local Chow motivic category}. We
introduce {\it local Chow groups} and conjecture
that, over flexible fields, these should coincide with {\it Chow groups modulo numerical equivalence with
$\ff_p$-coefficients}, which implies that {\it local Chow motives} coincide with
{\it numerical Chow motives}. We prove this Conjecture in various cases.
\end{abstract}

\section{Introduction}

The category of algebraic varieties is rich and marvelous, but not additive. In a sense that one can't add morphisms between varieties.
The program to "linearize" the algebro-geometric world was first introduced by A.Grothendieck in the 1960-ies, who proposed the
category of {\it Chow motives}. It is a close relative of the category of correspondences, where objects are smooth projective varieties
and morphisms are algebraic cycles on the product modulo {\it rational equivalence}. The result is a tensor additive category, as we can add
(and subtract) algebraic cycles and multiply them externally. Moreover, one doesn't have to limit oneself to only rational equivalence
of cycles. Instead, it is possible to consider {\it algebraic}, {\it numerical}, or {\it homological} equivalence, and actually, the theory
of Chow groups here can be substituted by any {\it oriented cohomology theory}
(in the sense of \cite[Def. 3.1.1]{PS} or \cite[Def. 1.1.2]{LM}).
Chow motives of varieties split into smaller pieces which permit to express in a precise form some of the similarities observed
in the behavior of different varieties. In particular, the {\it Tate-motives} appear, responsible for the {\it cellular} structure.
The above Grothendieck's category has innumerate remarkable applications, but it deals with smooth projective varieties only.
At the same time, in topology, motives were known already for a long time in full generality in the form of singular complexes of
topological spaces.
This problem was solved by V.Voevodsky in \cite{VoMo} who constructed the triangulated category of motives
$\dmk$ over a field $k$
(around the same time, alternative constructions were proposed by M.Hanamura and M.Levine). This category of Voevodsky contains the
Grothendieck's category of Chow motives as a full additive subcategory closed under direct summands, and, in particular, permits to
study these "pure" objects by triangulated methods.
Voevodsky supplied his category with many flexible tools and his motives found numerous bright applications, most notably, to
the proof of Milnor's and Bloch-Kato Conjectures.

In this article, we study Voevodsky category $\dmkF{\ff_p}$ with finite coefficients over a field of characteristic zero.
We introduce the {\it local} versions $\dmELF{E}{k}{\ff_p}$ of this category, corresponding to all finitely generated
field extensions $E/k$ together with the natural localization functors $\ffi_E:\dmkF{\ff_p}\row\dmELF{E}{k}{\ff_p}$.
In the case of a trivial extension we get the {\it isotropic motivic categories} $\dmELF{E}{E}{\ff_p}$ and localization
functors can be specialized further to isotropic functors $\psi_E:\dmkF{\ff_p}\row\dmELF{E}{E}{\ff_p}$.
Such {\it isotropic} versions (in the appropriate situation) appear to be much simpler than the original global category, and permit
to obtain "local" invariants of motives, residing in a rather simple world.

The construction of {\it isotropic motivic category} is based on the notion of an {\it anisotropic variety}, that is, a variety which
doesn't have closed point of degree prime to a given prime $p$ (so, the fact that coefficients are finite is really essential).
The rough idea is to "kill" the motives of all anisotropic varieties over $k$. This idea belongs to T.Bachmann, who in \cite{BQ}
considered the full tensor triangulated subcategory $\DQMgm$ of $\dmkF{\ff_2}$ generated by motives of smooth projective quadrics
and studied it with the help of functors $\Phi^E:\DQMgm\row K^b(Tate(\ff_2))$ to the category of bi-graded $\ff_2$-vector spaces.
These functors were defined by the property that they "kill" the motives of $k$-Quadrics, which stay anisotropic over $E$
(and act "identically" on Tate-motives). In our approach, the same idea comes naturally from the development of some ideas of Voevodsky
and that of \cite{IMQ}. Namely, we consider the natural $\otimes$-idempotents in Voevodsky category, given by
motives $\hii_Q$ of \v{C}ech simplicial schemes corresponding to smooth varieties $Q$ over $k$. The respective
projectors naturally commute with each other and form a
partially ordered set $\cp$, where $\hii_Q\geq\hii_P$ if there exists a correspondence of degree $1$ (modulo $p$) $Q\rightsquigarrow P$
(Definition \ref{D-Iso-hii-order}).
This condition is equivalent to the fact that
$\hii_Q\otimes\hii_P=\hii_Q$. That is, a stronger projector "consumes" a weaker one.
Moreover, there is a natural map $\hii_Q\row\hii_P$ (the unique lifting of the projection $\hii_Q\row T$).
For connected varieties, this is, actually, a condition on their
generic points.
The "smallest" idempotent is the unit object of the tensor
structure, given by the trivial Tate-motive $T$ (it corresponds to $P=\op{Spec}(k)$). Thus, we get a $\cp$-parameterized filtration by
idempotents on the unit object. We may consider the {\it upper graded components} of this filtration. In other words, we
take a particular idempotent $\hii_P$ and mod out all strictly large ones, that is, we consider the colimit of idempotents
$\hii_P\otimes\whii_Q$, for all $\hii_Q\gneqq\hii_P$, where $\whii_Q$ is an idempotent complementary to $\hii_Q$.
The result is a certain idempotent in $\dmkF{\ff_p}$, which actually can be described in terms of the \v{C}ech simplicial scheme of a
variety with infinitely many connected components. This idempotent will be zero, unless $P$ is connected up to equivalence
(i.e., can be replaced by a connected variety with the same $\hii$), and, in the latter case, it depends only on $E=k(P)$, and even
only on the equivalence class of such finitely generated field extension $E/k$.
Applying the respective projector to $\dmkF{\ff_p}$ we obtain the {\it motivic category} $\dmELF{E}{k}{\ff_p}$
of the extension $E/k$ - Definition \ref{D-Iso-2}.
It is naturally a full tensor localizing subcategory of the Voevodsky category $\dmkF{\ff_p}$ supplied with the
{\it localization functor} $\ffi_E:\dmkF{\ff_p}\row \dmELF{E}{k}{\ff_p}$.
In the case of a trivial extension $k/k$, we get the {\it isotropic motivic category} $\dmELF{k}{k}{\ff_p}$, where the respective projector
is the colimit of projectors $\whii_Q\otimes$, where $Q$ runs over all varieties with $\hii_Q\neq T$, or, in other words, over all
{\it anisotropic} varieties. Since $\whii_Q\otimes M(Q)=0$, such a projector "kills" the motives of all anisotropic varieties.
These {\it local} motivic categories were discovered in an attempt to find an alternative approach to the functors of Bachmann
(mentioned above) and were briefly introduced in \cite[Sect. 4]{PicQ}.

Our next point is that to extract the local information in a meaningful form one should first pass to an appropriate field extension
of a ground field. This is illustrated by the example of an algebraically closed ground field $k$, where, up to equivalence, there is
only one (trivial) class of field extensions, represented by the extension $k/k$, and the respective localization functor
$\ffi:\dmkF{\ff_p}\row \dmELF{k}{k}{\ff_p}$ is an equivalence of categories - see Remark \ref{alg-clo-local}. Thus, it is conservative,
but not very interesting. At the same time, there is a large class of fields, for which the localization really simplifies things.
These are, the so-called, {\it flexible fields} introduced in Section \ref{flex-f}. Fields, which are purely transcendental extensions
of infinite (transcendence) degree of some other fields. Note, that one can always pass from an arbitrary field $k_0$ to a flexible one
$k_0(t_1,t_2,\ldots)$ without loosing any motivic information. The class of flexible fields is closed under finitely generated extensions.
Thus, if the ground field $k$ is flexible, then all the functors $\psi_E:\dmkF{\ff_p}\row \dmELF{E}{E}{\ff_p}$ take values in "flexible"
{\it isotropic categories}. And such categories are really simple. We examine them from two points of view: we look at the
{\it isotropic motivic cohomology} of a point $\hm^{*,*'}(k/k;\ff_p)$ and at the {\it isotropic Chow motivic category}
$Chow(k/k;\ff_p)$.

We show in Theorem \ref{D-A5} that, in the case of a flexible field, $\hm^{*,*'}(k/k;\ff_2)$ is the external algebra $\Lambda_{\ff_2}(r_{\{i\}}|_{i\geq 0})$
with the generators in one-to-one correspondence with Milnor's operations and the
action of the latter given by $Q_i(r_{\{i\}})=1$ and $Q_i(r_{\{j\}})=0$, for $i\neq j$. Thus, the answer is the same for
all {\it flexible fields}, and all these cohomology elements are "rigid", as we can get $1$ from any such non-zero element by applying
appropriate combination of Milnor's operations. The answer is also remarkable in the sense that Milnor's operations are encoded into
the structure of isotropic motivic category (in the form of their "inverses" $r_{\{i\}}$'s).
The computation is done with the help of Voevodsky technique used in the proof of Milnor's conjecture, and our answer explains to some
extent why Milnor's operations played such an important role in Voevodsky proof - see Theorem \ref{D-A4}.
Finally, the answer is drastically different
from the "global" one and the localization functor $\hm^{*,*'}(k,\ff_2)\row\hm^{*,*'}(k/k;\ff_2)$ is zero outside the bi-degree $(0)[0]$.
We are restricted to the prime $p=2$, since in our calculations the crucial role is played by
\cite[Cor. 3]{kerM}, and there is no analogue of this statement for $p>2$.

It appears that isotropic Chow motives are closely related to the numerical equivalence of cycles with $\ff_p$-coefficients.
We conjecture that, in the case of a flexible field, {\it isotropic Chow groups} $\Ch^*_{k/k}$ (describing $Hom$s between such motives)
coincide with the Chow groups (with $\ff_p$-coefficients) modulo numerical equivalence $\Ch^*_{Num}$ - see Conjecture \ref{main-conj}.
This would imply that the category of {\it isotropic Chow motives} is equivalent to the category of
{\it Chow motives modulo numerical equivalence} (with finite coefficients).
We prove this Conjecture for divisors, for varieties of dimension $\leq 5$, and for cycles of dimension $\leq 2$ -
Theorem \ref{thm-conj5-1-2}. In particular, this implies that {\it isotropic Chow groups} are finite-dimensional in these
situations. This also shows that the projection $\Ch^*\twoheadrightarrow\Ch^*_{k/k}$ factors through the 3-rd {\it theory of higher type}
$\Ch^*_{(3)}$ (where $\Ch^*_{(0)}=\Ch^*$ is the original theory of {\it rational type} and $\Ch^*_{(1)}=\Ch^*_{alg}$ is the
algebraic version) - see Definition \ref{def-THT} and Proposition \ref{A24}.
The proof of Theorem \ref{thm-conj5-1-2} constitutes the bulk of this article. It is by induction on the dimension of a variety $X$.
Using the moving technique introduced in Section \ref{A-r}, we show that after an appropriate blow-up, any class $u$ numerically equivalent
to zero may be represented by a cycle supported on a smooth connected divisor $Z$ and numerically trivial already on $Z$. An important step
here is to represent $u$ by the class of a smooth connected subvariety and numerically annihilate certain characteristic classes of it
- cf. Corollaries \ref{LS-nump-anis}, \ref{LS-r2-nump-div}.
Interestingly, the latter is achieved by a combination of appropriate blow-ups and Steenrod operations, depending on a prime involved.

The paper is organized as follows. After briefly discussing flexible fields in Section \ref{flex-f}, in Section \ref{loc-mot-cat}
we introduce the local motivic category with $\ff_p$-coefficients as well as its Chow-motivic version. In Section \ref{Sect-coh-point}
we study the isotropic motivic cohomology of a point, while Section \ref{IsoChowM} is devoted to the study of isotropic
Chow groups and the respective Chow-motivic category. In Section \ref{thick} we expand the definition of local motivic category
beyond prime coefficients. Finally, in Section \ref{A-r} we prove various geometric results used in Section \ref{IsoChowM}.

{\bf Acknowledgements:} I'm very grateful to J.-L.Colliot-Th\'el\`ene for useful discussions,
to C.Voisin for helpful comments and drawing my attention to \cite[Theorem 5]{SoVo}, and to J.I.Kylling, P.A.\O stv\ae r, O.R\"ondigs  and
S.Yagunov for fruitful interaction during my visit to the University of Oslo which led to the inclusion of Section \ref{thick}.
Also, I'm very grateful to T.Bachmann for the valuable remarks on the previous version of the paper.
Finally, I would like to thank the Referee for many useful suggestions which helped to improve the exposition.

\subsection{Notations and conventions}

Everywhere below $k$ will denote a field of characteristic zero.

\noindent
$\smk$ - the category of smooth quasi-projective varieties over $k$.

\noindent
$\Ch^*$ - the Chow groups $\CH^*/p$ with finite coefficients, where $p$ is some prime
(in Section \ref{A-r}, $p$ will be replaced by an arbitrary natural number $n$).

\noindent
$\dmkF{\ff_p}$ will denote the triangulated category of {\it Voevodsky motives} over a field - \cite{VoMo}, \cite{CiDe} and
$\dmgmkF{\ff_p}$ will denote the full triangulated subcategory of {\it geometric motives} in it.

\noindent
$\laz$ - the {\it Lazard ring}, that is, the coefficient ring of the universal formal group law.

\subsection{Flexible fields}
\label{flex-f}

Traditionally, algebraic geometry was considered over algebraically closed fields. Over such fields, every algebraic variety
(of finite type) has a rational point, which simplifies many things. In the case of a general field the standard approach is
to consider the passage to its algebraic closure. Note, however, that (torsion) information is lost under such a passage.
One of the aims of the current paper is to convince the reader that there are other directions one can pursue.
Namely, I propose to move instead in the direction of the, so-called, {\it flexible} fields. Such fields have the advantage that
one doesn't have to distinguish between the ground field and finitely-generated purely transcendental extensions of it. This helps
with many algebro-geometric constructions.

\begin{definition}
\label{def-flex}
Let us call a field $k$ {\it flexible} if it is a purely transcendental extension of countable infinite degree of some other field:
$k=k_0(t_1,t_2,\ldots)=k_0(\pp^{\infty})$.
\end{definition}

Note, that any finitely generated extension of a flexible field $k$ is itself flexible. Indeed, such an extension is defined by
finitely many generators and relations, which can "spoil" only finitely many of the original transcendental generators.
Thus, all the points of the large Nisnevich site of $\op{Spec}(k)$  are flexible.
On the other hand, we have:

\begin{remark}
{\rm
The natural restriction functor $\dmEF{k_0}{\ff_p}\row\dmkF{\ff_p}$ is conservative. So, we can substitute a field by a flexible one
without loosing any motivic information.
}
\Red
\end{remark}

The main property of flexible fields we will need is the following obvious observation.

\begin{proposition}
\label{flex-main-prop}
Let $k$ be a flexible field, $X$ - variety of finite type over $k$, and $E/k$ be a finitely-generated purely transcendental
field extension. Then there exists a commutative diagram
$$
\xymatrix{
X \ar[d] \ar[r]^{\cong} & X_E \ar[d] \\
\op{Spec}(k) \ar[r]^{\cong} & \op{Spec}(E)
}
$$
with horizontal maps isomorphisms (over some subfield $k_0$).
\end{proposition}

\begin{proof}
Let $k=k_0(t_1,t_2,\ldots)$. Then $X$ is defined over some finitely-generated purely transcendental extension $F$ of $k_0$ such that
$k/F$ is purely transcendental. That is, there is a variety $\ov{X}$ of finite type over $F$, such that $\ov{X}_k=X$.
Since extensions $k/F$ and $E/F$ are isomorphic, we get what we need.
\Qed
\end{proof}

\section{Motivic category of a field extension}
\label{loc-mot-cat}

Everywhere below $\dmkF{\ff_p}$ will denote the triangulated motivic category of Voevodsky over $\op{Spec}(k)$
with $\ff_p$-coefficients (see \cite{VoMo}, \cite{CiDe}).
We will construct the {\it local} versions of this category, corresponding to all finitely-generated field extensions $E/k$,
or, in other words, to all points of the big Nisnevich site over $\op{Spec}(k)$.
The {\it local motivic categories} will be obtained as full localizing subcategories of a global one by application
of certain projectors. These projectors will be produced using {\it \v{C}ech simplicial schemes}.

Let $P$ be a smooth variety over $k$. The \v{C}ech simplicial scheme $\check{C}ech(P)$ has graded components
$(\check{C}ech(P))_n=P^{\times (n+1)}$ with faces and degeneracy maps given by partial projections and partial diagonals.
This object is an analogue of the contractible space $EG$ in topology, and it will be contractible in Morel-Voevodsky homotopic
motivic category as long as $P$ has a rational point, while in general, it "measures" how far we are from acquiring such a point.
In particular, it is a {\it form} of a point, since it certainly contracts over algebraic closure.
Let us denote the motive of $\check{C}ech(P)$ as $\hii_P$.
The natural projection $\check{C}ech(P)\row\op{Spec}(k)$ provides the morphism
$\hii_P\row T$ to a trivial Tate-motive, which is an isomorphism if and only if $P$ has a zero cycle of degree 1 (modulo $p$, in our case)
(a "weak form" of a rational point) - \cite[Thm 2.3.4]{IMQ}.
This gives an exact triangle
$$
\xymatrix{
\hii_P \ar[r] & T \ar[r] & \whii_P \ar[r] & \hii_P[1],
}
$$
where $\hii_P$ and $\whii_P$ are mutually orthogonal idempotents:
$$
\hii_P\otimes\hii_P\stackrel{\cong}{\row}\hii_P; \hspace{1cm}
\whii_P\otimes\whii_P\stackrel{\cong}{\low}\whii_P; \hspace{1cm}
\hii_P\otimes\whii_P\cong 0.
$$
Thus, the functors of tensor product with these objects:
$$
\hii_P\otimes:\dmkF{\ff_p}\row\dmkF{\ff_p};\hspace{1cm} \whii_P\otimes:\dmkF{\ff_p}\row\dmkF{\ff_p}
$$
are projectors. This defines the semi-orthogonal decomposition of the category $\dmkF{\ff_p}$ as
an extension of $\whii_P\otimes\dmkF{\ff_p}$ by $\hii_P\otimes\dmkF{\ff_p}$, as there are no $\Hom$'s from the latter subcategory to
the former one (by \cite[Thm 2.3.2]{IMQ}, which is, basically, \cite[Lem. 4.9]{Vo-BKMK}).

For different varieties, these projectors naturally commute and we have canonical (co-associative, respectively, associative)
identifications
$$
\hii_P\otimes\hii_Q\stackrel{=}{\low}\hii_{P\times Q}\hspace{5mm}\text{and}\hspace{5mm}
\whii_P\otimes\whii_Q\stackrel{=}{\row}\whii_{P\coprod Q}
$$
(note, that endomorphism rings of $\hii_V$ and $\whii_V$ are either $\ff_p$, or zero - \cite[Thms 2.3.2, 2.3.3]{IMQ}, and such
an endomorphism is fixed by the map $\hii_V\row T$, respectively, $T\row\whii_V$).
Thus, tensor product of any (finite) number of such objects can be always expressed as $\hii_R\otimes\whii_S$, for some $R$ and $S$.

Each $\hii_P$ corresponds to a sub-sheaf $\chi_P$ of the constant sheaf ${\cal T}=\ff_p$ on the big Nisnevich site over $\op{Spec}(k)$
defined as follows. For a smooth connected quasi-projective variety $X$,
$\chi_P(X)=\ff_p$, if $P$ has a zero-cycle of degree $1$ over $k(X)$, and it is zero, otherwise.
Equivalently, $\chi_P(X)=\ff_p$ exactly when $\whii_P\otimes M(X)=0$ (\cite[Thms 2.3.6, 2.3.3]{IMQ}).
Respectively, $\whii_P$ corresponds to the quotient sheaf
$\wt{\chi}_P={\cal T}/\chi_P$. We can introduce an order on the set of $\hii_Q$'s as follows:

\begin{definition}
\label{D-Iso-hii-order}
We say that $\hii_Q\geq\hii_P$ if any of the following equivalent conditions is satisfied:
\begin{itemize}
\item[$(1)$] The natural map $\hii_Q\stackrel{\cong}{\low}\hii_Q\otimes\hii_P$ is an isomorphism;
\item[$(2)$] The natural map $\whii_Q\otimes\whii_P\stackrel{\cong}{\low}\whii_P$ is an isomorphism;
\item[$(3)$] The map $\hii_Q\row T$ factors through $\hii_P\row T$;
\item[$(4)$] $P$ has a zero-cycle of degree $1$ modulo $p$ over the generic point of every component of $Q$;
\item[$(5)$] $\chi_Q$ is a sub-sheaf of $\chi_P$.
\end{itemize}
\end{definition}
Here $(1)\Leftrightarrow(2)$ is automatic from the definition; $(2)\Rightarrow(5)$ follows from the description of
$\chi_P$ above; $(5)\Rightarrow(4)$ follows from the fact that $\chi_Q(Q_l)=\ff_p$, for any connected component $Q_l$ of $Q$;
$(4)\Rightarrow(1)$ is \cite[Thm 2.3.6]{IMQ}; $(1)\Rightarrow(3)$ is straightforward; and, finally, $(3)\Rightarrow(1)$
is clear, since $\hii_Q\otimes\whii_P\otimes(\hii_Q\row T)$ is the identity map of $\hii_Q\otimes\whii_P$.

Note, that the relation $\hii_Q\geq\hii_P$ is obviously transitive. In the case of connected varieties, this
relation may be formulated in terms of the respective field extensions (generic points).
Let $E/k$ and $F/k$ be two finitely-generated extensions of a field of characteristic zero.
Let $P/k$ and $Q/k$ be smooth projective varieties whose function fields are identified with $E$ and $F$.

\begin{definition}
\label{D-Iso-1}
We say that $F/k\geq E/k$, if there exists a correspondence of degree $1$ (with $\ff_p$-coefficients)
$Q\rightsquigarrow P$. We call extensions "equivalent" $F/k\sim E/k$, if $F/k\geq E/k$ and $E/k\geq F/k$.
\end{definition}
By the composition of correspondences, the property $F/k\geq E/k$ is transitive.
It is also equivalent to the condition $\hii_Q\geq\hii_P$ above.
If $k/l$ is a field extension, then $F/k\geq E/k$ implies that $F/l\geq E/l$.

Let $P$ be some smooth variety (of finite type) over $k$.
Let ${\mathbf{Q}}$ be the disjoint union of all connected varieties $Q/k$, such that
$\hii_Q\gneqq\hii_P$ (so, it is a smooth variety, but with infinitely many components), and let $\hii_{{\mathbf{Q}}}$ be the motive
of the respective \v{C}ech simplicial scheme, which is still an idempotent in $\dmkF{\ff_p}$, and $\whii_{{\mathbf{Q}}}$ be the
complementary idempotent. Define
$$
\Upsilon_P:=\whii_{{\mathbf{Q}}}\otimes\hii_{P}.
$$
We can view $\Upsilon_P$ as a colimit of projectors $\whii_{Q}\otimes\hii_P$, where $Q$ runs over all smooth projective varieties
of finite type with $\hii_{Q}\gneqq\hii_{P}$.

Note, that if $P$ is not {\it connected up to equivalence}, that is, if $P$ can't be substituted by a connected variety with the
same $\hii$, then $\Upsilon_P=0$. Indeed, let $P_1$ be a "minimal" component, that
is, $\hii_{P_1}\geq\hii_{P_i}$ implies that $\hii_{P_1}=\hii_{P_i}$. Suppose, there exists another component $P_2$ with
$\hii_{P_2}\not\geq\hii_{P_1}$. Let $\hat{P}_1$ be the union of all the components equivalent to $P_1$. Then for
$Q_1=P\backslash \hat{P}_1$ and $Q_2=\hat{P}_1$ we have:
$\hii_{Q_1}\gneqq\hii_{P}$, $\hii_{Q_2}\gneqq\hii_{P}$, but $\hii_{Q_1\coprod Q_2}=\hii_{P}$.

Now we can define the {\it local motivic category} corresponding to a finitely-generated extension $E/k$ (cf. \cite[Sect. 4]{PicQ}). 

\begin{definition}
\label{D-Iso-2}
Let $E/k$ be a finitely generated extension and
$P/k$ be a smooth connected variety with $k(P)=E$.
Define the "motivic category of the extension $E/k$" as the full localizing subcategory
$$
\dmELF{E}{k}{\ff_p}=\Upsilon_P\otimes\dmkF{\ff_p}
$$
of $\dmkF{\ff_p}$, and the "local geometric category" $\dmgmELF{E}{k}{\ff_p}$ as the full thick triangulated subcategory
of $\dmELF{E}{k}{\ff_p}$ generated by (local) motives of smooth projective varieties.
\end{definition}
This definition doesn't depend on the choice of a smooth model $P$, since $\hii_P$ depends on $k(P)$ only.
Moreover, it depends only on the $\sim$-equivalence class of an extension $E/k$.

In the case of a trivial extension, we obtain (cf. \cite[Sect. 4]{PicQ}):
\begin{definition}
\label{D-Iso-3}
The "isotropic motivic category" is the full localizing subcategory $\dmELF{k}{k}{\ff_p}$ of $\dmkF{\ff_p}$ given by
$\Upsilon_{\op{Spec}(k)}\otimes\dmkF{\ff_p}$,
while the geometric version $\dmgmELF{k}{k}{\ff_p}$ is the full thick triangulated subcategory of it generated by (isotropic) motives
of smooth projective varieties.
\end{definition}

Now, we can read the information about motive by looking at local versions of it.
Namely, we get a collection $\{\ffi_E|\,E/k-\text{f.g. extension}\}$ of localization functors
$$
\ffi_E:\dmkF{\ff_p}\lrow\dmELF{E}{k}{\ff_p}
$$
parameterized by all points of the big Nisnevich site over $\op{Spec}(k)$. These can be further specialized
to {\it isotropic realizations}
$$
\psi_E:\dmkF{\ff_p}\lrow\dmELF{E}{E}{\ff_p}.
$$

The following result shows that there are no unexpected objects
in the {\it isotropic geometric} category.

\begin{proposition}
\label{geom-iso-vse}
The category $\dmgmELF{k}{k}{\ff_p}$ is the idempotent completion of the full subcategory $\Upsilon_{\op{Spec}(k)}\otimes\dmgmkF{\ff_p}$ of $\dmkF{\ff_p}$.
\end{proposition}

\begin{proof}
We need to prove that $\Upsilon_{\op{Spec}(k)}\otimes\dmgmkF{\ff_p}$ is closed under cones. For this, it is sufficient to show
that, for any objects $U,V$ of $\dmgmkF{\ff_p}$ and a map $\wt{f}:\whii_{{\mathbf{Q}}}\otimes U\row\whii_{{\mathbf{Q}}}\otimes V$
(where ${\mathbf{Q}}$ is the disjoint union of all anisotropic varieties over $k$), there is a map $f:U'\row V'$ in $\dmgmkF{\ff_p}$,
such that $f\otimes id_{\whii_{{\mathbf{Q}}}}\cong\wt{f}$. Composing the map $\wt{f}$ with $U\row\whii_{{\mathbf{Q}}}\otimes U$ we obtain
a map $g:U\row\whii_{{\mathbf{Q}}}\otimes V$ with the property that $g\otimes id_{\whii_{{\mathbf{Q}}}}\cong\wt{f}$.
Define $(\whii_{{\mathbf{Q}}})_{\leq n}$ as $\op{Cone}((\hii_{{\mathbf{Q}}})_{\leq n-1}\row T)$ and $(\whii_{{\mathbf{Q}}})_{>n}$
as $\op{Cone}((\whii_{{\mathbf{Q}}})_{\leq n}\row\whii_{{\mathbf{Q}}})$. Then $(\whii_{{\mathbf{Q}}})_{>n}$ is an extension of $M(Y)[r]$,
for some smooth varieties $Y$ and $r>n$. Since $U$ and $V$ are geometric, for sufficiently large $n$, there are no $\Hom$'s from $U$
to $(\whii_{{\mathbf{Q}}})_{>n}\otimes V$. Hence, the map $g$ can be lifted to a map
$f':U\row(\whii_{{\mathbf{Q}}})_{\leq n}\otimes V$, which, in turn, can be lifted to a geometric map
$f:U\row(\whii_{Q})_{\leq n}\otimes V$, for some anisotropic variety $Q$ of finite type over $k$.
Since $\whii_{{\mathbf{Q}}}\otimes(\whii_{Q})_{\leq n}=\whii_{{\mathbf{Q}}}$, for any $n\geq 0$,
we obtain that $f\otimes id_{\whii_{{\mathbf{Q}}}}\cong\wt{f}$.
\Qed
\end{proof}

We can describe $\Hom$'s from geometric isotropic motives as follows.
For an object $X$ of $\dmkF{\ff_p}$ and some idempotent $\xi$, we will denote by the same letter the image of $X$ in
$\xi\otimes\dmkF{\ff_p}$.

\begin{proposition}
\label{hom-iso-geom}
Let $U\in Ob(\dmgmkF{\ff_p})$ and $V\in Ob(\dmkF{\ff_p})$. Then
$$
\Hom_{\dmELF{k}{k}{\ff_p}}(U,V)=\operatornamewithlimits{colim}_{\hii_Q\neq T}\Hom_{\whii_Q\otimes\dmkF{\ff_p}}(U,V),
$$
where the colimit is taken over all the functors
$\whii_{S}\otimes:\whii_R\otimes\dmkF{\ff_p}\row\whii_S\otimes\dmkF{\ff_p}$,
for $\hii_R\geq\hii_S\neq T$. In other words, $Q$ runs over all anisotropic varieties over $k$.
\end{proposition}

\begin{proof}
We have: $\Hom_{\dmELF{k}{k}{\ff_p}}(U,V)=\Hom_{\dmkF{\ff_p}}(\Upsilon_{\op{Spec(k)}}\otimes U,\Upsilon_{\op{Spec}(k)}\otimes V)$,
and the latter can be identified with $\Hom_{\dmkF{\ff_p}}(U,\whii_{{\mathbf{Q}}}\otimes V)$, where
${\mathbf{Q}}$ is the disjoint union of all anisotropic varieties over $k$. But since $U$ is geometric, any map
$U\row\whii_{{\mathbf{Q}}}\otimes V$ factors through $U\row\whii_Q\otimes V$, for some anisotropic $Q$ of finite type,
and the map $U\row\whii_Q\otimes V$ vanishes when extended to a map to $\whii_{{\mathbf{Q}}}\otimes V$ if and only if
there exists an anisotropic $Q'$ with $\hii_{Q}\geq\hii_{Q'}$, such that the composition $U\row\whii_Q\otimes V\row\whii_{Q'}\otimes V$
is zero. Thus, our $\Hom$-group can be identified with the
$$
\operatornamewithlimits{colim}_{\hii_Q\neq T}\Hom_{\whii_Q\otimes\dmkF{\ff_p}}(U,V),
$$
where the colimit is taken over a directed system (as $\hii_{Q_i}\geq\hii_{\coprod_iQ_i}$ and if $\hii_{Q_i}\neq T$,
for each $i$, then $\hii_{\coprod_iQ_i}\neq T$, since the coefficients are $\ff_p$).
\Qed
\end{proof}

Geometric motives vanishing in the local category can be also detected by projectors corresponding to varieties of finite type.
Let $P$ be smooth connected variety with $E=k(P)$.

\begin{proposition}
\label{vanish-loc-fin-level}
An object $U$ of $\dmgmkF{\ff_p}$ vanishes in $\dmELF{E}{k}{\ff_p}$ if and only if there is a variety $Q$ of finite type over $k$,
with $\hii_Q\gneqq\hii_P$ and $\whii_Q\otimes\hii_P\otimes U=0$.
\end{proposition}

\begin{proof}
If $\Upsilon_{P}\otimes U=0$, then $(\whii_{{\mathbf{Q}}}\otimes U)_E=0$, where, as above, ${\mathbf{Q}}$ is the disjoint union of all
smooth connected varieties $Q$ over $k$, with $\hii_Q\gneqq\hii_P$. That means that the projection $(\hii_{{\mathbf{Q}}}\otimes U\row U)_E$
has a section (from the right). But since $U$ is geometric, such a section will factor through some section of
$(\hii_{Q}\otimes U\row U)_E$ for some variety $Q$ of finite type over $k$ with $\hii_Q\gneqq\hii_P$. Hence,
$(\whii_Q\otimes U)_E=0$ (as $U_E$ is a direct summand of $\hii_Q\otimes U_E$, and so, $\whii_Q\otimes U_E$ is a direct summand
of it as well, but the latter object is stable under $\whii_Q\otimes$ while the former one is killed by it).
But, according to \cite[Thm. 2.3.5]{IMQ}, the functor
$\hii_P\otimes\dmkF{\ff_p}\row\dmEF{E}{\ff_p}$ is conservative. Hence, $\whii_Q\otimes\hii_{P}\otimes U=0$ in $\dmkF{\ff_p}$.
\Qed
\end{proof}

Since $\whii_Q\otimes M(Q)=0$, the projection to the {\it isotropic motivic category} $\dmELF{k}{k}{\ff_p}$ kills the motives of all
{\it anisotropic} varieties over $k$. Hence, the name of this category.

\begin{remark}
\label{local-Verdier}
{\rm
The {\it isotropic motivic category} $\dmELF{k}{k}{\ff_p}$ is the Verdier localization of $\dmkF{\ff_p}$ modulo the localizing
subcategory ${\cal A}$ generated by motives of {\it anisotropic} varieties\footnote{I'm grateful to T.Bachmann for pointing this out.}.
Indeed, an object $U$ of $\dmkF{\ff_p}$ vanishes in $\dmELF{k}{k}{\ff_p}$ if and only if $U=U\otimes\hii_{{\mathbf{Q}}}$, where
${\mathbf{Q}}$ is the disjoint union of all connected anisotropic varieties $Q/k$. Hence, $U$ belongs to ${\cal A}$, since this
subcategory is a tensor ideal. By the universal property of the Verdier localization, $\psi_E:\dmkF{\ff_p}\row\dmELF{k}{k}{\ff_p}$
is equivalent to $\dmkF{\ff_p}\row\dmkF{\ff_p}/{\cal A}$.
}
\Red
\end{remark}

We have functoriality for the "denominator" of the extension $E/k$.
Suppose, we have a tower of fields $L\subset F\subset E$, and $P/L$, $Q/L$ are smooth projective varieties with $L(P)=E$ and
$\hii_Q\gneqq\hii_P$. Then $Q$ remains anisotropic over $L(P)$, and so, over $F(P)$ (since $F\subset L(P)$).
Hence, $\hii_Q|_F\gneqq\hii_P|_F$.
Thus, we get a natural functor
$$
\dmELF{E}{L}{\ff_p}\lrow\dmELF{E}{F}{\ff_p}.
$$

The following result shows that, in the case of a flexible ground field, we can pass from (geometric) {\it local realizations}
$\{\ffi_E|\,E/k-\text{f.g. extension}\}$ to {\it isotropic realizations} $\{\psi_E|\,E/k-\text{f.g. extension}\}$
without loosing any information.

\begin{proposition}
\label{D-A9}
Let $E/k$ be a finitely-generated extension of a flexible field. Then the functor
$$
\dmgmELF{E}{k}{\ff_p}\lrow\dmgmELF{E}{E}{\ff_p}
$$
is conservative on the image of $\ffi_E$.
\end{proposition}

\begin{proof}
Let us start with purely transcendental extensions.

\begin{lem}
\label{D-A7}
Let $E/L/k$ be a tower of finitely generated extensions where $L/k$ is a purely transcendental extension of a flexible field.
Then the functor
$$
\dmgmELF{E}{k}{\ff_p}\lrow\dmgmELF{E}{L}{\ff_p}
$$
is conservative on the image of $\ffi_E$.
\end{lem}

\begin{proof}
Let $L=k(\aaa^n)$, and $E=k(\ov{R})$ for some smooth variety $\ov{R}/k$.
Let $\ov{U}\in Ob(\hii_{\ov{R}}\otimes\dmgmkF{\ff_p})$ be an object vanishing in $\dmgmELF{E}{L}{\ff_p}$. Then,
according to the Proposition \ref{vanish-loc-fin-level}, there exists a
variety $\ov{Q}/L$ of finite type such that $\hii_{\ov{Q}}\gneqq\hii_{\ov{R}_L}$ and $\ov{U}_L\otimes\whii_{\ov{Q}}=0$ in $\dmEF{L}{\ff_p}$.
The condition $\hii_{\ov{Q}}\gneqq\hii_{\ov{R}_L}$ means that we have an $L$-correspondence
$\ov{\alpha}:\ov{Q}\rightsquigarrow \ov{R}_L$ of degree one, and there is no such correspondence in the opposite direction.
Since $k=k_0(\pp^{\infty})$ is flexible,
varieties $\ov{R}$ and $\ov{Q}$ are, actually, defined over $F$ and $M=F(\aaa^n)$, respectively, where extensions $k/F/k_0$ are
purely transcendental and $F/k_0$ is moreover finitely generated. By the same reason, we can assume that
the geometric object $\ov{U}$ is defined over $F$, while
the correspondence $\ov{\alpha}$ is defined over $M$.
So, there exist varieties $R/F$, $Q/M$, an object $U$ of $\hii_{R}\otimes\dmgmEF{F}{\ff_p}$ and a degree one $M$-correspondence
$\alpha:Q\rightsquigarrow R_M$ such that $R|_k=\ov{R}$, $Q|_L=\ov{Q}$, $U|_k=\ov{U}$ and $\alpha|_L=\ov{\alpha}$.
Note, that we still have: $\hii_Q\gneqq\hii_{R_M}$ (since $\alpha$ is defined over $M$ and by functoriality),
and $U_M\otimes\whii_Q=0$ (since the restriction $\dmEF{M}{\ff_p}\row\dmEF{L}{\ff_p}$ is conservative).
But the extension
$M/F$ can be embedded into $k/F$ making $k/M$ purely transcendental. Let $Q'$ be a variety over $k$ obtained from
$Q$ using this embedding. Then $\hii_{Q'}\gneqq\hii_{\ov{R}}$ (since $k/M$ is purely transcendental) and
$\ov{U}\otimes\whii_{Q'}=0$ in $\dmEF{k}{\ff_p}$. Hence,
$\ov{U}=0$ in $\dmELF{E}{k}{\ff_p}$.
\Qed
\end{proof}

Using Lemma \ref{D-A7} our problem is reduced to
the case of a finite extension. In this situation, the statement is true for an arbitrary field.

\begin{lem}
\label{D-A8}
Let $E/L$ be a finite extension of fields. Then the functor
$$
\dmELF{E}{L}{\ff_p}\lrow\dmELF{E}{E}{\ff_p}
$$
is conservative.
\end{lem}

\begin{proof}
Let $E=L(P)$ for some smooth connected 0-dimensional variety $P$.
Let $U\in Ob(\hii_P\otimes\dmEF{L}{\ff_p})$ be an object
vanishing in $\dmELF{E}{E}{\ff_p}$. Then, for the  disjoint union $\ov{{\mathbf{Q}}}$ of all anisotropic varieties over $E$,
we have: $U_E\otimes\whii_{\ov{{\mathbf{Q}}}}=0$ in $\dmEF{E}{\ff_p}$. Consider a smooth $L$-variety $\widehat{{\mathbf{Q}}}$
given by the composition
$\ov{{\mathbf{Q}}}\row\op{Spec}(E)\row\op{Spec}(L)$. We have a natural map $\ov{{\mathbf{Q}}}\row\widehat{{\mathbf{Q}}}_E$.
This means that $\hii_{\ov{{\mathbf{Q}}}}\geq\hii_{\widehat{{\mathbf{Q}}}_E}$,
and so, $U_E\otimes\whii_{\widehat{{\mathbf{Q}}}_E}=0$ as well. Clearly, $\hii_{\widehat{{\mathbf{Q}}}}\geq\hii_P$.
Suppose, these are equal. Then there exists a commutative
diagram
$$
\xymatrix{
\ov{{\mathbf{Q}}} \ar[r] & \op{Spec}(E) \ar[r] & \op{Spec}(L) \\
\op{Spec}(F) \ar[u] \ar[r] & \op{Spec}(E) \ar[r] & \op{Spec}(L) \ar@{=}[u]
}
$$
where $F$ is an extension of $E$ of degree prime to $p$. But since $[E:L]$ is finite, the composition
$\op{Spec}(F)\row\ov{{\mathbf{Q}}}\row\op{Spec}(E)$ has the same (prime to $p$) degree.
This contradicts to the fact that $\ov{{\mathbf{Q}}}$ is anisotropic.
Hence, $\hii_{\widehat{{\mathbf{Q}}}}\gneqq\hii_P$, and so,
$\hii_{\widehat{{\mathbf{Q}}}}\geq\hii_{{\mathbf{Q}}}$, where ${\mathbf{Q}}$ is a disjoint union of all $L$-varieties $Q$
with $\hii_Q\gneqq\hii_P$. Thus, $(U\otimes\whii_{{\mathbf{Q}}})_E=0$ as well.
By \cite[Thm 2.3.5]{IMQ}, the functor
$\hii_P\otimes\dmEF{L}{\ff_p}\row\dmEF{E}{\ff_p}$ is conservative,
so, $(U\otimes\whii_{{\mathbf{Q}}})\otimes\hii_P=0$ in $\dmEF{L}{\ff_p}$. This means that $U=0$ in $\dmELF{E}{L}{\ff_p}$.
\Qed
\end{proof}

This finishes the proof of the Proposition \ref{D-A9}.
\Qed
\end{proof}

Another type of functoriality we have is the following one. Let $k(A)/k$ be a purely transcendental extension of $k$. Then we have a
natural functor
$$
\dmELF{E}{k}{\ff_p}\row\dmELF{E(A)}{k(A)}{\ff_p}.
$$
One just needs to observe that the inequality $\hii_Q\gneqq\hii_P$ is preserved under the passage from $k$ to $k(A)$. \\

\medskip

It is natural to ask, in which situations will our localization functors be conservative?
\begin{question}
\label{conserv-main-q}
\phantom{a}\hspace{5mm}\phantom{a}
\begin{itemize}
\item[$(a)$] What is the kernel of the collection of functors $\{\ffi_E|\,E/k-\text{f.g. extension}\}$?
\item[$(b)$] What is the kernel of the collection of functors $\{\psi_E|\,E/k-\text{f.g. extension}\}$?
\end{itemize}
\end{question}

Since the passage from $k_0$ to $k=k_0(t_1,t_2,\ldots)$ is conservative, and any finitely generated extension $E$ of $k$ has the form
$E=E_0(t_N,\ldots)$, for some finitely-generated extension $E_0$ of $k_0$, and by Proposition \ref{D-A9},
the triviality of 
$\{\ffi_{E_0}|\,E_0/k_0-\text{f.g. extension}\}$ on $X_0$ 
implies the triviality of 
$\{\psi_{E_0}|\,E_0/k_0-\text{f.g. extension}\}$ on $X_0$,
implying the triviality of
$\{\psi_E|\,E/k-\text{f.g. extension}\}$ on $X_0|_k$,
which, in turn, is equivalent to the triviality of
$\{\ffi_E|\,E/k-\text{f.g. extension}\}$ on $X_0|_k$.
Thus, for a given geometric object $X_0/k_0$,
$$
(a)\Rightarrow (b)\Rightarrow (b)_{flex}\Leftrightarrow (a)_{flex},
$$
where $(a)$  means that $X_0$ is in the kernel of the family $\{\ffi_{E_0}|\,E_0/k_0-\text{f.g. extension}\}$, 
$(a)_{flex}$ means that $X_0|_k$ (the restriction to the flexible closure) is in the kernel of the family 
$\{\ffi_E|\,E/k-\text{f.g. extension}\}$, etc..

\begin{remark}
\label{Bachmann-functors}
{\rm Restricting the functors $\psi_E$ to the tensor triangulated subcategory $\op{DQM}^{gm}$ generated by motives of smooth projective quadrics,
and specializing it further, one gets the functors of T.Bachmann $\Phi^E:\op{DQM}^{gm}\row K^b(Tate(\ff_2))$ to the category of
bi-graded $\ff_2$-vector spaces - see \cite{BQ}.
This can be deduced from the fact that the functor $\psi_E$ maps the subcategory $\op{DQM}^{gm}$ to the subcategory of
{\it geometric Tate-motives} in $\dmELF{E}{E}{\ff_p}$ (by \cite[Prop. 4.9]{PicQ}).
The functors $\Phi^E$, constructed originally by completely different methods,
were shown by T.Bachmann to be conservative \cite[Thm. 31]{BQ}. In particular, the collection
$\{\psi_E|\,E/k-\text{f.g. extension}\}$ is conservative on $\op{DQM}^{gm}$. Our approach also permits to see conservativity on this and
other similar subcategories. Namely, it follows from \cite[Prop. 4.9]{PicQ} that, for any object $A$ of $\op{DQM}^{gm}$, there exists
a finite filtration by idempotents on the unit object $T$, such that associated graded idempotents map $A$ to geometric Tate-motives.
And this collection of associated graded idempotents (having a form $\whii_Q\otimes\hii_P$, for some smooth varieties $P$ and $Q$, with
$P$-connected) acts conservatively (as the unit object is an extension of them). It remains to
observe that, for geometric Tate-motives, the triviality of $\whii_Q\otimes\hii_P\otimes A$ is equivalent to the triviality
of $\Upsilon_P\otimes A\in Ob(\dmELF{E}{E}{\ff_p})$, for $E=k(P)$.
}
\Red
\end{remark}

\begin{remark}
\label{alg-clo-local}
{\rm If the ground field $k_0$ is algebraically closed, then there exists only one $\sim$-equivalence class of
finitely-generated extensions of $k_0$ (the trivial one). Thus, there is only one "local" point and only one localization functor
$\ffi: \dmEF{k_0}{\ff_p}\lrow\dmELF{k_0}{k_0}{\ff_p}$ which is an equivalence of categories
(as there are no anisotropic varieties over $k_0$).
Thus, in this case, the family $\{\ffi_{E_0}|\,E_0/k_0-\text{f.g. extension}\}$ is conservative, but it does not provide
any interesting information.
}
\Red
\end{remark}

The collection $\{\ffi_E|\,E/k-\text{f.g. extension}\}$ is not conservative, in general.

\begin{example}
\label{not-conserv-1}
{\rm
(1) Let $k$ be a flexible field and $C$ be an elliptic curve over $k$ without complex multiplication. Consider $p=2$.
Then $M(C)=T\oplus\widehat{M}(C)\oplus T(1)[2]$. Consider the Chow groups $\Ch_{Num(p)}$ modulo numerical equivalence with
$\ff_2$-coefficients - see Subsection \ref{nump-iso}. Then
$$
\Ch^1_{Num(p)}(C\times C)=[pt\times C]\cdot\ff_2\oplus[C\times pt]\cdot\ff_2.
$$
Indeed, for an arbitrary $p$, such a group is generated by $[pt\times C]$, $[C\times pt]$ and the class of the diagonal $[\Delta]$
(in the absence of complex multiplication).
But with $\ff_2$-coefficients, $[\Delta]\nump [pt\times C]+[C\times pt]$.
Thus, $\widehat{M}(C)=0$ in $Chow_{Num}(E;\ff_2)$, for any extension $E/k$. Hence, by Theorem \ref{thm-conj5-1-2}(1),
it is zero in $Chow(E/E;\ff_2)$ which is a subcategory of $\dmELF{E}{E}{\ff_2}$. So, all isotropic realizations $\psi_E(\widehat{M}(C))$
are trivial.
At the same time, $\widehat{M}(C)$ is non-trivial, since the (complex) topological realisation of it is non-trivial (has non-zero $H^1_{Top}$).
Alternatively, one can see that the restriction to the algebraic closure $\widehat{M}(C)|_{\kbar}$ is non-trivial.
Note, that the choice of a prime was essential here.

(2) Refining the previous example, we can show that even the combination
$$
\{\ffi_E|\,E/k-\text{f.g. extension}\}\cup \op{res}_{\kbar}
$$
is not conservative on $\dmgmkF{\ff_2}$.
In the above situation, consider some non-trivial quadratic extension $F=k(\sqrt{a})$ and $P=\op{Spec}(F)\stackrel{\pi}{\row}\op{Spec}(k)$.
Let $\alpha=\{a\}\in K^M_1(k)/2$ and
$\wt{M}_{\alpha}$ be the "completely" reduced Rost-motive - see the proof of Theorem \ref{D-A4}. This motive fits into the
exact triangle $\wt{M}_{\alpha}\row\whii_{\alpha}[1]\row\whii_{\alpha}\row\wt{M}_{\alpha}[1]$, where
$\whii_{\alpha}=\whii_{P}$.
Let us show that $U=\widehat{M}(C)\otimes\whii_{\alpha}\neq 0$. Indeed,
such a triviality is equivalent to the fact that the projection $\hii_P\times\widehat{M}(C)\row\widehat{M}(C)$ has a section.
And since $\widehat{M}(C)$ is a {\it pure motive} (= {\it Chow motive}), any such section is liftable to a section of
$P\times\widehat{M}(C)\row\widehat{M}(C)$.
This would mean
that the projector $\rho$ defining $\widehat{M}(C)$ is in the image of the natural map $\pi_*:\Ch^1(C\times C\times P)\row\Ch^1(C\times C)$.
Note however, that $\Ch^1(C\times C|_{\kbar})=[\Delta]\cdot\ff_2\oplus [pt\times C]\cdot\ff_2\oplus [C\times pt]\cdot\ff_2$ (since
$\kbar$-points of the Jacobian form a $2$-divisible group), and so the map
$\op{res}_{\kbar}\circ\pi_*:\Ch^1(C\times C\times P)\row\Ch^1(C\times C|_{\kbar})$ is zero (since the action of the Galois group on
$\Ch^1(C\times C|_{\kbar})$ is trivial, which implies that $\op{res}_{\kbar}\circ\pi_*=2\cdot\op{res_{\kbar/F}}$).
On the other hand, $\rho|_{\kbar}\neq 0$, since it is non-zero even in the topological realization. Hence, $\rho$ is not in the
image of $\op{res}_{\kbar}\circ\pi_*$ and $\widehat{M}(C)\otimes\whii_{\alpha}\neq 0$. Notice, that $\psi_E(U)=0$, since
$\psi_E(\widehat{M}(C))=0$, while $\op{res}_{\kbar}(U)=0$, since $\op{res}_{\kbar}(\whii_{\alpha})=0$. Thus, we have produced a non-trivial
example on which the needed combination of functors vanishes, but so far, not a geometric one.

Consider $V=\widehat{M}(C)\otimes\wt{M}_{\alpha}$. Then we have a distinguished triangle
$V\row U[1]\row U\row V[1]$. In particular, $V$ is geometric and all the above functors vanish on it. It remains to show
that $V\neq 0$. Note that, since $U\neq 0$, the homology $\Hom_{\dmEF{E}{\ff_2}}(T(*)[*'],U)$ considered for all finitely generated
extensions $E/k$, is non-trivial. At the same time, this homology is zero for $*'<*$ (below the main diagonal).
This implies that $V\neq 0$. Indeed, if it would be zero, then the homology of $U$ would be $[1]$-periodic, which is not the case.
}
\Red
\end{example}

\subsection{Local Chow motivic category}

Let $X$ be a scheme of finite type over $k$. We can define its {\it isotropic Chow groups} as
$$
\Ch_{k/k;r}(X):=\Hom_{\dmELF{k}{k}{\ff_p}}(T(r)[2r],M^c(X)),
$$
where $M^c(X)$ is the {\it motive with compact support} of $X$ - see \cite{VoMo}.
For smooth varieties, we have from duality:
$$
\Ch_{k/k}^r(X)=\Hom_{\dmELF{k}{k}{\ff_p}}(M(X),T(r)[2r]).
$$
The theory $\Ch_{k/k}$ has natural pull-backs and push-forward maps coming from the respective maps between motives of varieties,
which satisfy all the axioms of \cite[Def. 1.1.2]{LM} (since these follow from the properties of motives). Finally, we have
the excision axiom $(EXCI)$, claiming that for a scheme $X$ with the closed subscheme $Z$ and open complement $U$, there is an
exact sequence:
$$
\Ch_{k/k;*}(Z)\stackrel{i_*}{\lrow}\Ch_{k/k;*}(X)\stackrel{j^*}{\lrow}\Ch_{k/k;*}(U)\row 0.
$$
This follows from the Gysin exact triangle \cite[(4.1.5)]{VoMo}:
$$
M^c(Z)\row M^c(X)\row M^c(U)\row M^c(Z)[1]
$$
and the fact that the map $\Ch^*\twoheadrightarrow\Ch^*_{k/k}$ is surjective, which follows from Proposition \ref{D-ChowGroups}
below. Thus, $\Ch^*_{k/k}$ is an {\it oriented cohomology theory} (with excision) on $\smk$ in the sense of \cite[Def. 2.1]{SU}.

\begin{definition}
\label{anis-general}
Let $Q$ be a scheme (of finite type) over $k$ and $n\in\nn$. We say that $Q$ is "$n$-anisotropic", if the degrees of
all closed points of $Q$ are divisible by $n$.
\end{definition}

Schemes which don't have this property will be called {\it not $n$-anisotropic}, while we will reserve the term {\it isotropic} for
a scheme having a zero-cycle of degree $1$ (mod $n$). Below almost everywhere we will be dealing with $n=p$-prime, and so {\it isotropic}
will be the complement to {\it anisotropic}. Unless specified, the term {\it anisotropic} will mean $p$-{\it anisotropic}, for some prime
$p$.

\begin{definition}
\label{anis-class}
Let $X$ be a scheme over $k$, and $x\in\Ch_r(X)$. We call $x$ "anisotropic", if there exists a proper morphism
$f:Y\row X$ from a $p$-anisotropic scheme $Y$ and a class $y\in\Ch_r(Y)$ such that $x=f_*(y)$.
\end{definition}

For fields of characteristic zero and $X$ projective, $x$ is {\it anisotropic} if and only if it is a push-forward of the generic cycle
from some smooth projective anisotropic variety over $k$.

{\it Isotropic Chow groups} can be alternatively described as follows.

\begin{proposition}
\label{D-ChowGroups}
$$
\Ch_{k/k}(X)=\Ch(X)/(\text{anisotropic classes}).
$$
\end{proposition}

\begin{proof}
By Proposition \ref{hom-iso-geom}, $\Hom_{\dmELF{k}{k}{\ff_p}}(T(r)[2r],M^c(X))$ is the colimit of the groups
$$
\Hom_{\dmkF{\ff_p}}(\whii_Q(r)[2r],M^c(X)\otimes\whii_Q),
$$
where $Q$ runs over all anisotropic varieties over $k$.
Recall, that we have an exact triangle
$$
\xymatrix{
\hii_Q \ar[r] & T \ar[r] & \whii_Q \ar[r] & \hii_Q[1].
}
$$
By \cite[Thm 2.3.2]{IMQ} (which is, basically, \cite[Lem 4.9]{Vo-BKMK}),
$\Hom_{\dmkF{\ff_p}}(\hii_Q(*)[*'],M^c(X)\otimes\whii_Q)=0$, and so, our group is the colimit of groups
$\Hom_{\dmkF{\ff_p}}(T(r)[2r],M^c(X)\otimes\whii_Q)$, where $Q$ can be assumed to be projective.
Since $\whii_Q$ is an extension of $M(Q^{\times i})[i]$, for $i\geq 0$, and $\Hom_{\dmkF{\ff_p}}(T(r)[2r],M^c(Y)[i])=0$, for any $i>0$
and any scheme $Y$ of finite type, we can identify:
$$
\Hom_{\dmkF{\ff_p}}(T(r)[2r],M^c(X)\otimes\whii_Q)=\op{Coker}(\Ch_r(X\times Q)\stackrel{{\pi_Q}_*}{\lrow}\Ch_r(X)).
$$
Thus,
$\Hom_{\dmELF{k}{k}{\ff_p}}(T(r)[2r],M^c(X))=\Ch_r(X)/I$, where $I$ is the subgroup generated by the images of $(\pi_Q)_*$, for all
anisotropic varieties $Q/k$. In other words, we mod-out all {\it anisotropic classes}.
\Qed
\end{proof}

The {\it isotropic motivic category} $\dmgmELF{k}{k}{\ff_p}$ has a pure part.
\begin{definition}
\label{isot-Chow-mot}
Define the "isotropic Chow motivic category" $Chow(k/k;\ff_p)$ as the full additive subcategory of $\dmgmELF{k}{k}{\ff_p}$
- the image of $Chow(k;\ff_p)$ under the natural projection
$$
\dmgmkF{\ff_p}\row\dmgmELF{k}{k}{\ff_p}.
$$
\end{definition}

Thus, the objects of $Chow(k/k,\ff_p)$
can be identified with direct summands of motives of smooth projective varieties over $k$, while the morphisms are described as follows.

\begin{proposition}
\label{D-ChowHom}
Let $X$ and $Y$ be smooth projective $k$-varieties. Then
$$
\Hom_{Chow(k/k,\ff_p)}(M(X),M(Y))=\Ch_{k/k}^{\ddim(Y)}(X\times Y).
$$
\end{proposition}

\begin{proof}
If $B$ is an object of $\dmgmkF{\ff_p}$ with the dual $B^{\vee}$, and $A,C$ are objects of $\dmkF{\ff_p}$, then we have a functorial
identification
$$
\Hom_{\dmkF{\ff_p}}(A\otimes B,C)=\Hom_{\dmkF{\ff_p}}(A,B^{\vee}\otimes C).
$$
Hence, for the projector $\rho_Q=\whii_Q\otimes$ we also have a functorial identification
$$
\Hom_{\dmkF{\ff_p}}(\rho_Q(A)\otimes\rho_Q(B),\rho_Q(C))=\Hom_{\dmkF{\ff_p}}(\rho_Q(A),\rho_Q(B^{\vee})\otimes\rho_Q(C)).
$$
Taking into account that $M(X)^{\vee}=M(X)(-d)[-2d]$,  where $d=\ddim(X)$, we obtain that
$$
\Hom_{Chow(k/k,\ff_p)}(M(X),M(Y))=\Hom_{\dmELF{k}{k}{\ff_p}}(T(d)[2d],M(X\times Y))=\Ch_{k/k}^{\ddim(Y)}(X\times Y).
$$
\Qed
\end{proof}

We can describe Chow motives disappearing in the {\it isotropic category}.

\begin{remark}
\label{Chow-vanish}
{\rm
An object $U$ of $Chow(k,\ff_p)$ vanishes in $Chow(k/k,\ff_p)$ if and only if it is a direct summand in the motive of a (smooth
projective) {\it anisotropic} variety\footnote{I'm grateful to T.Bachmann for emphasizing this.}.
Indeed, a direct summand $U$ of $M(P)$ vanishes in $Chow(k/k,\ff_p)$ if and only if the identity map $id_U:U\row U$ does.
By Propositions \ref{D-ChowGroups} and \ref{D-ChowHom}, this means that the map $\Delta_U:T\row U\otimes U^{\vee}$ factors through
(the motive of) a smooth projective {\it anisotropic} variety $Q$. Consequently, $U$ is a direct summand of $M(Q)\otimes U$, which, in turn, is a direct summand of $M(Q\times P)$, and the latter variety is still anisotropic.
}
\Red
\end{remark}

From Propositions \ref{D-ChowGroups} and \ref{D-ChowHom} we obtain:

\begin{corollary}
\label{local-mor=global}
The functor $Chow(k,\ff/p)\row Chow(k/k,\ff/p)$ is surjective
on morphisms.
\end{corollary}

In other words, all "local" morphisms between (isotropic) Chow motives are defined "globally".

We will have a closer look at the category $Chow(k/k,\ff_p)$ in Section \ref{IsoChowM}.

\section{Local motivic cohomology of a point}
\label{Sect-coh-point}

In this Section we will compute the motivic cohomology of a point in the {\it isotropic motivic category}, for $p=2$.
This will be achieved by substituting all anisotropic $k$-varieties in the colimit of Proposition \ref{hom-iso-geom}
by norm-varieties for non-zero pure symbols from $K^M_*(k)/2$ (anisotropic Pfister quadrics, in our case).
This makes the problem amenable to calculation due to Voevodsky technique. Moreover, the resulting answer, drastically different from
the "global" one, in turn, sheds some light on this technique.

The starting point is the following statement,
which is a slight modification of the result of J.-L.Colliot-Th\'el\`ene and M.Levine \cite[Theorem 3]{C-TL}.
We provide a somewhat different proof.

\begin{statement}
\label{B1}
Let $B$ be an anisotropic (mod $n$) projective variety. Then, over some finitely generated purely transcendental extension,
it can be embedded into an anisotropic hypersurface of degree $n$.
\end{statement}

\begin{proof}
Embed $B$ into a projective space. Passing to a Veronese embedding, we can assume that all the relations in the projective coordinate
ring of $B$ are generated by quadratic ones, or in other words, that $B$ is defined by quadrics.
Then it will be also defined by hypersurfaces of degree $n$ (in our $\pp^m$), for any $n\geq 2$.
Let $\pp^r=\op{Proj}|D|$ be the projective
system of hypersurfaces of degree $n$ containing $B$. I claim that the generic element of this linear system is an anisotropic hypersurface.
Consider $Y\subset (\pp^m\backslash B)\times \pp^r$ defined by $Y=\{(x,H)|x\in H\}$. Then $Y$ is a projective bundle
$\op{Proj}_{(\pp^m\backslash B)}(V)$ over $(\pp^m\backslash B)$, where $V$ is a co-dimension one subbundle in the trivial bundle $|D|$.
Let $Y_{\eta}$ be the generic fiber of the projection
$Y\row\pp^r$. This is exactly $(Q_{\eta}\backslash B)$, where $Q_{\eta}$ is the generic hypersurface of degree $n$ passing through $B$.
Note, that the degree (mod $n$) is well-defined on the zero-cycles on $Y_{\eta}$, since $B$ is anisotropic.
By the projective bundle theorem, $\CH^*(Y)$ is a free module over $\CH^*(\pp^m\backslash B)$ with the basis $1,\rho,\ldots,\rho^{r-1}$,
where $\rho=c_1(O(1))$. On the other hand, we have a surjective ring homomorphism $\CH^*(Y)\twoheadrightarrow\CH^*(Y_{\eta})$
which is zero on $\rho$ (as this class is supported on a hypersurface in $\pp^r$). Thus, we obtain the surjective map
$\CH^*(\pp^m\backslash B)\twoheadrightarrow\CH^*(Y_{\eta})$ which sends the class $c\in\CH^*(\pp^m\backslash B)$ to the
restriction of $\pi^*(c)$ to $Y_{\eta}$, where $\pi$ is our projective bundle fibration. In particular, $c\in\CH_1(\pp^m\backslash B)$
is mapped to a zero cycle on $Y_{\eta}$ whose degree is equal to the intersection number of $c$ and any hypersurface from our linear
system (which, again, makes sense, since $B$ is anisotropic). Hence, it is a zero cycle of degree $0$ (mod $n$). Thus, the degrees of all
zero-cycles on $Y_{\eta}$ are divisible by $n$, and so, the same is true about $Q_{\eta}=Y_{\eta}\cup B$.
\Qed
\end{proof}

\begin{corollary}
\label{B2}
Let $k$ be a flexible field, $U\in Ob(\dmgmkF{\ff_p})$ and $V\in Ob(\dmkF{\ff_p})$. Then
$$
\Hom_{\dmELF{k}{k}{\ff_p}}(U,V)=\operatornamewithlimits{colim}_{Q}\Hom_{\whii_Q\otimes\dmkF{\ff_p}}(U,V),
$$
where the colimit is taken over all the functors $\otimes\whii_{Q}$, where $Q$ runs over all anisotropic hypersurfaces of degree $p$
over $k$. This system is directed.
\end{corollary}

\begin{proof}
Let $B$ be any anisotropic variety over $k$.
By the Statement \ref{B1}, there exists a purely transcendental field extension $E/k$ and anisotropic hypersurface $Q$ over $E$
such that $\hii_B|_E\geq\hii_Q$. Let $k=k_0(\pp^{\infty})$. Then there exists a diagram of purely transcendental extensions of fields
$$
\xymatrix{
& k \ar[r] & E\\
k_0 \ar[r] & l \ar[u] \ar[r] & L \ar[u]
}
$$
with the extensions of the bottom row finitely generated, such that
the variety $B$ is defined over $l$, while the variety $Q$ and the correspondence $B\rightsquigarrow Q$ (of degree $1$)
are defined over $L$. But we can embed $L$ into $k$ over $l$ so that $k/L$ will be purely transcendental.
Thus, we obtain an anisotropic hypersurface $Q'$ over $k$ together with
a correspondence $B\rightsquigarrow Q'$ of degree $1$. 
Anisotropic hypersurfaces of degree $p$ thus form a
{\it final} subsystem in the system of all anisotropic varieties
which is {\it directed}, hence this subsystem is {\it directed}
as well.
\Qed
\end{proof}

\begin{corollary}
\label{B3}
Let $k$ be a flexible field and $p=2$. Let $U\in Ob(\dmgmkF{\ff_p})$ and $V\in Ob(\dmkF{\ff_p})$. Then
$$
\Hom_{\dmELF{k}{k}{\ff_p}}(U,V)=\operatornamewithlimits{colim}_{\alpha}\Hom_{\whii_{Q_{\alpha}}\otimes\dmkF{\ff_p}}(U,V),
$$
where the colimit is taken over all the functors $\otimes\whii_{Q_{\alpha}}$, where $\alpha$ runs over all non-zero pure symbols from
$K^M_*(k)/2$ and $Q_{\alpha}$ is the respective Pfister quadric. This is a directed system.
\end{corollary}

\begin{proof}
By \cite[Cor. 3]{kerM} (see also \cite{HoIz}), every anisotropic quadric $Q$ (over any field $k$) can be embedded into an anisotropic
Pfister quadric $Q_{\alpha}$
over an appropriate purely transcendental extension of finite transcendence degree. If now $k$ is flexible, then arguing as in the proof
of Corollary \ref{B2}, we can embed $Q$ into some anisotropic Pfister quadric $Q_{\alpha'}$ over $k$. Thus, the set of anisotropic Pfister
quadrics form a
{\it final} subsystem in the system of all anisotropic varieties over a flexible field, which, again, must be {\it directed}.
\Qed
\end{proof}

From the fact that the system in Corollary \ref{B3} is directed, as a by-product, we obtain the following result
(which, of course, is a simple consequence of Statement \ref{B1} and \cite[Cor. 3]{kerM}, and can be even seen from the
latter result alone):

\begin{proposition}
\label{B-pure-dir}
Let $k$ be a flexible field, and $\{\alpha_l\}_{l\in L}$ be a finite collection of non-zero pure symbols
from $K^M_*(k)/2$.  Then there exists a non-zero pure symbol $\alpha\in K^M_*(k)/2$ divisible by every $\alpha_l$.
\end{proposition}

\medskip

Using Corollary \ref{B3},
we can compute the cohomology of a point in {\it isotropic motivic category} for $p=2$.
For a non-zero {\it pure symbol} $\alpha\in K^M_r(k)/2$, let us denote as $\dmELF{\tilde{\alpha}}{k}{\ff_2}$
the full triangulated subcategory
$\whii_{\alpha}\otimes\dmkF{\ff_2}$, where $\whii_{\alpha}=\whii_{Q_{\alpha}}$ and $Q_{\alpha}$ is the respective Pfister quadric.
Hom's between Tate-objects in this category can be computed as follows.

Define an $\ff_2$-vector space $\Qn{n}=\oplus_I r_I\cdot\ff_2$, where $I$ runs over all subsets of $\ov{n}=\{0,1,\ldots,n\}$,
with the structure of a module over
Milnor's operations $Q_i$ defined by: $Q_i(r_I)=r_{I\backslash i}$, if $i\in I$, and zero otherwise, and with the bi-degree of
$r_{\emptyset}$ being $(0)[0]$.
Let $r_{\{n+1\}}$ be a polynomial generator with $Q_{n+1}(r_{\{n+1\}})=r_{\emptyset}$ and
$Q_i(r_{\{n+1\}})=0$, for $i\neq n+1$.
Let $\Ia$ be a module over $K^M_*(k)/2$ isomorphic to the principal ideal $\alpha\cdot K^M_*(k)/2$
with the generator in bi-degree $(0)[0]$. In other words, $\Ia=K^M_*(k)/\kker(\cdot\alpha)$.
In particular, it has a natural ring structure. The multiplicative structure on $\Qn{n}[r_{\{n+1\}}]\otimes_{\ff_2}\Ia$ is provided by
$r_I=\prod_{i\in I}r_{\{i\}}$ and the identity:
$r_{\{i\}}^2=r_{\{i+1\}}\cdot\rho$, for $0\leq i\leq n$, and $\rho=\{-1\}$.
In other words, this is the ring $\Ia[r_{\{i\}}|_{0\leq i\leq n+1}]/(r_{\{i\}}^2-r_{\{i+1\}}\cdot\rho|_{0\leq i\leq n})$.

For a motivic category $\op{D}$ with Tate-objects, let us denote as $\eend_{\op{D}}(V)$ the ring $\oplus_{a,b}\Hom_{\op{D}}(V,V(a)[b])$.

\begin{theorem}
\label{D-A4}
Let $\alpha\in K^M_m(k)/2$ be a non-zero pure symbol. Then
$$
\eend_{\dmELF{\tilde{\alpha}}{k}{\ff_2}}(T)=\Qn{m-2}[r_{\{m-1\}}]\otimes_{\ff_2}\Ia=
\Ia[r_{\{i\}}|_{0\leq i\leq m-1}]/(r_{\{i\}}^2-r_{\{i+1\}}\cdot\rho|_{0\leq i\leq m-2}).
$$
\end{theorem}

\begin{proof}
By definition,
$\eend_{\dmELF{\tilde{\alpha}}{k}{\ff_2}}(T)=\eend_{\dmkF{\ff_2}}(\whii_{\alpha})$.
From this point, all the Hom's will be in the category $\dmkF{\ff_2}$, unless specified otherwise, so I will omit it from notations.

Let $M_{\alpha}$ be the respective Rost motive (\cite{Ro98}).
We have natural maps
$T(d)[2d]\row M_{\alpha}\row T$, where $d=2^{m-1}-1$, whose composition is zero. Cutting out the respective Tate-motives from $M_{\alpha}$
and tensoring the result by $\hii_{\alpha}$ and $\whii_{\alpha}$, respectively, we obtain:
$$
\xymatrix{
& T \ar[d]^{[1]} \ar[rd]^{[1]} & \\
M_{\alpha} \ar[ru] & R \ar[l] \ar[r] \ar@{}[lu]|-(0.22){\star} \ar@{}[rd]|-(0.22){\star} & \wt{M}_{\alpha} \ar[ld]^{[1]}\\
& T(d)[2d] \ar[u] \ar[lu] &
}
\hspace{1cm}\text{and}\hspace{1cm}
\xymatrix{
& \hii_{\alpha} \ar[dd]^{[1]}_{\mu} \\
M_{\alpha} \ar[ru] \ar@{}[r]|-(0.60){\star} & \\
& \hii_{\alpha}(d)[2d] \ar[lu]
}
\hspace{0.5cm};\hspace{0.5cm}
\xymatrix{
\whii_{\alpha} \ar[rd]^{[1]} & \\
 & \wt{M}_{\alpha} \ar[ld]^{[1]} \ar@{}[l]|-(0.60){\star} \\
\whii_{\alpha}(d)[2d] \ar[uu]^{[-1]}_{\eta}  &
}
$$
Here we are using the fact that $M_{\alpha}\otimes\whii_{\alpha}=0$ and that $\wt{M}_{\alpha}\otimes\hii_{\alpha}=0$,
which are equivalent to the exactness
of the left triangle - \cite[Thm 4.4]{Voe03}. Let us denote the above half of the octahedron as $\diamondsuit$.
Note, that since there are no $\Hom$'s from $\hii_{\alpha}$ to $\whii_{\alpha}(*)[*']$, we can naturally identify groups
$\Hom(\whii_{\alpha},\whii_{\alpha}(*)[*'])=\Hom(T,\whii_{\alpha}(*)[*'])$.

For each $0\leq i<m-1$, let $\beta\in K^M_{i+1}(k)/2$ be any pure symbol dividing $\alpha$. We obtain a similar map
$\eta_{\beta}(-d_i)[-2d_i]:\whii_{\beta}\row\whii_{\beta}(-d_i)[-2d_i-1]$, where $d_i=2^i-1$.
Tensoring it with $\whii_{\alpha}$, we obtain the map $r_{\{i\}}:\whii_{\alpha}\row\whii_{\alpha}(-d_i)[-2d_i-1]$.
Or, in other words, an element $r_{\{i\}}\in\Hom_{\dmELF{\tilde{\alpha}}{k}{\ff_2}}(T,T(-d_i)[-2d_i-1])$.
Below we will see that this map doesn't depend on the choice of the divisor $\beta$.

For any smooth projective $R$, we have the natural (homological) action of the Steenrod algebra on
$\Hom(T,\whii_R(*)[*'])$ and the natural (cohomological) action of it on $\Hom(\whii_R,T(*)[*'])$,
which commute with the maps $\whii_R\row\whii_S$ (for $\hii_R\geq\hii_S$).
In particular, we have the action of the Milnor's operations $Q_i$.
If $R$ is a $\nu_i$-variety, then by the arguments of V.Voevodsky \cite[Cor 3.8]{Voe03},
the differential $Q_i$ is exact on $\Hom(\whii_R,T(*)[*'])$. By the same arguments, it is exact on
$\Hom(T,\whii_R(*)[*'])$. In particular, in our case,
$Q_i$ is exact on $\Hom(T,\whii_{\alpha}(*)[*'])$ and $\Hom(\whii_{\alpha},T(*)[*'])$, for any $i\leq m-1$.

Consider $N=\oplus\Hom(\whii_{\alpha}[-1],\hii_{\alpha}(*)[*'])$. It has the natural right action of
$\wt{A}=\eend(\whii_{\alpha})$
as well as the left action of $A=\eend(\hii_{\alpha})$. In particular, there is a right action by $\eta$ and a
left one by $\mu$. I claim that these are mutually inverse. Indeed, to see that for any $f:\whii_{\alpha}\row\hii_{\alpha}(a)[b]$,
one has $\mu\bullet f\bullet\eta=f$, it is sufficient to look at the "vertical axis" of
$\diamondsuit\otimes(\whii_{\alpha}\stackrel{f}{\row}\hii_{\alpha}(a)[b])$, which is (after rotating by $90^{\circ}$):
$$
\xymatrix{
\whii_{\alpha} \ar[d]_{f} \ar@{=}[r]^{[1]} & \whii_{\alpha}[-1] \ar[d] & \whii_{\alpha}(d)[2d] \ar[l]_{\eta} \ar[d]^{f(d)[2d]} \\
\hii_{\alpha}(a)[b] \ar[r]^(0.4){[1]}_(0.4){\mu(a)[b]} & \hii_{\alpha}(a+d)[b+2d] & \hii_{\alpha}(a+d)[b+2d] \ar@{=}[l]
}.
$$
Thus, $N$ is $\mu-\eta$-periodic, as the action by $\mu$ and 
$\eta$ are mutually inverse isomorphisms on $N$. 

At the same time,
$\Hom(\whii_{\alpha},\whii_{\alpha}(a)[b])=\Hom(T,\whii_{\alpha}(a)[b])$ is zero for $b>a$, while $\Hom(\whii_{\alpha},T(a)[b])=0$,
for $b\leq a+1$ by the Beilinson-Lichtenbaum "conjecture" (note, that $\whii_{\alpha}$ disappears in the etale topology).
Considering $\Hom$'s from $\whii_{\alpha}[-1]$ to the $(a)[b]$-shifted exact triangle
$
\hii_{\alpha}\row T\row\whii_{\alpha}\row\hii_{\alpha}[1],
$
we obtain that
$\wt{A}$ is exactly the $\leq 0$ diagonal part of $N$, while
$\hm^{*,*'}(\whii_{\alpha}[-1];\ff_2)$ is exactly the $>0$ diagonal part of it.
Thus, $N$ combines the homology and cohomology groups of $\whii_{\alpha}$.

Since there are no $\Hom$'s from $T$ to $\hii_{\alpha}(a)[a+1]$, the map $\Hom(T,T(a)[a])\twoheadrightarrow\Hom(T,\whii_{\alpha}(a)[a])$
is surjective, and so, the $0$-th diagonal of $\wt{A}$ (or, which is the same, the $0$-th diagonal of $N$) as a $K^M_*(k)/2$-module
is generated by $1$ - the unit of this ring. Let $R_{\alpha}$ be this $0$-th diagonal.
From $\mu-\eta$-periodicity, the diagonal number $(-2^{m-1})$ (where $\eta$ resides),
as a $K^M_*(k)/2$-module, is generated by $\eta$. Since the differential $Q_{m-1}$ is exact on $\wt{A}$, we obtain that $1$ is covered
by the image of $Q_{m-1}$ (since $\wt{A}$ is concentrated in non-positive diagonals). But the only non-zero element of the needed
bi-degree is $\eta$. Thus, $Q_{m-1}(\eta)=1$. Applying the same arguments to the symbol $\beta$ (considered above), we obtain that
$Q_i(\eta_{\beta})=1$, and so, $Q_i(r_{\{i\}})=1$.

For any $I\subset\ov{(m-2)}$, denote $r_I:=\prod_{i\in I}r_{\{i\}}$ and $Q_I=\circ_{i\in I}Q_i$.
Denote as $D_j$ the $j$-th diagonal of $N$. For $0\leq i\leq m-2$, let $I_i=\{i,i+1,\ldots,m-2\}$.
Since $\hm^{*,*'}(\whii_{\alpha}[-1];\ff_2)$ is trivial below the $1$-st diagonal,
and $Q_l$'s are exact, the composition $Q_{I_i}:D_{2^i}\row D_{2^{m-1}}$ is injective.
But from $\mu-\eta$-periodicity, $D_{2^{m-1}}$ as a $K^M_*(k)/2$-module is generated by $\mu$.
In particular, the $D_{2^i}$ is trivial below $\mu\bullet 1\bullet r_{I_i}$. By the $\mu-\eta$-periodicity,
$D_{2^i-2^{m-1}}$ is trivial below $r_{I_i}$. In particular, $Q_i(\eta)=0$, for any $0\leq i\leq m-2$ (since this element is below $r_{I_i}$),
and since $\eta$ generates $D_{-2^{m-1}}$, all the differentials $Q_i$, for $i\leq m-2$ are trivial on this diagonal.
Applying the same arguments to the divisor $\beta$ of degree $(i+1)$, we obtain that $Q_l(\eta_{\beta})=0$, and hence,
$Q_l(r_{\{i\}})=0$, for $l<i$ (the fact that $Q_l(r_{\{i\}})=0$, for $l>i$, is obvious, since $\wt{A}$ is concentrated in non-positive
diagonals).
Combining it with the (external) co-multiplication identity for Milnor's operations:
\begin{equation}
\label{comult-Qi}
Q_{K}(x\times y)=\sum_{2^I+2^J=2^K}Q_{I}(x)\times Q_{J}(y)\cdot\{-1\}^{|I|+|J|-|K|},
\end{equation}
where $2^I=\sum_{i\in I}2^i$, we obtain that $Q_I(r_I)=1$, for any $I\subset\ov{(m-2)}$.

Let us show that $D_{-2^I}$ is a free module over $D_0=\Ia$ generated by $r_I$. Let $s$ be the smallest element of $I$
(which we assume to be $m-1$, if $I$ is empty).
Decreasing induction on $s$. The (base) $s=m-1$ follows from $\mu-\eta$-periodicity.
The (step): let $2^J=2^I+2^s$. Since $I\backslash s$ and $J$ consist of elements
larger than $s$, by inductive assumption, $D_{-(2^I+2^s)}=r_J\cdot D_0$ and $D_{-(2^I-2^s)}=r_{I\backslash s}\cdot D_0$ (where we
denote $r_{\{m-1\}}:=\eta$).
In particular, $Q_s$ is trivial on $D_{-(2^I+2^s)}$, and from the exact sequence
$D_{-(2^I+2^s)}\stackrel{Q_s}{\lrow}D_{-2^I}\stackrel{Q_s}{\lrow}D_{-(2^I-2^s)}$, taking into account that $Q_s(r_I)=r_{I\backslash s}$,
we see that $Q_s:D_{-2^I}\stackrel{\cong}{\lrow}D_{-(2^I-2^s)}$ is an isomorphism inverse to the multiplication by $r_s$. Thus,
$D_{-2^I}=r_I\cdot D_0$.

From $\mu-\eta$-periodicity we obtain that $\wt{A}$ is a free module over $\Ia$ generated by $\eta^k\cdot r_I$, for
$k\geq 0$ and $I\subset\ov{(m-2)}$, while $\hm^{*,*'}(\whii_{\alpha}[-1];\ff_2)$ is a free module over $\Ia$
generated by $\mu^l\bullet r_I$, for $l>0$ and $I\subset\ov{(m-2)}$.
Consider $\gamma=\mu\bullet 1\bullet r_{\ov{(m-2)}}$ - the generator of $D_1$ - the 1-st diagonal in $N$, or, which is the same,
the 1-st diagonal in the motivic cohomology of $\whii_{\alpha}[-1]$.
By the Beilinson-Lichtenbaum Conjecture,
multiplication by $\tau$ identifies $D_1$ with the kernel $\kker(K^M_*(k)/2\row K^M_*(k(Q_{\alpha}))/2)$ - see \cite[Lem 6.4]{Vo-BKMK}.
The generator $\gamma$ is identified with some element of degree $m$, which must coincide with the symbol $\alpha$
(since $\alpha$ vanishes over $k(Q_{\alpha})$ and there exists exactly one non-zero element of the respective degree in $D_1$).
Hence, as a $K^M_*(k)/2$-module, $\Ia$ can be identified
with the principal ideal of $K^M_*(k)/2$ generated by $\alpha$. So, as a ring,
$R_{\alpha}=(K^M_*(k)/2)/(\kker(\cdot\alpha))$.

This gives the description of $\wt{A}$ (as well as $N$) as a module over $K^M_*(k)/2$ and over Steenrod algebra.
Finally, the equation $r_{\{i\}}^2=r_{\{i+1\}}\cdot\rho$ follows from the co-multiplication identity for Milnor's operations
(\ref{comult-Qi}).
\Qed
\end{proof}

As a by-product, we obtain the description of motivic cohomology of $\whii_{\alpha}$ (known already from the original version of \cite{OVV}
and \cite[Thm 5.8]{Yap}) but now enhanced with the structure of a module over motivic homology of $\whii_{\alpha}$.

\begin{corollary}
\label{mot-coh-whii}
Let $\alpha\in K^M_m(k)/2$ be a non-zero pure symbol.
As a module over $\wt{A}=\eend_{\dmkF{\ff_2}}(\whii_{\alpha}$),
$$
\hm^{*,*'}(\whii_{\alpha}[-1];\ff_2)=\wt{A}[r_{\{m-1\}}^{-1}]/\wt{A}.
$$
It is a free $\Ia$-module with the basis $r_{\{m-1\}}^{-l}\cdot r_I$, for $l>0$ and $I\subset\ov{(m-2)}$.
\end{corollary}

Now we can compute the "local" motivic cohomology $\hm^{*,*'}(k/k;\ff_2)=\eend_{\dmELF{k}{k}{\ff_2}}(T)$.

\begin{theorem}
\label{D-A5}
Let $k$ be a flexible field. Then
\begin{equation*}
\hm^{*,*'}(k/k;\ff_2)=\Qn{\infty}=\Lambda_{\ff_2}(r_{\{i\}}|_{i\geq 0}).
\end{equation*}
\end{theorem}

\begin{proof}
As we know, our colimit (from Corollary \ref{B3}) is taken over a {\it directed} system.

Let $\alpha\in K^M_m(k)/2$ and $\beta=\alpha\cdot\{b\}\in K^M_{m+1}(k)/2$ be non-zero pure symbols. Consider the restriction
$$
\eend_{\dmELF{\tilde{\alpha}}{k}{\ff_2}}(T)\stackrel{res}{\lrow}
\eend_{\dmELF{\tilde{\beta}}{k}{\ff_2}}(T).
$$
Then $res=\whii_{\beta}\otimes$ is a ring homomorphism respecting Steenrod algebra action, and, in the notations
of Theorem \ref{D-A4},
for $I\subset\{0,\ldots,m-2\}$, we have: $res(r_I)=r_I$ and $res(r_{\{m-1\}}\cdot r_I)=r_{I\cup\{m-1\}}$.
Indeed, $res$ sends $r_{\emptyset}$ to $r_{\emptyset}$, while
$r_I$ and $r_I$ (respectively, $r_{\{m-1\}}\cdot r_I$ and $r_{I\cup\{m-1\}}$) are the only elements (in the source and the target)
which are mapped to $r_{\emptyset}$ via $Q_I$ (respectively, $Q_{I\cup\{m-1\}}$). Also, $res:\Ia\row\Ib$ is
the natural projection, corresponding to the map $\alpha\cdot K^M_*(k)/2\stackrel{\cdot\{b\}}{\lrow}\beta\cdot K^M_*(k)/2$.

Combining this with Corollary \ref{B3} we obtain that
$$
\hm^{*,*'}(k/k;\ff_2)=\Qn{\infty}\otimes_{\ff_2} (K^M_*(k)/2)/N,
$$
where $N=\cup_{\alpha}\kker(\cdot\alpha)$, where $\alpha$ runs over all non-zero pure symbols in $K^M_*(k)/2$.
It remains to observe that $N$ contains $K^M_1(k)/2$. Indeed, let $\{a\}\in K^M_1(k)/2$ be any element such that $\{a\}\neq 0\neq\{-a\}$.
Then, from $\{a,-a\}=0$, both $\{a\}$ and $\{-a\}$ belong to $N$. This implies that $\{-1\}\in N$. Since, over a flexible field, an
element $a$ as above always exists, we obtain that $N$ contains $K^M_1(k)/2$, and so, coincides with the augmentation ideal $K^M_{>0}(k)/2$.
Hence, $(K^M_*(k)/2)/N=\ff_2$.

Finally, from the equation $r_{\{i\}}^2=r_{\{i+1\}}\cdot\rho$ it follows that $r_{\{i\}}^2=0$.
So, we obtain the external algebra in $r_{\{i\}}$'s over $\ff_2$.
\Qed
\end{proof}

What is remarkable here, the Milnor's operations are intertwined into the very fabric of the local motivic category. Also, all the non-zero
elements of $\hm^{*,*'}(k/k;\ff_2)$ are "rigid", in the sense that the identity map of the unit object of the local category can be
obtained from any such element using Milnor's operations, and so these classes disappear only together with the category itself.

We can now compare "local" and "global" motivic cohomology of a point:
$$
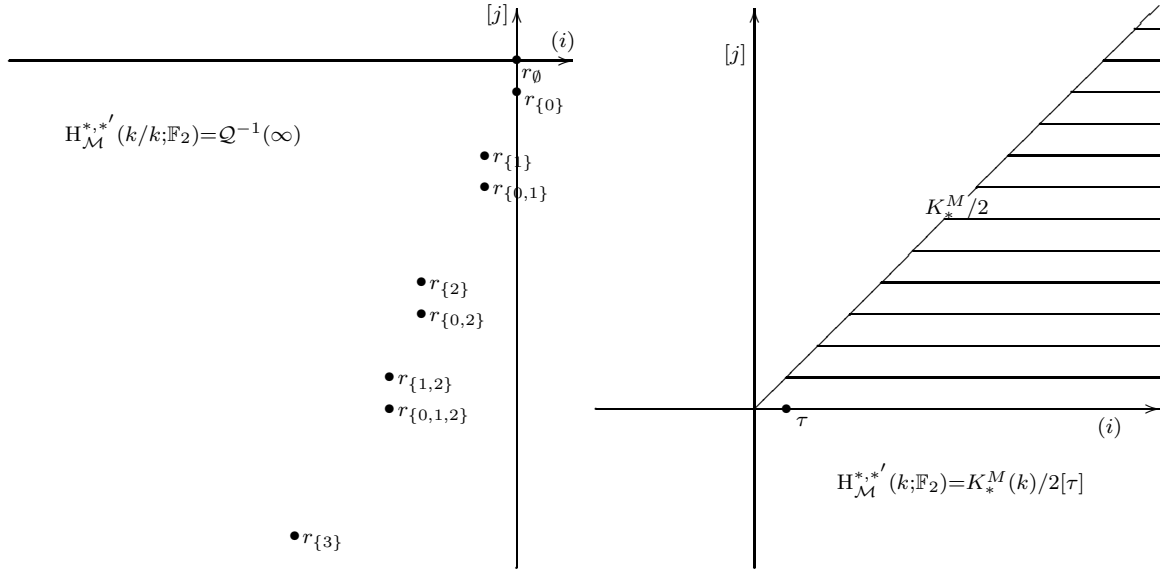
\def\objectstyle{\scriptstyle}
\def\labelstyle{\scriptstyle}
\xymatrix @-1.5pc{
& & & & & & & & & & & & & & & & & &  \\
& & & & & & & & & & & & & & & & & &  \\
\dub \ar@{->}[rrrrrrrrrrrrrrrrrr]^(0.97){(i)} & & & & & & & & & & & & & & & & \bub \ar@{}[r]_{r_{\emptyset}} & &  \\
& & & & & & & & & & & & & & & & \bub \ar@{}[d]^(0.4){r_{\{0\}}} & &  \\
& & & & & & & & & & & & & & & & & &  \\
& & & & & \ar@{}[r]^{\hm^{*,*'}(k/k;\ff_2)=\Qn{\infty}} & & & & & & & & & & \bub \ar@{}[d]^(0.2){r_{\{1\}}} & & &  \\
& & & & & & & & & & & & & & & \bub \ar@{}[d]^(0.2){r_{\{0,1\}}} & & &  \\
& & & & & & & & & & & & & & & & & &  \\
& & & & & & & & & & & & & & & & & &  \\
& & & & & & & & & & & & & \bub \ar@{}[d]^(0.2){r_{\{2\}}} & & & & &  \\
& & & & & & & & & & & & & \bub \ar@{}[d]^(0.2){r_{\{0,2\}}} & & & & &  \\
& & & & & & & & & & & & & & & & & &  \\
& & & & & & & & & & & & \bub \ar@{}[d]^(0.2){r_{\{1,2\}}} & & & & & &  \\
& & & & & & & & & & & & \bub \ar@{}[d]^(0.2){r_{\{0,1,2\}}} & & & & & &  \\
& & & & & & & & & & & & & & & & & &  \\
& & & & & & & & & & & & & & & & & &  \\
& & & & & & & & & & & & & & & & & &  \\
& & & & & & & & & \bub \ar@{}[d]^(0.2){r_{\{3\}}} & & & & & & & & &  \\
& & & & & & & & & & & & & & & & \dub \ar@{->}[uuuuuuuuuuuuuuuuuu]^(0.97){[j]} & &  \\
}
\def\objectstyle{\scriptstyle}
\def\labelstyle{\scriptstyle}
\xymatrix @-1.5pc{
& & & & & & & & & & & & & & & & & &  \\
& & & & & & & & & & & & & & & & & \dub \ar@{-}[r] &  \\
& & & & & & & & & & & & & & & & \dub \ar@{-}[rr] & &  \\
& & & & & & & & & & & & & & & \dub \ar@{-}[rrr] & & &  \\
& & & & & & & & & & & & & & \dub \ar@{-}[rrrr] & & & &  \\
& & & & & & & & & & & & & \dub \ar@{-}[rrrrr] & & & & &  \\
& & & & & & & & & & & & \dub \ar@{-}[rrrrrr] & & & & & &  \\
& & & & & & & & & & & \dub \ar@{-}[rrrrrrr] & & & & & & &  \\
& & & & & & & & & & \dub \ar@{-}[rrrrrrrr] & & & & & & & &  \\
& & & & & & & & & \dub \ar@{-}[rrrrrrrrr] & & & & & & & & &  \\
& & & & & & & & \dub \ar@{-}[rrrrrrrrrr] & & & & & & & & & &  \\
& & & & & & & \dub \ar@{-}[rrrrrrrrrrr] & & & & & & & & & & &  \\
& & & & & & \dub \ar@{-}[rrrrrrrrrrrr] & & & & & & & & & & & &  \\
\dub \ar@{->}[rrrrrrrrrrrrrrrrrr]_(0.9){(i)}& & & & & \dub \ar@{-}[rrrrrrrrrrrrruuuuuuuuuuuuu]|-{K^M_*/2} &
\bub \ar@{}[r]_{\tau} & & & & & & & & & & & &  \\
& & & & & & & & & & & & & & & & & &  \\
& & & & & & & & & & & & & & & & & &  \\
& & & & & & & & & & & \ar@{}[r]^{\hm^{*,*'}(k;\ff_2)=K^M_*(k)/2[\tau]} & & & & & & &  \\
& & & & & & & & & & & & & & & & & &  \\
& & & & & \dub \ar@{->}[uuuuuuuuuuuuuuuuuu]^(0.9){[j]} & & & & & & & & & & & & &  \\
}
$$

Note, that in contrast to "global" motivic cohomology of a point residing in the I-st quadrant, the "local" version resides in the
III-rd one.
In particular, the global-to-local map
$$
\hm^{*,*'}(k,\ff_2)\lrow\hm^{*,*'}(k/k,\ff_2)
$$
is zero in all bi-degrees aside from $(0)[0]$.\\

Our ring generators $r_{\{i\}}$ are related by the action of the Steenrod algebra. Namely, since (modulo $\rho=\{-1\}$),
$Q_{i+1}=[Q_i,\op{Sq}^{2^{i+1}}]$ - \cite{VoOP}, and $\rho$ disappears locally, anyway, we obtain from bi-degree considerations that
$\op{Sq}^{2^{i+1}}(r_{\{i+1\}})=r_{\{i\}}$.
In particular,
$\op{Sq}^1\op{Sq}^2\ldots\op{Sq}^{2^{i-1}}\op{Sq}^{2^i}r_{\{i\}}=r_{\emptyset}=1$.

\begin{remark}
\label{tau=0}
{\rm
Despite some similarities between the (complex) topological realization functor and isotropic functors, there is a difference in the way
they handle $\tau$. Namely, $\tau_{Top}=1$, while $\psi_E(\tau)=0$ (in the case of a flexible field).
}
\Red
\end{remark}

\smallskip

The only obstacle which prevents us from performing the same calculations for odd primes is the lack of the analogue of
\cite[Cor. 3]{kerM} in this situation.
In particular, it would be sufficient to have a positive answer to the following:

\begin{question}
\label{odd-loc-mot-coh}
Let $Q$ be an anisotropic hypersurface of degree $p$ over $k$.
Is it true that, over some finitely-generated purely transcendental extension $E/k$, the kernel
$$
\kker(K^M_*(E)/p\row K^M_*(E(Q))/p)
$$
contains a non-zero pure symbol?
\end{question}

\section{Isotropic category of Chow motives}
\label{IsoChowM}

In this Section we will study in details {\it local Chow motives}.
As we will see, over a flexible field, these resemble in many respects their topological counterparts, and
are closely related to the numerical equivalence of cycles with finite coefficients.
In particular, the $\Hom$'s between such {\it local
pure motives} are expected to be not larger than $\Hom$'s between their topological realizations, and so
finite-dimensional. We will prove it in various situations.

Let us start by introducing some "gradual" approach to the numerical equivalence of cycles, which will permit to measure
our progress towards the goal.

\subsection{Theories of higher types and numerical equivalence}

Let $A$ be a commutative ring and $\laz\stackrel{\ffi}{\lrow} A$ be some formal group law with $A$-coefficients.
Denote as $A_{(0)}^*:=\Omega^*\otimes_{\laz}A$ the respective {\it free theory} in the sense of Levine-Morel \cite[Rem. 2.4.14]{LM}.
By \cite[Prop. 4.7]{SU} this is a {\it theory of rational type}. We call this {\it type $0$}.
We are going to introduce the theory $A_{(n)}^*$ of {\it type $n$} as some quotient of $A_{(0)}^*$.
This is based on the following construction (cf. \cite[Exa. 4.6]{RNCT}):

\begin{example}
\label{theo-constr}
{\rm Let $A^*$ be some oriented cohomology theory (with localization) in the sense of \cite[Def. 2.1]{SU}, and $\Gamma=\{Q_{\lambda},a_{\lambda}\}_{\lambda\in\Lambda}$ be a collection of smooth
projective $k$-varieties with some classes $a_{\lambda}\in A^*(Q_{\lambda})$. That is, we have a collection of $A^*$-correspondences
$\rho_{\lambda}:Q_{\lambda}\rightsquigarrow\op{Spec}(k)$.
Construct the new theory $A^*_{\Gamma}$ as follows:
$$
A^*_{\Gamma}(X):=A^*(X)/(im((\rho_{\lambda}\times id)_*)_{\lambda\in\Lambda}),
$$
where $\rho_{\lambda}\times id$ is the correspondence $Q_{\lambda}\times X\rightsquigarrow X$.
In other words, we mod out all the elements of the form $\beta_*(\alpha^*(a_{\lambda})\cdot u)$,
where $u$ is an arbitrary element of $A^*(Q_{\lambda}\times X)$ and $\alpha$ and $\beta$ are natural projections:
$$
\xymatrix{
Q_{\lambda} & Q_{\lambda}\times X \ar[l]_-{\alpha} \ar[r]^-{\beta} & X.
}
$$
One can check that the resulting theory $A^*_{\Gamma}$ will be an oriented cohomology theory in the sense of \cite[Def. 2.1]{SU}
}
\Red
\end{example}

\begin{definition}
\label{def-num-eq}
Let $Q\stackrel{\pi}{\lrow}\op{Spec}(k)$ be a smooth projective variety and $a\in A^*(Q)$. We say that $a\numeq 0$, if
$\pi_*(a\cdot b)=0$, for any $b\in A^*(Q)$.
\end{definition}

Now we can introduce the {\it theories of higher types}.

\begin{definition}
\label{def-THT}
Consider the collection $\Gamma_n=\{Q_{\lambda},a_{\lambda}\}_{\lambda\in\Lambda}$, where $Q_{\lambda}$ runs through all smooth
projective $k$-varieties of dimension $2n-1$ and $a_{\lambda}\in A^*(Q_{\lambda})$ runs over all elements $\numeq 0$. Define
$$
A^*_{(n)}:=(A_{(0)})^*_{\Gamma_n}.
$$
\end{definition}

If $a\in A^*(Q)$ is $\numeq 0$, then $a\times [(\pp^1)]\in A^*(Q\times (\pp^1)^{\times 2})$ is also $\numeq 0$.
So, the images of correspondences from $\Gamma_n$ are covered by those from $\Gamma_{n+1}$. Hence, we get a chain of surjections
$$
A_{(0)}^*\twoheadrightarrow A_{(1)}^*\twoheadrightarrow A_{(2)}^*\twoheadrightarrow\ldots\twoheadrightarrow
A_{(n)}^*\twoheadrightarrow\ldots
$$
with the colimit $A^*_{Num}$. Here $A^*_{Num}$ is obtained from $A^*$ by moding-out all classes $\numeq 0$ on all varieties.

\begin{remark}
\label{rem-TT1-alg}
{\rm For $n=1$, the theory $A^*_{(1)}$ is, by definition, the {\it algebraic} version $A^*_{alg}$.
In particular, $\CH^*_{(1)}=\CH^*_{alg}$.}
\Red
\end{remark}

The meaning of the theory $A^*_{Num}$ is described by the following universal property.

\begin{proposition}
\label{num-univ-prop}
For any oriented generically constant cohomology theory (with localization) $\wt{A}^*$ (in the sense of \cite[Def. 2.1]{SU} and \cite[Def.4.4.1]{LM})
with the formal group law $\laz\stackrel{\ffi}{\lrow} A$,
there exists a unique morphism of theories $\wt{A}^*\twoheadrightarrow A^*_{Num}$
which is moreover surjective.
\end{proposition}

\begin{proof}
By \cite[Prop. 4.8]{SU}, the canonical morphism of theories $G:A^*_{(0)}\twoheadrightarrow \wt{A}^*$ is surjective
(this morphism is induced by the canonical morphism $\Omega^*\row\wt{A}^*$ of \cite[Thm 1.2.6]{LM}).
And by the same universality of algebraic cobordism such morphism is unique. In particular, there is a unique morphism
of theories $A^*_{(0)}\twoheadrightarrow A^*_{Num}$.
Note, that $G|_{\op{Spec}(k)}:A\stackrel{=}{\row}A$ is an isomorphism. Hence, if $x\in A^*_{(0)}(X)$ belongs to the kernel of $G$,
then $x\numeq 0$. Thus, the unique morphism of theories $A^*_{(0)}\twoheadrightarrow A^*_{Num}$ factors through $\wt{A}^*\row A^*_{Num}$.
\Qed
\end{proof}

Thus, $A^*_{Num}$ plays the role opposite to that of $A^*_{(0)}$, and can be denoted as $A^*_{(\infty)}$, while any
generically constant theory $\wt{A}^*$ with the formal group law $\ffi$ is canonically squeezed between $A^*_{(0)}$ and $A^*_{(\infty)}$:
$$
\xymatrix{
A^*_{(0)} \ar@{>>}[r] & \wt{A}^* \ar@{>>}[r] & A^*_{(\infty)}
},
$$
and the latter provides an alternative way of describing such theories $\wt{A}^*$.

\subsection{Numerical equivalence modulo $p$ and isotropy}
\label{nump-iso}

Everywhere below $X$ is a smooth projective variety over a field $k$ of characteristic zero.

In the case of a theory $\Ch^*=\CH^*/p$, we will denote the numerical equivalence $\numeq$ as $\nump$
to stress that we consider finite coefficients. Thus, $\Ch^*_{Num}(X)=\Ch^*(X)/N$, where
$N$ is the ideal of elements $\nump 0$. In our situation, $x\in\Ch(X)$ is $\nump 0$ if, for any $y\in\Ch(X)$, the
$\ddeg(x\cdot y)=0\in\ff_p$.

By the very definition, we have a non-degenerate pairing
$$
\Ch_{Num}(X)\times\Ch_{Num}(X)\lrow\ff_p
$$
defined by $(x,y)\mapsto\ddeg(x\cdot y)$.
Moreover, since $\ddeg(x\cdot y)$ may be defined on the level of the (complex) topological realization,
the kernel of the topological realization functor is contained in $N$, and so, $\Ch_{Num}(X)$
is a sub-quotient of the topological cohomology $\HH_{Top}(X;\ff_p)$ of $X$. In particular,
$\Ch_{Num}(X)$ is finite-dimensional $\ff_p$-vector space. Note, however, that this sub-quotient depends on the
ground field $k$, since for different fields, the images of the topological realization functor will be different.

The theory $\Ch^*_{Num}$ inherits the action of the reduced power operations from $\Ch^*$,
as follows from the Proposition \ref{AB1} below.
Recall, that we have reduced power operations $P^i:\Ch^*\row\Ch^{*+i(p-1)}$ (see \cite{Br} and \cite{VoOP})
commuting with pull-back morphisms, and the respective homological operations $P_i:\Ch_*\row\Ch_{*-i(p-1)}$ commuting
with push-forwards. These are connected as follows:
$$
P^i=\sum_{l=0}^i d_l(T_X)\cdot P_{i-l},
$$
where $d_l$ satisfies Cartan's formula and, for a line bundle $L$, $d_1(L)=x^p$, where $x=c_1(L)$.

\begin{proposition}
\label{AB1}
Let $u\in\Ch^*(X)$. Then $u\nump 0$ $\Rightarrow$ $P^i(u)\nump 0$.
\end{proposition}

\begin{proof} Induction on $i$. Since $P^0=id$, we have the (base) $i=0$.

\noindent
(step) We need to show that $\ddeg(P^i(u)\cdot v)=0$, for any class $v$ of complementary dimension.
By Cartan's formula, $P^i(u)\cdot v=P^i(u\cdot v)-\sum_{j=0}^{i-1}P^j(u)\cdot P^{i-j}(v)$.
By the inductive assumption, we have that $\ddeg(P^j(u)\cdot P^{i-j}(v))=0$, for any $j<i$.
And the degree of the first summand can be rewritten as
\begin{equation*}
\begin{split}
&\ddeg(P^i(u\cdot v))=\ddeg(\sum_{l=0}^i d_l(T_X)\cdot P_{i-l}(u\cdot v))=\\
&\ddeg(P_i(u\cdot v))+\ddeg(\sum_{l=1}^i d_l(T_X)\cdot\sum_{m=0}^{i-l}d_m(-T_X)\cdot P^{i-l-m}(u\cdot v)).
\end{split}
\end{equation*}
The second summand is zero by the inductive assumption, while the degree of
$P_i(u\cdot v)$ is the same as that of $P_i(\pi_*(u\cdot v))$, where $\pi:X\row\spec(k)$ is the natural projection.
The latter degree is equal to zero for any $i>0$. The induction step is proven.
\Qed
\end{proof}

If $x$ is an {\it anisotropic} cycle, then $x\nump 0$. Indeed, $x=f_*(x')$, for some projective map $f:Z\row X$ from
anisotropic variety $Z$, and some $x'\in\Ch^*(Z)$. But then $\ddeg(x\cdot y)=\ddeg(f_*(x')\cdot y)=\ddeg(x'\cdot f^*(y))=0$,
since all zero cycles on $Z$ have zero degree (modulo $p$). Thus we get the surjective map:
$$
\Ch^*_{k/k}(X)\twoheadrightarrow\Ch^*_{Num}(X).
$$
I conjecture that, over flexible fields, these two cohomology theories, in reality, coincide.
\begin{conj}
\label{main-conj}
Let $k$ be a flexible field. Then
$\Ch^*_{k/k}=\Ch^*_{Num}$.
\end{conj}

\begin{remark}
\label{conj-flex-cond}
{\rm
The condition on the flexibility of $k$ is essential here. For example, if $k$ is algebraically closed, then
from Remark \ref{alg-clo-local} we know that $\Ch^*_{k/k}=\Ch^*$, while $\Ch^*_{Num}$ is some sub-quotient of it.
}
\Red
\end{remark}

\begin{remark}
\label{conj-chow-mot}
{\rm
This Conjecture, in particular, implies that the {\it local} Chow motivic category
$Chow(k/k;\ff_p)$ is equivalent to the {\it numerical} Chow motivic category $Chow_{Num}(k;\ff_p)$.

Since the degree pairing is defined over the algebraic closure and even in the topological realization, we obtain that the numerical
Chow groups over $E$ are sub-quotients of the respective groups over $\ov{E}$, which, in turn, are sub-quotients of topological cohomology:
$$
\Ch_{Num/E}(X)\subq\Ch_{Num/\ov{E}}(X)\subq\HH_{Top}(X;\ff_p).
$$
Thus, our Conjecture implies, that any object of $Chow(k;\ff_p)$ which vanishes over $\kbar$, or even in the topological realization,
should vanish in every isotropic category $Chow(E/E;\ff_p)$.

In contrast, "non-pure" motives behave differently. For example, the idempotent, corresponding to the projector
$\pi_E:\dmkF{\ff_p}\row\dmELF{E}{k}{\ff_p}$ is mapped via $\psi_E$ to the unit object of the category $\dmELF{E}{E}{\ff_p}$.
On the other hand, its restriction to $\kbar$ is trivial, since $\whii_Q|_{\kbar}=0$, for any non-empty $Q$
(note that $E\neq\ov{E}$, since $k$ is {\it flexible}, so the respective directed system contains non-empty $Q$'s).
Consequently, all the subcategories $\dmELF{E}{k}{\ff_p}$, for all
finitely-generated $E/k$, are killed by the restriction to the algebraic closure functor $\dmkF{\ff_p}\row\dmEF{\kbar}{\ff_p}$.
There are {\it geometric} examples as well: the motive $\wt{M}_{\alpha}$ of Section \ref{Sect-coh-point} vanishes over $\kbar$, but
is non-zero in the {\it isotropic motivic category}.
}
\Red
\end{remark}

\begin{remark}
\label{Chow-notgen-T}
{\rm
Although, the {\it numerical} Chow motivic category $Chow_{Num}(k;\ff_p)$ is a sub-quotient of a topological category
(that is, of the category of graded $\ff_p$-vector spaces), it is more interesting. In particular, it is not generated by a single object (Tate-motive). This is reflected by the absence of the {\it Kunneth formula}. Namely, for smooth projective $X$ and $Y$, the product map
$$
\Ch_{Num}(X)\otimes\Ch_{Num}(Y)\row\Ch_{Num}(X\times Y)
$$
is not an isomorphism ($\Leftrightarrow$ not surjective), in general.
}
\Red
\end{remark}

Our aim is to prove the following result.

\begin{theorem}
\label{thm-conj5-1-2}
The Conjecture \ref{main-conj} is true in the following cases:
\begin{itemize}
\item[$(1)$] $\ddim(X)\leq 5$;
\item[$(2)$] $\Ch^1$;
\item[$(3)$] $\Ch_m$, for $m\leq 2$.
\end{itemize}
\end{theorem}

Item $(2)$ will be proven in Proposition \ref{A12}, item $(3)$ follows from Corollary \ref{A14} and Propositions \ref{A17} and
\ref{A22}, and for the item $(1)$ we need to add Proposition \ref{A23}.

\begin{corollary}
\label{fin-dim-locChow}
In the situation of Theorem \ref{thm-conj5-1-2}, the local Chow groups $\Ch^*_{k/k}(X)$ are finite-dimensional $\ff_p$-vector spaces.
\end{corollary}

This contrasts with the {\it global} situation, where Chow groups of varieties are often huge.

We have the following Chow-motivic questions of increasing strength related to Question \ref{conserv-main-q}.

\begin{question}
\label{conserv-Chow-q}
\begin{itemize}
\item[$(1)$] Is it true that any $U\in Ob(Chow(k;\ff_p))$ which vanishes in the (complex) topological realization is zero?
\item[$(2)$] Is it true that any $f\in\op{End}_{Chow(k;\ff_p)}(V)$ which vanishes in the (complex) topological
realization is nilpotent?
\end{itemize}
\end{question}

As we saw, by Conjecture \ref{main-conj}, for a flexible
field, the triviality of the topological realization should imply the triviality of all {\it local realizations}.

\begin{remark}
\label{strong-conserv-fail}
{\rm
Note, that the stronger variant of Question \ref{conserv-Chow-q}(2) fails. Namely, C.Soule and C.Voisin produced an example
of a class $c\in\Ch^3(X)$, for some smooth projective $X$, such that $c$ vanishes in $\HH_{Top}^6(X;\ff_p)$, but $c$ is not
smash-nilpotent (that is, $c^{\times r}\neq 0\in\Ch^{3r}(X^{\times r})$, for any $r$) - \cite[Theorem 5]{SoVo}.

If we are interested, instead, only in the triviality of local realizations, it is sufficient to take the image $\ov{x}$ of any torsion
class $x$ from $\CH^1(X)$, whose topological realization in $\HH_{Top}^2(X;\ff_p)$ is non-trivial. Then, since the degree pairing is defined
integrally, $\ov{x}|_E\nump 0$, for any $E/k$. Thus all local realizations of $\ov{x}$ are trivial by Theorem \ref{thm-conj5-1-2}(2).
At the same time, its topological realization is non-trivial, and so, not smash-nilpotent. Hence,
$\ov{x}$ is not smash-nilpotent either.
}
\Red
\end{remark}

A weaker variant of Question \ref{conserv-Chow-q} with "topological realization"
replaced by the "restriction to the algebraic closure" is a safer bet.
In this form, the question (2) becomes the {\it Rost Nilpotence Conjecture}.

\subsection{The proof of the Main Theorem}

\subsubsection{Divisors and zero-cycles}

We start the proof of Theorem \ref{thm-conj5-1-2} with the case of divisors (item $(2)$), which is, actually, the base for the whole technique.

\begin{proposition}
\label{A12}
Let $k$ be a flexible field and $u\in\Ch^1(X)$ be $\nump 0$. Then $u=0\in\Ch_{k/k}^1(X)$.
\end{proposition}

\begin{proof}
Adding to an effective divisor representing $u$ a $p$-multiple of a very ample divisor,
we may assume that $u$ is represented by a very ample divisor $D$.
Let $(\pp^n)^{\vee}=\op{Proj}|D|$ be the projective linear system of $D$. This defines the embedding of $X$ into $\pp^n$. From
Statement \ref{alg-geom-Lef},
the embedding $\iota:D_{\eta}\row X$ of the generic representative of our linear system into $X$ induces a surjective map $\iota^*:\Ch_1(X)\twoheadrightarrow\Ch_0(D_{\eta})$.
But since $D\nump 0$, the degree of the zero-cycle $\iota^*(v)$ is zero, for any $v\in\Ch_1(X)$.
Hence, $D_{\eta}$ is anisotropic. Thus, over $k((\pp^n)^{\vee})$, our class $u$ is represented by the class of an
anisotropic divisor $D_{\eta}$.
Since $k$ is flexible, it is represented by an anisotropic class already over $k$ (by Proposition \ref{flex-main-prop}).
\Qed
\end{proof}

Since the theory $\Ch^*_{k/k}$ has the structure of push-forwards and pull-backs, we obtain:

\begin{corollary}
\label{A14}
Let $k$ be a flexible field.
The projection $\Ch^*\twoheadrightarrow\Ch_{k/k}^*$ passes through $\Ch^*_{alg}=\Ch^*_{(1)}$.
\end{corollary}

\begin{proof}
A class $u\in\Ch^*(X)$ is algebraically equivalent to zero, if it can be presented as $f_*(y\cdot g^*(v))$, where
$X\stackrel{f}{\llow}X\times C\stackrel{g}{\lrow}C$ are natural projections, $C$ is a smooth projective curve,
$y\in\Ch^*(X\times C)$ and $v\in\Ch_0(C)$ is a zero cycle of degree zero. Since $v\nump 0$, Proposition \ref{A12}
implies that $v=0\in\Ch_{k/k}^*(C)$, and so, $u=0\in\Ch_{k/k}^*(X)$.
\Qed
\end{proof}

Since any zero-cycle on a smooth projective variety $X$ is a push-forward of some zero-cycle from a curve, we also get the case of zero-cycles.

\begin{corollary}
\label{A13N}
Let $k$ be a flexible field and $u\in\Ch_0(X)$ be $\nump 0$. Then $u=0\in\Ch_{k/k;0}(X)$.
\end{corollary}

Our general strategy of proving that the class $u\nump 0$ is anisotropic will be to find an appropriate blow-up
$\pi:\wt{X}\row X$, so that $\pi^*u$ may be represented by a cycle supported on a smooth connected divisor $Z\subset\wt{X}$,
which is $\nump 0$ already on $Z$. Then we use induction on the dimension of $X$ and the fact that $u=\pi_*\pi^*u$.
In order to achieve this, we will need first to present $u$ by the class of a smooth connected subvariety $S$,
and make the appropriate characteristic classes of it $\nump 0$ on $X$.

\subsubsection{$3$-folds and $1$-cycles}

We will be moving up the dimension of varieties.
The above statements settle the case of curves and surfaces. Our 
next aim are 3-folds, where only the case of $1$-cycles remains open.

\begin{proposition}
\label{A15}
Let $k$ be a flexible field and $X$ be a smooth projective variety over $k$ of dimension $3$.
Let $u\in\Ch_1(X)$ be $\nump 0$. Then $u=0\in\Ch_{k/k;1}(X)$.
\end{proposition}

\begin{proof} We will show that there is a blow-up $\pi:\wt{X}\row X$ such that $\pi^*(u)$ is represented by the class of a smooth
anisotropic curve on $\wt{X}$. Since $\pi_*\pi^*(u)=u$, this will show that the class $u$ is anisotropic.
Here, as in many statements below, we will be gradually reducing
a general case to the one having better and better special properties. 

\begin{lem}
\label{A15-L2}
We may assume that $u$ is represented by a class of a smooth curve $S$ on $X$,
and moreover, $\ddeg(c_1(N_{S\subset X})\cdot [S])=0\,\,(\,mod\, p)$.
\end{lem}

\begin{proof}
By Corollary \ref{AA2}, after some blow-up, we may represent $u$ by the class of a smooth curve $S$ on $X$ (we keep the same name for
the variety).
Note, that $c_1(N_{S\subset X})=c_1(T_X)+c_1(-T_S)$ and $\ddeg(c_1(T_X)\cdot [S])=0$, since $[S]\nump 0$ and $c_1(T_X)$ is a
class defined on $X$.
Now we need to treat separately $p=2$ and odd primes. Such separate treatment of different primes is the feature which we will see repeatedly below.

\noindent
${\mathbf{\un{(p=2)}}}$  Our degree is equal to the $\ddeg(c_1(-T_S)\cdot [S])=\ddeg(P_1([S]))$, where $P_1$ is the homological
Steenrod operation $Sq^2$. But $\ddeg(P_1([S]))=\ddeg(P_1(\eps_*[S]))$, where $\eps:S\row\op{Spec}(k)$ is the projection, and $\eps_*[S]=0$.

\noindent
${\mathbf{\un{(p\neq 2)}}}$ Let $[Z]$ be a smooth very ample divisor on $S$ representing $\frac{1}{2}c_1(-T_S)$ in $\Ch_0(S)$.
Let $\pi:\wt{X}\row X$ be the blow-up of $X$ at $Z$. Let $\wt{S}$ be the proper pre-image of $S$ in $\wt{X}$.
Then by \cite[Thm 6.7]{Fu}, we have:
$$
\pi^*([S])=[\wt{S}]+[\pp^1_Z],
$$
where $[\pp^1_Z]$ is the class supported on the special divisor $\pp^2_Z$.
And so, $\pi^*([S])$ is represented by the class of a smooth curve $S'=\wt{S}\coprod\pp^1_Z$ (we can always choose the curve $\pp^1_Z$ not
intersecting $\wt{S}$).
Note, that $\wt{S}\cong S$. Then
\begin{equation*}
\begin{split}
&\ddeg(c_1(-T_{S'})\cdot [S'])=\ddeg(c_1(-T_{\wt{S}})\cdot [\wt{S}])+
\ddeg(c_1(-T_{\pp^1_Z})\cdot [\pp^1_Z])=\\
&\ddeg(c_1(-T_S)\cdot [S])+\ddeg([Z])\cdot (-2)=0.
\end{split}
\end{equation*}
\Qed
\end{proof}

Now, we can make our class to be, in addition, supported on
some smooth surface.

\begin{lem}
\label{A15-L3}
We may assume that $u$ is represented by the class of a smooth (possibly, disconnected) curve $S$ which is contained in some smooth
(possibly, disconnected) surface $E$ on $X$, and the curve $S$ satisfies also: $\ddeg(c_1(N_{S\subset X})\cdot [S])=0$.
\end{lem}

\begin{proof}
By Lemma \ref{A15-L2}, we can assume that $u=[S]$, where $S\subset X$ is a smooth curve, and moreover,
$\ddeg(c_1(N_{S\subset X})\cdot [S])=0$.
Consider $\wt{X}=\op{Bl}_S(X)$ with the
projection $\pi:\wt{X}\row X$. Then $\pi^*(u)$ is supported on $E=\pp_S(N_{S\subset X})$, which is a smooth surface.
More precisely, it is represented by $\rho+\eps^*(c_1(N_{S\subset X}))$, where $\eps:E\row S$ is the natural projection
and $\rho=c_1(O(1))=-[E]$.
By adding a $p$-multiple of a very ample divisor, we can assume (by Statement \ref{LS-Sbs}) that $\pi^*(u)$
is represented by a very ample divisor on every
component of $E$, and so, is represented by a smooth curve $S'$ on $E$. Moreover, $c_1(N_{S'\subset\wt{X}})$ is
the restriction of $\eps^*(c_1(N_{S\subset X}))$ to $S'$. Hence, $\ddeg(c_1(N_{S'\subset\wt{X}})\cdot [S'])=
\ddeg(c_1(N_{S\subset X})\cdot\eps_*([S']))=\ddeg(c_1(N_{S\subset X})\cdot [S])=0$.
\phantom{a}\hspace{5mm}\Qed
\end{proof}

In view of Corollary \ref{LS-r2-nump-div}, it remains to make our curve and surface connected.

\begin{lem}
\label{A15-L4}
We may assume that $u$ is represented by the class of a smooth connected curve contained in a smooth connected divisor $E$ on $X$,
and the curve $S$ satisfies also: $\ddeg(c_1(N_{S\subset X})\cdot [S])=0$.
\end{lem}

\begin{proof}
By Lemma \ref{A15-L3} we may assume that $u=[S]$, where $S$ is smooth and $S\subset E'\subset X$, where $E'$ is a (possibly, disconnected)
smooth surface and $\ddeg(c_1(-T_S)\cdot [S])=\ddeg(c_1(N_{S\subset X})\cdot [S])=0$.
By Statement \ref{LS-div-conn} applied to the divisor $E'$ considered as a single component,
over some purely transcendental extension of $k$, there is an irreducible divisor $E''$ on $X$,
containing $S$ and smooth outside an anisotropic subset.
Since $k$ is flexible, we may assume that the divisor $E''$ of $X$ is defined already over $k$.
Let $\pi:\wt{X}\row X$ be the embedded desingularization of $E''$.
Let $\wt{E}$ be the proper pre-image of $E''$ and $\wt{Y}$ be the proper pre-image of $S$. Then $\pi^*(u)$ is equal to $[\wt{S}]$ plus
some classes supported on the special divisors of our blow-up. But these special divisors are anisotropic (since the singularities were).
Hence, modulo anisotropic classes, $\pi^*(u)$ is equal to the class $[\wt{S}]$ supported on a smooth connected surface $\wt{E}$.
And since $\wt{S}\cong S$, we have $\ddeg(c_1(-T_{\wt{S}})\cdot [\wt{S}])=\ddeg(c_1(-T_S)\cdot [S])=0$.
Our class is a divisor on $\wt{E}$. Adding a $p$-multiple of an appropriate very ample divisor,
we may assume that our divisor is very ample (by Statement \ref{LS-Sbs}),
and so, is represented by a smooth connected curve on $\wt{E}$. Note, that this procedure doesn't change the
$\ddeg(c_1(-T_S)\cdot [S])=0\in\ff_p$.
\Qed
\end{proof}

Proposition \ref{A15} follows now from Corollary \ref{LS-r2-nump-div} and flexibility of $k$.
\Qed
\end{proof}

Now, the case of varieties of dimension $\leq 3$ is settled,
which, due to the presence of push-forwards and pull-backs, 
implies that isotropic Chow groups factor through the second theory of higher type. 

\begin{proposition}
\label{A16}
Let $k$ be a flexible field. The projection $\Ch^*\twoheadrightarrow\Ch^*_{k/k}$ factors through $\Ch^*_{(2)}$.
\end{proposition}

\begin{proof}
A class $u\in\Ch^*(X)$ is $=0\in\Ch^*_{(2)}$, if it can be presented as $f_*(y\cdot g^*(v))$, where
$X\stackrel{f}{\llow}X\times Q\stackrel{g}{\lrow}Q$ are natural projections, $Q$ is a smooth projective variety of dimension $3$,
$y\in\Ch^*(X\times Q)$ and $v\in\Ch^*(Q)$ is $\nump 0$. Then, by Proposition \ref{A12}, Corollary \ref{A13N} and Proposition \ref{A15},
$v=0\in\Ch_{k/k}^*(Q)$, and so, $u=0\in\Ch_{k/k}^*(X)$.
\Qed
\end{proof}

The case of surfaces and 3-folds permits to start induction and deal with $1$-cycles on a variety of an arbitrary dimension.

\begin{proposition}
\label{A17}
Let $k$ be a flexible field and $X$ be a smooth projective variety over $k$.
Let $u\in\Ch_1(X)$ be $\nump 0$. Then $u=0\in\Ch_{k/k;1}(X)$.
\end{proposition}

\begin{proof}
Induction on $n=\ddim(X)$.

\noindent
{\bf (base)} The case $n=1$ is trivial. The cases $n=2$ and $n=3$ follow from Propositions \ref{A12} and \ref{A15}, respectively.

\noindent
{\bf (step ${\mathbf (n-1)\row (n)}$)} We may assume that $n>3$.

First of all, we need to make our class supported
on a smooth divisor.

\begin{lem}
\label{A17-L2}
We may assume that $u$ is represented by a class supported on some smooth
(possibly, disconnected) divisor $Z$ on $X$.
\end{lem}

\begin{proof}
By Corollary \ref{AA2}, we may assume that $u$ is represented by the class of a smooth curve $S$ on $X$.
Consider the blow-up $\pi:\wt{X}=\op{Bl}_S(X)\lrow X$. Then $\pi^*(u)$ is supported on the special divisor
$E=\pp_S(N_{S\subset X})$ which is smooth.
\Qed
\end{proof}

Now we can make the supporting (smooth) divisor connected and, 
in addition, can make all co-dimension 1 classes on it restrictions of some classes from $X$, at least, modulo anisotropic classes. 

\begin{lem}
\label{A17-L3}
We may assume that $u$ is represented by a class supported on some smooth
connected divisor $Z$ on $X$ such that the restriction $\Ch^1_{k/k}(X)\twoheadrightarrow\Ch^1_{k/k}(Z)$
is surjective.
\end{lem}

\begin{proof}
By Lemma \ref{A17-L2}, we may assume that $u$ is represented by the class of a curve $S$ contained in a smooth divisor $Z$.
By Statement \ref{LS-div-conn} and flexibility of $k$, $S$ is contained in some irreducible divisor $Z'$,
smooth outside an anisotropic subset, and such that the restriction
$\Ch^1_{k/k}(X)\twoheadrightarrow\Ch^1_{k/k}(Z'\backslash S)$ is surjective.
Since $\ddim(Z')-1>1=\ddim(S)$, we also get the surjection $\Ch^1_{k/k}(X)\twoheadrightarrow\Ch^1_{k/k}(Z')$.

Let $\pi:\wt{X}\row X$ be the embedded resolution of singularities of $Z'$. Since the singularities of $Z'$ were anisotropic, so will be the
special divisors of $\wt{X}$. Let $\wt{Z}'$ be the proper pre-image of $Z'$ and $\wt{S}$ be the proper pre-image of $S$.
Then $\Ch^*_{k/k}(Z')=\Ch^*_{k/k}(\wt{Z}')$, and so,
the map $f^*:\Ch^1_{k/k}(X)\twoheadrightarrow\Ch^1_{k/k}(\wt{Z}')$, induced by the natural projection $f:\wt{Z}'\row X$, is surjective.
The image of $\pi^*(u)$ in $\Ch^*_{k/k}(\wt{X})=\Ch^*_{k/k}(X)$
is represented by the class of $\wt{S}$ supported on $\wt{Z}'$.
\Qed
\end{proof}

Since the restriction, $j^*:\Ch^1_{k/k}(X)\twoheadrightarrow\Ch^1_{k/k}(Z)$ is surjective, $u=j_*(u')$, for some class
$u'\in\Ch_{k/k;1}(Z)$, and $u\nump 0$ on $X$, we get that $u'\nump 0$ on $Z$.
As $Z$ is smooth connected projective variety of dimension $n-1$,
by the inductive assumption, $u'=0\in\Ch_{k/k;1}(Z)$. Then the class $u$ is equal to $0\in\Ch_{k/k;1}(X)$ as well.
Proposition \ref{A17} is proven.
\Qed
\end{proof}

\subsubsection{$4$-folds}

The next target is 4-folds, where only the case of co-dimension $2$ cycles is left.

\begin{proposition}
\label{A18}
Let $k$ be a flexible field and $X$ be a smooth projective $k$-variety of dimension $4$.
If $u\in\Ch^2(X)$ is $\nump 0$, then $u=0\in\Ch^2_{k/k}(X)$.
\end{proposition}

\begin{proof}
Our strategy will be to find an appropriate blow-up $\pi:\wt{X}\row X$ such that $\pi^*(u)$ is supported on some smooth connected
divisor $Z$ and is $\nump 0$ on it.

By Corollary \ref{AA2}, we may assume that $u$ is represented by the class of a union of smooth complete intersections of very ample divisors with components meeting transversally.
In such a situation (of transversal smooth components), let us denote as $c_1^2(N_{S\subset X})\cdot [S]$ the sum
$\sum_ic_1^2(N_{S_i\subset X})\cdot [S_i]$, and similar for other characteristic classes.

The case of a prime $2$ requires certain preparations to be made
still at this level of transversal complete intersections, before passing to a single smooth subvariety (to be done in the next step).

\begin{lem}
\label{A18-L2}
We may assume that $u$ is represented by the class of $\cup_i S_i$, where each $S_i$ is a complete intersection of very ample divisors,
with components meeting transversally. For $p=2$, we may moreover assume that $\ddeg(c_1^2(N_{S\subset X})\cdot [S])=0$.
\end{lem}

\begin{proof}
Let $p=2$ and $[S]=\sum_i x_iy_i$, where $x_i$ and $y_i$ are classes of very ample divisors. Then, in $\Ch^2$, we can substitute
this presentation by
$$
[S']=\sum_i(x_iy_i+x_i(x_i+y_i)+y_i(x_i+y_i)+(x_i+y_i)(x_i+y_i)),
$$
where all the divisor classes involved are very ample, and so, components can be made transversal.
Then $c_1^2(N_{S'\subset X})\cdot [S']=\sum_i(x_iy_i(x_i+y_i)^2+x_i(x_i+y_i)y_i^2+y_i(x_i+y_i)x_i^2+0)=0$.
\Qed
\end{proof}

For a surface $S=\cup_i S_i$ with smooth transversal components, let us denote $\Lambda^2[S]:=\sum_{\{i,j\}}[S_i]\cdot[S_j]$
in $\Ch_0(X)$, where the sum is over all 2-element subsets of the set of components.

The following result permits to combine our transversal components into a single smooth connected surface. Moreover, there is some control over the characteristic classes of the surface obtained
this way.

\begin{lem}
\label{A18-L3}
Let $S=\cup_i S_i$ be a surface on $X$ with smooth transversal components. Then there exists a blow-up $\eta:\hat{X}\row X$
such that $\eta^*([S])$ is represented by the class of a smooth connected surface $\hat{S}$
contained in a smooth connected divisor $\hat{Z}$, and
$\ddeg(c_1^2(N_{\hat{S}\subset\hat{X}})\cdot[\hat{S}])=\ddeg(c_1^2(N_{S\subset X})\cdot[S]-2\Lambda^2([S]))$.
If, moreover, $[S]\nump 0$, then this degree is equal to the $\ddeg\left(\left(c_1^2+c_2\right)(N_{S\subset X})\cdot[S]\right)$.
\end{lem}

\begin{proof}
Let $u=[\cup_iS_i]=\sum_iu_i$ be the class of $[S]$.
Let $\pi:\wt{X}\row X$ be the blow-up of $X$ in all components $S_i$. Let $E_i$ be the respective components of the special divisor
and $\rho_i=c_1(O(1)_i)=[-E_i]$. Then, by \cite[Prop. 5.27]{so2}, $\wt{u}_i=\pi^*([S_i])=[E_i]\cdot (c_1(N_{S_i\subset X})+\rho_i)$
is supported on $E_i$ and may be represented by a smooth surface $\wt{S}_i$. Note, that
$c_1(N_{\wt{S}_i\subset\wt{X}})=\pi^*(c_1(N_{S_i\subset X}))$, and so,
$\ddeg(c_1^2(N_{\wt{S}\subset\wt{X}})\cdot[\wt{S}])=\ddeg(c_1^2(N_{S\subset X})\cdot[S])$.

Let $T_{i,j}=E_i\cap E_j$ be
the intersection of the components of the special divisor, and $t_{i,j}=[T_{i,j}]$. We have:
$\ddeg(\wt{u}_i\cdot\wt{u}_j)=\ddeg(u_i\cdot u_j)$, while $\ddeg(\wt{u}_i\cdot t_{i,j})=\ddeg(-[E_i]^3\cdot [E_j])=0$
and $\ddeg(\wt{u}_i\cdot t_{j,k})=0$, for any $i\not\in\{j,k\}$, as well. Finally,
$\ddeg(t_{i,j}^2)=\ddeg([T_{i,j}]\cdot\rho_i\cdot\rho_j)=\ddeg(u_i\cdot u_j)$, while $t_{i,j}\cdot t_{k,l}=0$, for $\{i,j\}\neq\{k,l\}$.
Now, for all pairs $(i,j)$ of distinct numbers let us choose signs $\eps_{(i,j)}\in\pm 1$ with the condition that
$\eps_{(i,j)}+\eps_{(j,i)}=0$. Let us substitute classes $\wt{u}_i$ by $\wt{u}'_i:=\wt{u}_i+\sum_{j\neq i}\eps_{(i,j)}t_{i,j}$.
Note, that the sum $\sum_i\wt{u}'_i$ is still equal to the $\sum_i\wt{u}_i=\pi^*(u)$.
On the other hand,
$$
\ddeg(\wt{u}'_i\cdot\wt{u}'_j)=\ddeg\left(\bigl(\wt{u}_i+{\textstyle{\sum_{k\neq i}}}\eps_{(i,k)}t_{i,k}\bigr)\cdot
\bigl(\wt{u}_j+{\textstyle{\sum_{l\neq j}}}\eps_{(j,l)}t_{j,l}\bigr)\right)=\ddeg(u_i\cdot u_j)-\ddeg(t_{i,j}\cdot t_{j,i})=0.
$$
Since these classes $\wt{u}'_i$ are divisors on (connected) $E_i$, we can make them very ample (by adding a $p$-multiple of some
very ample divisor) and so, can move them around and make them smooth connected and transversal to any given subvariety.
Let $\wt{S}'_i$ be the generic representative of the respective linear system $|\wt{u}'_i|$ on $E_i$. It is a
smooth connected surface on $E_i$ representing $\wt{u}'_i$.
We now have a divisor with
strict normal crossings $E=\cup_i E_i$ and a surface $\wt{S}'=\cup_i \wt{S}'_i$ on it, with $\wt{S}'_i$ smooth connected,
transversal to each other and to the other components of $E$.

Since $\ddeg(\wt{u}'_i\cdot\wt{u}'_j)=0$, for $i\neq j$, by the arguments of the proof of Proposition \ref{A12},
all the intersections $\wt{S}'_i\cap\wt{S}'_j$ are anisotropic.
Since $k$ is flexible, we may assume that $\wt{S}'_i$ is defined already over $k$ (by Proposition \ref{flex-main-prop}).
Finally, denoting $\gamma_i=\pi^*(c_1(N_{S_i\subset X}))$,
since $\rho_i\rho_j\gamma_i$, $\rho_i\rho_j^3$, $\rho_i\rho_j\rho_k$ are zero, for distinct $i,j,k$, we have
\begin{equation*}
\begin{split}
&\ddeg(c_1^2(N_{\wt{S}'\subset\wt{X}})\cdot [\wt{S}'])=
\sum_i\ddeg\left(-\rho_i\bigl(\gamma_i+\rho_i+{\textstyle\sum_{j\neq i}}\eps_{(i,j)}\rho_j\bigr)
\bigl(\gamma_i+{\textstyle{\sum_{j\neq i}}}\eps_{(i,j)}\rho_j\bigr)^2\right)=\\
&\sum_i\ddeg(-\rho_i(\gamma_i+\rho_i)\gamma_i^2-{\textstyle{\sum_{j\neq i}}}\rho_i^2\rho_j^2)=
\ddeg(c_1^2(N_{\wt{S}\subset\wt{X}})\cdot [\wt{S}]-2\Lambda^2([S]))=
\ddeg(c_1^2(N_{S\subset X})\cdot [S]-2\Lambda^2([S])).
\end{split}
\end{equation*}
If $[S]\nump 0$, then $\ddeg(2\Lambda^2([S]))=\ddeg([S]\cdot[S]-\sum_i[S_i]\cdot[S_i])=-\ddeg(c_2(N_{S\subset X})\cdot [S])$.
Thus, in this case,
$\ddeg\left(c_1^2(N_{S\subset X})\cdot [S]-2\Lambda^2([S])\right)=\ddeg\left(\left(c_1^2+c_2\right)(N_{S\subset X})\cdot [S]\right)$.

After all, we managed to present our cycle by the class of the union of smooth surfaces transversal to each other and all intersections
anisotropic.
It remains to use Statement \ref{LS-codim2-single}.
\Qed
\end{proof}

In order to apply Corollary \ref{LS-r2-nump-div}, we need to eliminate (numerically) the powers of the 1-st Chern class of the normal bundle of our surface. We will proceed from highest to smallest powers.

\begin{lem}
\label{A18-L5}
We may assume that $u$ is represented by the class of a smooth connected surface $S$ which is contained in a smooth connected
divisor $Z$, and $\ddeg(c_1^2(N_{S\subset X})\cdot [S])=0\in\ff_p$.
\end{lem}

\begin{proof}
By Lemmas \ref{A18-L2} and \ref{A18-L3},
we may assume that $u$ is represented by the class of a smooth connected surface $S$ contained in a smooth connected divisor $Z$, and
(again by the same Lemmas) we already know the case ${\mathbf{\un{(p=2)}}}$.
We need to treat separately $p=3$ and the remaining primes.

\noindent
${\mathbf{\un{(p=3)}}}$ Let $d_1=c_1^2+c_2$. This characteristic class of degree $2$ corresponds to the reduced power operation
$P^1:\Ch^r\row\Ch^{r+2}$. Namely, $P^1([S])=d_1(N_{S\subset X})\cdot [S]$. By Proposition \ref{AB1},
$\ddeg(d_1(N_{S\subset X})\cdot [S])=0\in\ff_3$, since $[S]\nump 0$.
On the other hand, $\ddeg(c_2(N_{S\subset X})\cdot [S])=\ddeg([S]\cdot[S])=0$ (by the same reason).
Hence, $\ddeg(c_1^2(N_{S\subset X})\cdot [S])=0\in\ff_3$.

\noindent
${\mathbf{\un{(p\neq 2,3)}}}$ Let $R$ be a smooth zero-cycle representing the complete intersection (of very ample divisors)
$\frac{1}{2}c_1(N_{S\subset X})\cdot\frac{1}{3}c_1(N_{S\subset X})$ on $S$.
Let $\pi:\wt{X}=\op{Bl}_R(X)\row X$ be the blow-up of $X$ at $R$,
and $E\stackrel{\eps}{\row}S$ be the special divisor of $\pi$, with $\rho=c_1(O(1))=-[E]$.
Then $\pi^*([S])=[\wt{S}]+[F]$, where $\wt{S}$ is the proper transform of $S$, and
$[F]=[E]\cdot(\eps^*(c_1(N_{S\subset X}))+\rho)$.
We have: $c_1(N_{\wt{S}\subset\wt{X}})=2\rho+\pi_S^*(c_1(N_{S\subset X}))$, while $c_1(N_{F\subset\wt{X}})=\eps^*(c_1(N_{S\subset X}))$.
Hence, (taking into account that $\pi(E)$ is zero-dimensional),
\begin{equation*}
\begin{split}
&c_1^2(N_{\wt{S}\subset\wt{X}})\cdot [\wt{S}]+c_1^2(N_{F\subset\wt{X}})\cdot [F]=j_*\pi_S^*(c_1^2(N_{S\subset X})\cdot [S])+
4\rho^2\cdot [\wt{S}]=\\
&j_*\pi_S^*(c_1^2(N_{S\subset X})\cdot [S]-4[R])=j_*\pi_S^*\left({\textstyle{\frac{1}{3}}}c_1^2(N_{S\subset X})\cdot [S]\right),
\end{split}
\end{equation*}
where $S\stackrel{\pi_S}{\llow}\wt{S}\stackrel{j}{\lrow}\wt{X}$ are natural maps.
On the other hand,
$$
[\wt{S}]\cdot [F]=[\wt{S}]\cdot [E]\cdot (\eps^*(c_1(N_{S\subset X}))+\rho)=[\pp^1_R]\cdot\rho=
j_*\pi_S^*\left({\textstyle{\frac{1}{6}}}c_1^2(N_{S\subset X})\cdot [S]\right).
$$
By Lemma \ref{A18-L3}, there exists a blow-up $\mu:\ov{X}\row\wt{X}$, such that $\mu^*\pi^*([S])\in\Ch^2_{k/k}(\ov{X})$
is represented by a smooth connected
surface $\ov{S}$ contained in a smooth connected divisor $\ov{Z}$ with
\begin{equation*}
\begin{split}
&\ddeg(c_1^2(N_{\ov{S}\subset\ov{X}})\cdot [\ov{S}])=\ddeg(c_1^2(N_{\wt{S}\subset\wt{X}})\cdot [\wt{S}])+
\ddeg(c_1^2(N_{F\subset\wt{X}})\cdot [F])-\ddeg(2[\wt{S}]\cdot [F])=\\
&\ddeg\left(\left({\textstyle{\frac{1}{3}}}c_1^2(N_{S\subset X})-{\textstyle{\frac{1}{3}}}c_1^2(N_{S\subset X})\right)\cdot [S]\right)=0.
\end{split}
\end{equation*}
\Qed
\end{proof}

It remains to treat the first power of the 1-st Chern class.

\begin{lem}
\label{A18-L6}
We may assume that $u$ is represented by the class of a smooth connected surface $S$ which is contained in a smooth connected
divisor $Z$, with $c_1^2(N_{S\subset X})\cdot [S]\nump 0$ and $c_1(N_{S\subset X})\cdot [S]\nump 0$.
\end{lem}

\begin{proof}
By Lemma \ref{A18-L5}, we may assume that $u=[S]$, where $S\subset Z\subset X$ are smooth connected, with the needed condition on $c_1^2$.
It remains to terminate $c_1$. We need to treat separately $p=2$ and larger primes.

\noindent
${\bf{\un{(p=2)}}}$ The characteristic class $c_1$ corresponds to the reduced power operation $P^1:\Ch^r\row\Ch^{r+1}$ (modulo $2$).
Since $[S]\nump 0$, by Proposition \ref{AB1}, $c_1(N_{S\subset X})\cdot [S]=P^1([S])\nump 0$.

\noindent
${\bf{\un{(p\neq 2)}}}$ Let $R$ be a smooth connected curve on $S$ representing $\frac{1}{2}c_1(N_{S\subset X})\cdot [S]$.
Let $\pi:\wt{X}=\op{Bl}_R(X)\row X$ be the blow-up at $R$, with the (connected) special divisor
$E\stackrel{\eps}{\row}S$ and $\rho=c_1(O(1))=-[E]$.
Then $\pi^*([S])=[\wt{S}]+[F]$, where $\wt{S}$ is the proper pre-image of $S$, and
$[F]=[E]\cdot(\eps^*(c_1(N_{S\subset X}))+\rho)$ is supported on $E$.  Note, that $\pi_S:\wt{S}\row S$ is an isomorphism.
Then
$$
c_1^m(N_{\wt{S}\subset\wt{X}})\cdot [\wt{S}]=[\wt{S}]\cdot (\pi_S^*(c_1(N_{S\subset X}))+2\rho)^m=0\in\Ch^*(\wt{X}),\,\,\text{for}\,\, m>0,
$$
since $\rho+\pi_S^*(\frac{1}{2}c_1(N_{S\subset X}))=0$ on $\wt{S}$.

Since $c_1^2(N_{S\subset X})\cdot [S]\nump 0$ on $X$, and $R$ is connected, we have that $c_1(N_{S\subset X})\cdot R\nump 0$ on $R$,
which implies that $[E]\cdot\eps^*(c_1(N_{S\subset X}))\nump 0$ on $E$. Hence,
$$
c_1^m(N_{F\subset\wt{X}})\cdot [F]=[E]\cdot (\eps^*(c_1(N_{S\subset X}))+\rho)\cdot\eps^*(c_1^m(N_{S\subset X}))\nump 0
\,\,\text{on}\,\, \wt{X},\,\,\text{for}\,\, m>0.
$$
Finally, $[\wt{S}]\cdot [F]=[\wt{S}]\cdot (-\rho)\cdot (\pi_S^*(c_1(N_{S\subset X}))+\rho)=
[\wt{S}]\cdot\pi_S^*(\frac{1}{2}c_1(N_{S\subset X}))\cdot\pi_S^*(\frac{1}{2}c_1(N_{S\subset X}))\nump 0$ on $\wt{S}$, since
$c_1^2(N_{S\subset X})\cdot [S]\nump 0$ on $S$.

Substituting $F$ by the generic representative of the (very ample) linear system $|F|$ on $E$, by the proof of Proposition \ref{A12},
we may assume that the intersection $\wt{S}\cap F$ is anisotropic. Then, by Statement \ref{LS-codim2-single}, there exists a blow-up
$\mu:\ov{X}\row\wt{X}$ such that $\mu^*\pi^*([S])$ is represented by the class of a smooth connected surface $\ov{S}$ contained in a smooth
connected divisor $\ov{Z}$, and such that $c_1^m(N_{\ov{S}\subset\ov{X}})\cdot [\ov{S}]\nump
\mu^*(c_1^m(N_{\wt{S}\subset\wt{X}})\cdot [\wt{S}]+c_1^m(N_{F\subset\wt{X}})\cdot [F])\nump 0$, for $m\geq 0$.
\Qed
\end{proof}

Now Proposition \ref{A18} follows from Corollary \ref{LS-r2-nump-div} and flexibility of $k$.
\Qed
\end{proof}

\subsubsection{$2$-cycles}

The case of 4-folds is completed. Our next destination is 2-cycles. The main difficulty here is the case of 2-cycles on a 5-fold, which (together with the treated 4-folds) will form a base
of our induction.

\begin{proposition}
\label{A20}
Let $k$ be a flexible field and $X$ be a smooth projective $k$-variety of dimension $5$.
If $u\in\Ch_2(X)$ is $\nump 0$, then $u=0\in\Ch_{k/k;2}(X)$.
\end{proposition}

\begin{proof}
The strategy, as usual, is to find an appropriate blow-up, so that the pull-back of $u$ is supported on some smooth connected
divisor $Z$ and is $\nump 0$ already on $Z$.

We start by presenting $u$ by the class of a disjoint union of smooth complete intersections and eliminating numerically its
$2$-nd (normal) Chern class. 

\begin{lem}
\label{A20-L1}
We may assume that $u$ is represented by the class of a smooth surface $S$ on $X$, with all components complete intersections and
$\ddeg(c_2(N_{S\subset X})\cdot [S])=0\in\ff_p$.
\end{lem}

\begin{proof}
By Corollary \ref{AA2}, we may assume that $u$ is represented by the class of a disjoint union of smooth complete intersections:
$u=[S]=[\coprod_iS_i]$. We need to treat separately the case $p=2$ and that of odd primes.

\noindent
${\mathbf{\un{(p=2)}}}$
The characteristic class $c_2$ corresponds to the reduced power operation $P^2$ (modulo $2$), that is,
$P^2([S])=c_2(N_{S\subset X})\cdot [S]$. But $[S]\nump 0$, and so, by Proposition \ref{AB1}, $P^2([S])\nump 0$ too.
Hence, $\ddeg(c_2(N_{S\subset X})\cdot [S])=0\in\ff_2$.

\noindent
${\mathbf{\un{(p\neq 2)}}}$ If the degrees of $c_2$ of all the components are trivial, there is nothing to prove.
Otherwise, there is a component $S_l$ given by $x_1x_2x_3$ such that $\ddeg(x_1^2x_2^2x_3)=r\neq 0\in\ff_p$.
Let $R$ be a disjoint union of $d$ copies of the curve $x_1^2x_2^2$, where one factor of $x_1$ and $x_2$ here is the same as in $S_l$
and the other one is generic, so that $Q=R\cap S_l$ is given by the $d$ disjoint copies of $x_1^2x_2^2x_3$ and $R$ doesn't meet
other components of $S$. Let $\pi:\wt{X}\row X$ be the blow-up at $R$ with the special divisor $E$ and $\rho=c_1(O(1))=-[E]$.
Then, by \cite[Thm 6.7]{Fu}, $\pi^*([S_l])=[\wt{S}_l]+[V]$, where $\wt{S}_l\stackrel{\pi_S}{\lrow}S_l$ is the proper transform of $S_l$ and
$V=\pp^2_Q$ is a subvariety of the $Q$-fiber of the $\pp^3$-bundle $E\row R$. Here
$\wt{S}_l$ is a complete intersection $(x_1+\rho)(x_2+\rho)x_3$, while $V$ is a complete intersection $-\rho^2\cdot x_3$.
We can move $[Q]$ along $R$ and make $V$ disjoint from $\wt{S}_l$ (it is automatically disjoint from other components).
On $\wt{S}_l$, $x_i\cdot\rho=0$ (since $\pi_S(\wt{S}\cap E)$ is zero-dimensional) and $\rho^2=-\pi_S^*[Q]$, while on $V=\pp^2_Q$, $x_i=0$ and $\rho^2$ is the class of a section
$Q\row\pp^2_Q$.
Hence, the
\begin{equation*}
\begin{split}
&\ddeg(c_2(N_{\wt{S}_l\subset\wt{X}})\cdot [\wt{S}_l])+\ddeg(c_2(N_{V\subset\wt{X}})\cdot [V])=
\ddeg((x_1x_2+x_2x_3+x_3x_1+\rho^2)[\wt{S}_l])+\ddeg(-\rho^2[V])=\\
&\ddeg(c_2(N_{S\subset X})[S])-2\ddeg([Q])=
\ddeg(c_2(N_{S\subset X})[S])-2rd.
\end{split}
\end{equation*}
Since $p\neq 2$, by choosing $d$ appropriately, we can always make the total degree of $c_2(N_{S\subset X})\cdot [S]$ to be zero
(in $\ff_p$), while keeping all the components complete intersections.
\Qed
\end{proof}

Having made the 2-nd (normal) Chern
class of our surface numerically trivial, 
now we will do the same with
every connected component of it. This will make the mentioned
Chern class numerically trivial already on the surface itself 
(not just after the push-forward to $X$).

\begin{lem}
\label{A20-L2}
We may assume that $u$ is represented by the class of a smooth surface $S$ on $X$, and for each component $S_l$ of $S$, we have
$\ddeg(c_2(N_{S_l\subset X})\cdot [S_l])=0\in\ff_p$.
\end{lem}

\begin{proof}
By Lemma \ref{A20-L1}, we may assume that $u$ is represented by the class of a disjoint union of smooth complete intersections:
$u=[S]=[\coprod_lS_l]$ with $\ddeg(c_2(N_{S\subset X})\cdot [S])=0$.
Let $S_l$ be represented by the intersection $x_1x_2x_3$ of very ample divisors.
Let $R_{\{i,j\}}\subset X$ be the complete intersection $x_i^2x_j^2$, where one copy of $x_i$ and $x_j$ is the same as in $S_l$,
but the other copy is generic, so that $R_{\{i,j\}}$ meets $S_l$ at a zero-cycle $x_i^2x_j^2x_m$ (where $\{m\}=\{1,2,3\}\backslash\{i,j\}$),
and doesn't meet other components of $S$. Let $R=\coprod_{\{i,j\}\in\binom{3}{2}}R_{\{i,j\}}$ (we can make components disjoint).
Note, that $[R\cap S_l]$ represents the class $c_2(N_{S_l\subset X})$ on $S_l$.
Let $\pi:\wt{X}\row X$ be the blow-up at $R$, where
$E_{\{i,j\}}\stackrel{\pi_{\{i,j\}}}{\lrow}R_{\{i,j\}}$ are the components of the special divisor,
$\rho_{\{i,j\}}=-[E_{\{i,j\}}]$ and $\rho=\sum_{\{i,j\}}\rho_{\{i,j\}}$.
Let $\wt{S}_l\stackrel{\pi_S}{\lrow}S_l$ be the proper transform of $S_l$.

The Chern roots of $N_{\wt{S}_l\subset\wt{X}}$ are
$x_i+\sum_{j\neq i}\rho_{\{i,j\}},i=1,2,3$.
Moreover, $\wt{S}_l$ is still a complete intersection: $[\wt{S}_l]=\prod_i(x_i+\sum_{j\neq i}\rho_{\{i,j\}})$,
and on $\wt{S}_l$ we have identities: $x_i\cdot\rho_{\{j,k\}}=0$, while $\rho_{\{i,j\}}^2=-\pi_S^*[R_{\{i,j\}}\cap S_l]$ and
$\rho_{\{i,j\}}\cdot\rho_{\{i',j'\}}=0$, for $\{i,j\}\neq\{i',j'\}$. Then, on $\wt{S}_l$,
$$
c_2(N_{\wt{S}_l\subset\wt{X}})=-\pi_S^*[R\cap S_l]+
\pi_S^*c_2(N_{S\subset X})=0.
$$
Let $Q_{\{i,j\}}=S_l\cap R_{\{i,j\}}$.
By \cite[Thm 6.7]{Fu}, $\pi^*([S_l])=[\wt{S}_l]+[V]$, where
$V=\coprod_{\{i,j\}\in\binom{3}{2}}[V_{\{i,j\}}]$,
$V_{\{i,j\}}=\pp^2_{Q_{\{i,j\}}}$ given by $\rho\cdot\pi_{\{i,j\}}^{-1}(Q_{\{i,j\}})$ and contained in the $\pp^3$-bundle
$E_{\{i,j\}}\row R_{\{i,j\}}$ is a complete intersection $-\rho_{\{i,j\}}^2x_m$ ($m$ as above). We can move the class $[Q_{\{i,j\}}]$
along $R_{\{i,j\}}$, and so can make $\pp^2_{Q_{\{i,j\}}}$ disjoint from $\wt{S_l}$ (and it is automatically disjoint from the other components of $S$).
The Chern roots of $N_{V\subset\wt{X}}$ are $-\rho,0,\rho$, and so,
$$
\ddeg(c_2(N_{V\subset\wt{X}})\cdot [V])=-\sum_{\{i,j\}}\ddeg([Q_{\{i,j\}}])=-\ddeg(c_2(N_{S_l\subset X})\cdot [S_l]).
$$
Let $P_{\{i,j\}}$ be the $(p-1)$ (disjoint) copies of $\pp^1_{Q_{\{i,j\}}}$ contained in $\pi_{\{i,j\}}^{-1}(Q_{\{i,j\}})$, but not in
$\pp^2_{Q_{\{i,j\}}}$. Let $P=\coprod_{\{i,j\}}P_{\{i,j\}}$ and $\mu:\ov{X}\row\wt{X}$ be the blow-up at $P$. Let
$G=\coprod_{\{i,j\}}G_{\{i,j\}}$ be the special
divisor of $\mu$, with projections $\mu_{\{i,j\}}:G_{\{i,j\}}\row P_{\{i,j\}}$ and $\alpha=c_1(O(1))=-[G]$.
Let $\ov{V}_{\{i,j\}}$ be the proper transform of $V_{\{i,j\}}$, and
$F_{\{i,j\}}=\mu_{\{i,j\}}^{-1}(\pp^2_{Q_{\{i,j\}}}\cap P_{\{i,j\}})\cdot \alpha$ (which is isomorphic to the disjoint union of
$(p-1)$ copies of $\pp^2_{Q_{\{i,j\}}}$) whose class
is given by $-\alpha\cdot\alpha\cdot\rho_{\{i,j\}}$. Let $F=\coprod_{\{i,j\}}F_{\{i,j\}}$.
By \cite[Thm 6.7]{Fu}, $\mu^*([V_{\{i,j\}}])=[\ov{V}_{\{i,j\}}]+[F_{\{i,j\}}]$.
The Chern roots of
$N_{\ov{V}_{\{i,j\}}\subset\ov{X}}$ are $\alpha-\rho,\alpha,\rho$. Hence,
$$
c_2(N_{\ov{V}_{\{i,j\}}\subset\ov{X}})\cdot [\ov{V}_{\{i,j\}}]=(\alpha^2-\rho^2+\alpha\rho)\cdot [\ov{V}_{\{i,j\}}]=
((p-1)(-1)-1+0)\cdot \mu_V^*q=0,
$$
where $q$ is the class of a section $Q_{\{i,j\}}\row V_{\{i,j\}}=\pp^2_{Q_{\{i,j\}}}$.
The Chern roots of $N_{F_{\{i,j\}}\subset\ov{X}}$ are $-\alpha,\alpha,0$. So, the
$\ddeg(c_2(N_{F\subset\ov{X}})\cdot [F])=\ddeg(-\alpha^2\cdot [F])=-(p-1)\sum_{\{i,j\}}\ddeg([Q_{\{i,j\}}])=
\ddeg(c_2(N_{S_l\subset X})\cdot [S_l])$.

Let us apply the above construction to every component $S_l$ of $S$, and denote the respective objects by the subscript $l$.
Now, after applying $\mu^*\pi^*$, the degree of $c_2(N_{S\subset X})\cdot [S]$ is concentrated in the $F$-components, where
$[F_l]$ is given by $-\alpha_l\cdot\alpha_l\cdot\rho_l$. Then $[F]=\sum_l(-\alpha_l\cdot\alpha_l\cdot\rho_l)$.
But since $\rho_l\cdot\alpha_k=0$, and $\alpha_l\cdot\alpha_k=0$, for $l\neq k$, this can be rewritten as
$-\ov{\alpha}\cdot\ov{\alpha}\cdot\ov{\rho}$, where $\ov{\alpha}=\sum_l\alpha_l$ and $\ov{\rho}=\sum_l\rho_l$.
We can substitute $-\ov{\alpha}$, $\ov{\alpha}$ and $\ov{\rho}$ by very ample divisors and represent $[F]$ by (the class of)
a smooth connected surface with the same $\ddeg(c_2(N_{F\subset\ov{X}})\cdot [F])$, not meeting other components of $S$.
Now, the whole degree of $c_2(N_{S\subset X})\cdot [S]$ is concentrated in a single component $F$. But, by Lemma \ref{A20-L1},
this total degree is zero.

Lemma \ref{A20-L2} is proven.
\Qed
\end{proof}

Now we can finish the proof of Proposition \ref{A20}.
By Lemma \ref{A20-L2}, we may assume that $u$ is represented by the class of a smooth surface $S$ with
$c_2(N_{S\subset X})\nump 0$ on $S$.

Let $\pi:\wt{X}\row X$ be the blow-up at $S$, with the special divisor $E\stackrel{\pi_E}{\lrow}S$ and $\rho=c_1(O(1))=-[E]$.
Let $c_1=\pi_E^*c_1(N_{S\subset X})$. By \cite[Prop. 6.7]{Fu}, $\pi^*([S])=(\rho^2+c_1\rho)\cdot [E]$ in $\Ch^*_{k/k}(\wt{X})$,
since $c_2(N_{S\subset X})\nump 0$ on $S$ (and by Corollary \ref{A13N}).
Being a complete intersection on $E$, this class may be represented by a smooth surface
$\wt{S}$ on $E$. Note also, that $\rho(\rho^2+c_1\rho)\cdot [E]=0$ and that the Chern roots of $N_{\wt{S_i}\subset\wt{X}}$ are
$-\rho,\rho,\rho+c_1$, and so, the $\ddeg(c_2(N_{\wt{S}_i\subset\wt{X}})\cdot [\wt{S}_i])$ is still zero.

By Statement \ref{LS-div-conn} and flexibility of $k$,
there is an irreducible divisor $Z$, containing $\wt{S}$, smooth outside an anisotropic closed
subscheme of $\wt{S}$, and such that the restriction $\Ch^*(\wt{X})\twoheadrightarrow\Ch^*(Z\backslash\wt{S})$ is surjective.
Let $\mu:\ov{X}\row\wt{X}$ be the embedded desingularization of $Z$ and $\wt{S}$ (note, that we may assume that no component of
$\wt{S}$ belongs to the singular locus of $Z$, since this locus is anisotropic). Let $\ov{Z}$ and $\ov{S}$ be the proper pre-image of $Z$
and $\wt{S}$, respectively. Then, $\ov{Z}$ is smooth connected and, modulo anisotropic classes, $\mu^*([\wt{S}])$ is represented by
$[\ov{S}]$ supported on $\ov{Z}$. Since the maps $Z\backslash\wt{S}\llow\ov{Z}\backslash\mu^{-1}(\wt{S})\lrow\ov{Z}\backslash\ov{S}$
induce isomorphisms $\Ch^*_{k/k}(Z\backslash\wt{S})\stackrel{=}{\row}\Ch^*_{k/k}(\ov{Z}\backslash\mu^{-1}(\wt{S}))
\stackrel{=}{\low}\Ch^*_{k/k}(\ov{Z}\backslash\ov{S})$, we obtain that the group $\Ch^2_{k/k}(\ov{Z})$ is generated by the image
of $j^*:\Ch^2_{k/k}(\ov{X})\row\Ch^2_{k/k}(\ov{Z})$ and the classes $[\ov{S}_i]$ of all the connected components of $\ov{S}$.
Since $\ov{S}\nump 0$ on $\ov{X}$, the image of $j^*$ is orthogonal to $[\ov{S}]$ on $\ov{Z}$. Finally, on $\ov{Z}$
\begin{equation*}
\begin{split}
&\ddeg([\ov{S}]\cdot[\ov{S}_i])=\ddeg([\ov{S}_i]\cdot[\ov{S}_i])=\ddeg(c_2(N_{\ov{S}_i\subset\ov{Z}})\cdot [\ov{S}_i])=\\
&\ddeg(c_2(N_{\ov{S}_i\subset\ov{X}})\cdot [\ov{S}_i])=\ddeg(c_2(N_{\wt{S}_i\subset\wt{X}})\cdot [\wt{S}_i])=0,
\end{split}
\end{equation*}
since $c_2(N_{\ov{S}_i\subset\ov{X}})=c_2(N_{\ov{S}_i\subset\ov{Z}})+c_1(N_{\ov{S}_i\subset\ov{Z}})\cdot c_1(N_{\ov{Z}\subset\ov{X}})$,
where
$c_1(N_{\ov{Z}\subset\ov{X}})=\mu^*c_1(N_{E\subset\wt{X}})=-\mu^*\rho\in\Ch^1_{k/k}(\ov{X})$ (since $\ov{S}$ coincides with $\wt{S}$,
$\ov{Z}$ coincides with $Z$ and $\ov{X}$ coincides with $X$ modulo anisotropic subvarieties) and $\rho\cdot [\wt{S}_i]=0\in Ch^*(\wt{X})$.

Hence, $[\ov{S}]\nump 0$ on $\ov{Z}$. By Proposition \ref{A18}, the class $[\ov{S}]$ is represented by an anisotropic subvariety
 on $\ov{Z}$, and so, on $\ov{X}$. Proposition \ref{A20} is proven.
\Qed
\end{proof}

Having treated $2$-cycles on 4 and 5-folds, now the general case
follows by an easy induction.

\begin{proposition}
\label{A22}
Let $k$ be a flexible field and $X$ be a smooth projective $k$-variety.
If $u\in\Ch_2(X)$ is $\nump 0$, then $u=0\in\Ch_{k/k;2}(X)$.
\end{proposition}

\begin{proof}
Induction on $n=\ddim(X)$.
The case $n=2$ is trivial, while the cases $n=3,4,5$ are covered by Propositions \ref{A12}, \ref{A18} and \ref{A20}, respectively.
This gives the base of induction.

\noindent
{\bf \un{(step)}} Let $\ddim(X)>5$. By Corollary \ref{AA2}, we may assume that $u$ is represented by a class of a smooth
(possibly, disconnected) surface $S$ on $X$. Let $\pi:\wt{X}\row X$ be the blow-up at $S$, with the special divisor $E$.
Then $\pi^*(u)$ has support on a smooth divisor $E$, and is represented there by a class of some surface $S'$ (possibly, non-smooth).
By Statement \ref{LS-div-conn} and flexibility of $k$,
there is an irreducible divisor $Z$, containing $S'$, smooth outside an anisotropic closed
subscheme of $S'$, and such that the restriction $\Ch^*(\wt{X})\twoheadrightarrow\Ch^*(Z\backslash S')$ is surjective.

Let $\mu:\ov{X}\row\wt{X}$ be the embedded resolution of singularities of $Z$, with $\ov{Z}$ and $\ov{S}$ - the proper transforms
of $\ov{Z}$ and $S'$, respectively. Then $\ov{Z}$ is a smooth connected divisor on $\ov{X}$ and $\mu^*([S'])$ is represented by
$[\ov{S}]\in\Ch^*_{k/k}(\ov{X})$ supported on $\ov{Z}$ (since the remaining ingredients are anisotropic).
Since $\ov{Z}$, respectively, $\ov{S}$,
coincides with $Z$, respectively, $S'$, modulo anisotropic subvarieties, the restriction
$\Ch^2_{k/k}(\ov{X})\twoheadrightarrow\Ch^2_{k/k}(\ov{Z}\backslash\ov{S})=\Ch^2_{k/k}(\ov{Z})$ is surjective
(note, that $\ddim(\ov{Z})\geq 5$). But $[\ov{S}]\nump 0$ on $\ov{X}$, hence, it is $\nump 0$ on $\ov{Z}$ as well.
By inductive assumption, $[\ov{S}]$ is represented by the class of an anisotropic surface on $\ov{Z}$, and so, on $\ov{X}$.
Induction step and Proposition \ref{A22} are proven.
\Qed
\end{proof}

\subsubsection{Co-dimension $2$ cycles on a $5$-fold}

The last remaining case of Theorem 
\ref{thm-conj5-1-2} is that of the co-dimension $2$ cycles 
on a 5-fold.
This is, by far, the hardest one, which will require various new tools and extensive computations.

\begin{proposition}
\label{A23}
Let $k$ be a flexible field and $X$ be a smooth projective variety of dimension $5$. If
$u\in\Ch^2(X)$ is $\nump 0$, then $u=0\in\Ch^2_{k/k}(X)$.
\end{proposition}

\begin{proof}
By Corollary \ref{AA2}, we may assume that $u$ is represented by the class $[S]=\sum_ix_iy_i$, where $x_i,y_i$ are classes of very ample
divisors. In particular, all the components of $S$ are smooth and transversal to each other.

We start by eliminating (numerically) certain zero-dimensional characteristic classes of $S$. This needs to be done still 
at the level of the union of complete intersections (before passing to a single component). In the case of a prime $2$, we
also need to make numerically trivial the square of the 
$1$-st Chern class of $S$ at this stage.

\begin{lem}
\label{A23-L1}
We may assume that $u$ is represented by the class $[S]$, where components of $S$ are smooth complete intersections
transversal to each other, with
$c_1^3(N_{S\subset X})\cdot [S]\nump 0$ and $c_1c_2(N_{S\subset X})\cdot [S]\nump 0$ on $X$. For $p=2$, we may assume, in addition,
that $c_1^2(N_{S\subset X})\cdot [S]\nump 0$ on $X$.
\end{lem}

\begin{proof}
We need to treat separately the case $p=2$ and that of odd primes.

\noindent
${\mathbf{\un{(p=2)}}}$ Replace $[S]=\sum_i x_iy_i$ by
$[S']=\sum_i\left(x_iy_i+(x_i+y_i)x_i+(x_i+y_i)y_i+(x_i+y_i)(x_i+y_i)\right)$.
Then
\begin{equation*}
c_1^3(N_{S'\subset X})\cdot [S']=\sum_i(x_iy_i(x_i+y_i)^3+(x_i+y_i)x_iy_i^3+(x_i+y_i)y_ix_i^3)=0\in\Ch^5(X).
\end{equation*}
On the other hand, $(c_1^3+c_1c_2)(N_{S'\subset X})\cdot [S']=P^2P^1([S'])$, where $P^l$ is the reduced power operation (modulo $2$).
By Proposition \ref{AB1}, since $[S']\nump 0$ on $X$, so is $P^2P^1([S'])$. Thus, $c_1c_2(N_{S'\subset X})\cdot [S']\nump 0$ on $X$.
Finally,
$$
c_1^2(N_{S'\subset X})\cdot [S']=\sum_i(x_iy_i(x_i+y_i)^2+(x_i+y_i)x_iy_i^2+(x_i+y_i)y_ix_i^2)=0\in\Ch^4(X).
$$

\noindent
${\mathbf{\un{(p\neq 2)}}}$
In the case of an odd prime, we need to complement the above method with blowing certain zero-cycles on $S$. This keeps the result in the form of a union of complete intersections, while
modifying the degrees of the needed zero-dimenional Chern classes. How exactly it does it is described by the following result.

\begin{sublem}
\label{A23-L1-SL1}
Let $\pi:\wt{X}\row X$ be a blow-up at a smooth point of degree $1$ on $S$. Then
$\pi^*([S])$ may be represented by the class of $S'$, where the components of $S'$ are smooth complete intersections
(transversal to each other), and
$\ddeg(c_1c_2(N_{S'\subset\wt{X}})\cdot [S'])=\ddeg(c_1c_2(N_{S\subset X})\cdot [S])-2$, while
$\ddeg(c_1^3(N_{S'\subset\wt{X}})\cdot [S'])=\ddeg(c_1^3(N_{S\subset X})\cdot [S])-8$.
\end{sublem}

\begin{proof}
Clearly, we may assume that $S$ consists of a single smooth complete intersection. Let $E=\pp^4$ be the special divisor of
$\pi$ and $\rho=c_1(O(1))=-[E]$.
By \cite[Thm 6.7]{Fu}, $\pi^*([S])=[\wt{S}]+[F]$, where $\wt{S}$ is the proper transform of $S$ and $F=\pp^3$ is a divisor on $E$ given by
$\rho$. We can make $\wt{S}$ and $F$ transversal. If $[S]=xy$, then $\wt{S}$ is a complete intersection $(x+\rho)(y+\rho)$, and $F$ is a
complete intersection $-\rho\cdot\rho$. On $\wt{S}$ we have:
$\rho\cdot x=\rho\cdot y=0$, while $\rho^3$ is the minus class of a point. Then
\begin{equation*}
\begin{split}
&\ddeg\left(c_1c_2(N_{\wt{S}\subset\wt{X}})\cdot [\wt{S}]+c_1c_2(N_{F\subset\wt{X}})\cdot [F]\right)=\\
&\ddeg([\wt{S}](x+\rho)(y+\rho)(x+y+2\rho))=\ddeg([\wt{S}]\cdot (xy(x+y)+2\rho^3))=\ddeg(c_1c_2(N_{S\subset X})\cdot [S])-2.
\end{split}
\end{equation*}
Similarly,
\begin{equation*}
\begin{split}
&\ddeg\left(c_1^3(N_{\wt{S}\subset\wt{X}})\cdot [\wt{S}]+c_1^3(N_{F\subset\wt{X}})\cdot [F]\right)=\\
&\ddeg([\wt{S}](x+y+2\rho)^3)=\ddeg([\wt{S}]\cdot ((x+y)^3+8\rho^3))=\ddeg(c_1^3(N_{S\subset X})\cdot [S])-8.
\end{split}
\end{equation*}
\Qed
\end{proof}

Denote as $c_1^3$ the $\ddeg(c_1^3(N_{S\subset X})\cdot [S])$ and similar for $c_1c_2$.
For $[S]=\sum_ix_iy_i$, let us denote $[S_{exp}]=\sum_i\left(x_iy_i-(x_i+y_i)x_i-(x_i+y_i)y_i+(x_i+y_i)(x_i+y_i)\right)$,
which represents the same class.
This operation affects the degrees of Chern classes as follows.

\begin{sublem}
\label{A23-L1-SL2}
The substitution of $[S]$ by $[S_{exp}]$ acts on characteristic numbers as follows:
$$
\begin{pmatrix}
(c'_1)^3\\
c'_1c'_2
\end{pmatrix}
=
\begin{pmatrix}
20 & -2\\
5 & 1
\end{pmatrix}
\begin{pmatrix}
c_1^3\\
c_1c_2
\end{pmatrix}.
$$
\end{sublem}

\begin{proof}
It is sufficient to treat the case of a single complete intersection $[S]=xy$. Then
\begin{equation*}
\begin{split}
&{c'_1}^3-c_1^3=\ddeg((x+y)^2(2(x+y))^3-(x+y)x(2x+y)^3-(x+y)y(2y+x)^3)=\\
&\hspace{5mm}\ddeg(xy(x+y)(19(x+y)^2-2xy))=19c_1^3-2c_1c_2,\,\,\,\text{and}\\
&c'_1c'_2-c_1c_2=\ddeg((x+y)^2 2(x+y)^3-(x+y)^2x^2(2x+y)-(x+y)^2y^2(2y+x))=\\
&\hspace{5mm}\ddeg(5xy(x+y)^3)=5c_1^3.
\end{split}
\end{equation*}
\Qed
\end{proof}

Now we can combine both methods.
Substituting $S$ by $S_{exp}$ and blowing up the zero-cycles $\frac{5}{2}c_1^3(N_{S\subset X})$ and $\frac{1}{2}c_1c_2(N_{S\subset X})$
on it (note, that $p\neq 2$), we obtain $[S'']=\pi^*([S])$ such that all the components of $S''$ are still smooth complete intersections
transversal to each other, and by Sublemma \ref{A23-L1-SL1},
$$
\begin{pmatrix}
(c''_1)^3\\
c''_1c''_2
\end{pmatrix}
=
\begin{pmatrix}
0 & -6\\
0 & 0
\end{pmatrix}
\begin{pmatrix}
c_1^3\\
c_1c_2
\end{pmatrix}.
$$
Applying this procedure twice, we obtain $c_1^3=0$ and $c_1c_2=0$.
Lemma \ref{A23-L1} is proven.
\Qed
\end{proof}

The next step is to make $S$ into a single connected component. This will be possible due to the preparations we made
(trivial zero-dimensional Chern classes). 

\begin{lem}
\label{A23-L2}
We may assume that $S$ is smooth connected with $c_1^3(N_{S\subset X})\cdot [S]\nump 0$,
$c_1c_2(N_{S\subset X})\cdot [S]\nump 0$. For $p=2$, we may assume, in addition, that $c_1^2(N_{S\subset X})\cdot [S]\nump 0$
\end{lem}

\begin{proof}
By Lemma \ref{A23-L1} we may assume that $S=\cup_iS_i$ consists of smooth transversal complete intersections and the needed
conditions on the characteristic classes are satisfied.
Let $\pi:\wt{X}\row X$ be the blow-up at all intersections $S_i\cap S_j$, $i\neq j$.
Let $E_{\{i,j\}}\stackrel{\pi_{\{i,j\}}}{\lrow}S_i\cap S_j$ be the respective component of the special divisor, $\rho_{\{i,j\}}=-[E_{\{i,j\}}]$ and $\rho=\sum_{\{i,j\}}\rho_{\{i,j\}}$.
By \cite[Thm 6.7]{Fu}, $\pi^*([S])=[\wt{S}]+[F]+[G]$, where $\wt{S}=\coprod_i\wt{S}_i$ is the proper transform of $S$,
while $F=\coprod_{\{i,j\}}F_{\{i,j\}}$ where
$$
[F_{\{i,j\}}]=[E_{\{i,j\}}]\cdot(\rho_{\{i,j\}}+\pi_{\{i,j\}}^*(c_1(N_{S_i\subset X})+c_1(N_{S_j\subset X})))
\hspace{3mm}\text{and}\hspace{3mm} [G]=(-\rho)\cdot \rho.
$$
Let $S'=\wt{S}\cup F\cup G$. Here $\rho_{\{i,j\}}$ satisfies:
$\rho_{\{i,j\}}^4+\rho_{\{i,j\}}^3(x_i+y_i+x_j+y_j)+x_iy_ix_jy_j=0$, where $[S_k]=x_ky_k$.
Of course, one can get rid of the $G$-term by considering
$$
[\wt{F}_{\{i,j\}}]=[E_{\{i,j\}}]\cdot(2\rho_{\{i,j\}}+\pi_{\{i,j\}}^*(c_1(N_{S_i\subset X})+c_1(N_{S_j\subset X}))).
$$
This will work for odd primes. But for $p=2$, this term really makes a difference.

The following result computes the degrees of characteristic
classes of $\wt{S}$, $F$, $G$ and their intersections in terms
of those of $S$.

\begin{sublem}
\label{A23-L2-SL1}
\begin{itemize}
\item[$(1)$] $\ddeg(c_1c_2(N_{\wt{S}\subset\wt{X}})\cdot [\wt{S}])=2\cdot\ddeg(c_1c_2(N_{S\subset X})\cdot [S])$;
\item[$(2)$] $\ddeg(c_1c_2(N_{F\subset\wt{X}})\cdot [F])=\ddeg(c_1c_2(N_{S\subset X})\cdot [S])$;\,\,\,
$\ddeg(c_1c_2(N_{G\subset\wt{X}})\cdot [G])=0$;
\item[$(3)$] $\ddeg(c_1^3(N_{S'\subset\wt{X}})\cdot [S'])=\ddeg((c_1^3+4c_1c_2)(N_{S\subset X})\cdot [S])$;
\item[$(4)$] $\ddeg(c_1(N_{\wt{S}\subset\wt{X}})\cdot[\wt{S}]\cdot[F])=-3\cdot\ddeg(c_1c_2(N_{S\subset X})\cdot [S])$;
\item[$(5)$] $\ddeg(c_1(N_{F\subset\wt{X}})\cdot[\wt{S}]\cdot[F])=-2\cdot\ddeg(c_1c_2(N_{S\subset X})\cdot [S])$;
\end{itemize}
\end{sublem}

\begin{proof}
Let $\rho_i=\sum_{j\neq i}\rho_{\{i,j\}}$. The Chern roots of $\wt{S_i}$ are $\rho_i+x_i, \rho_i+y_i$, while the Chern roots
of $F_{\{i,j\}}$ are $-\rho_{\{i,j\}}, \rho_{\{i,j\}}+x_i+y_i+x_j+y_j$. Denote: $a_i=x_i+y_i$, $b_i=x_iy_i$, and $a_{\{i,j\}}=a_i+a_j$,
$b_{\{i,j\}}=b_ib_j$.
Using the equation for $\rho_{\{i,j\}}$
and the fact that $\rho_{\{i,j\}}$, multiplied by any monomial of degree $\geq 2$ in $x$ and $y$'s, is zero, we get:
\begin{equation*}
\begin{split}
&\ddeg(c_1c_2(N_{\wt{S}\subset\wt{X}})\cdot[\wt{S}])=\ddeg({
{\sum_i}}(\rho_i+x_i)^2(\rho_i+y_i)^2(2\rho_i+a_i))=\\
&\ddeg({
{\sum_i}}(2\rho_i^5+5\rho_i^4a_i+b_i^2a_i))=
\ddeg(2{
{\sum_i}}b_i^2a_i-{
{\sum_i}}b_ia_i\cdot({
{\sum_j}}b_j))=
2\cdot\ddeg(c_1c_2(N_{S\subset X})\cdot [S]),
\end{split}
\end{equation*}
since $\sum_jx_jy_j=[S]\nump 0$ on $X$. Analogously,
\begin{equation*}
\begin{split}
&\ddeg(c_1c_2(N_{F\subset\wt{X}})\cdot[F])=
\ddeg({
{\sum_{\{i,j\}}}}\rho_{\{i,j\}}^2(\rho_{\{i,j\}}+a_{\{i,j\}})^2a_{\{i,j\}})=
\ddeg({
{\sum_{\{i,j\}}}}\rho_{\{i,j\}}^4a_{\{i,j\}})=\\
&\ddeg({
{\sum_i}}{
{\sum_{j\neq i}}}-b_ib_ja_i)=
\ddeg({
{\sum_i}}b_i^2a_i-({
{\sum_i}}b_ia_i)({
{\sum_j}}b_j))=
\ddeg(c_1c_2(N_{S\subset X})\cdot [S]),
\end{split}
\end{equation*}
again, since $[S]\nump 0$ on $X$. Using the same properties, we obtain:
\begin{equation*}
\begin{split}
&\ddeg(c_1^3(N_{\wt{S}\subset\wt{X}})\cdot[\wt{S}])=\ddeg({
{\sum_i}}(\rho_i+x_i)(\rho_i+y_i)(2\rho_i+a_i)^3)=
\ddeg({
{\sum_i}}(8\rho_i^5+20\rho_i^4a_i+b_ia_i^3))=\\
&\ddeg{
{\sum_i}}(4b_i^2a_i+b_ia_i^3)
-4\ddeg({
{\sum_i}}b_ia_i\cdot{
{\sum_j}}b_j)=
\ddeg((4c_1c_2+c_1^3)(N_{S\subset X})\cdot [S]),\,\,\,\text{and}\\
&\ddeg(c_1^3(N_{F\subset\wt{X}})\cdot[F])=
\ddeg{
{\sum_{\{i,j\}}}}-\rho_{\{i,j\}}(\rho_{\{i,j\}}+a_{\{i,j\}})a_{\{i,j\}}^3=0.
\end{split}
\end{equation*}
Clearly, $\ddeg(c_1^3(N_{G\subset\wt{X}})\cdot [G])=0$ and $\ddeg(c_1c_2(N_{G\subset\wt{X}})\cdot [G])=0$, since
$c_1(N_{G\subset\wt{X}})=0$.
As a result, $\ddeg(c_1^3(N_{S'\subset\wt{X}})\cdot [S'])=\ddeg((c_1^3+4c_1c_2)(N_{S\subset X})\cdot [S])$.
Finally, $\ddeg(c_1(N_{\wt{S}\subset\wt{X}})\cdot[\wt{S}]\cdot[F])=$
\begin{equation*}
\begin{split}
&\ddeg{
{\sum_{\{i,j\}}}}
\big(c_1(N_{\wt{S}_i\subset\wt{X}})[\wt{S}_i]+c_1(N_{\wt{S}_j\subset\wt{X}})[\wt{S}_j]\big)
(-\rho_{\{i,j\}})(\rho_{\{i,j\}}+a_{\{i,j\}})=\\
&\ddeg({
{\sum_{\{i,j\}}}}-\rho_{\{i,j\}}^3(4\rho_{\{i,j\}}+3a_{\{i,j\}})(\rho_{\{i,j\}}+a_{\{i,j\}}))=
\ddeg({
{\sum_{\{i,j\}}}}-\rho_{\{i,j\}}^4(4\rho_{\{i,j\}}+7a_{\{i,j\}}))=\\
&\ddeg{
{\sum_{\{i,j\}}}}(4\rho_{\{i,j\}}^4+7b_{\{i,j\}})a_{\{i,j\}}=
3\ddeg{
{\sum_{\{i,j\}}}}b_{\{i,j\}}a_{\{i,j\}}=-3\ddeg(c_1c_2(N_{S\subset X})\cdot [S]);
\hspace{5mm}\text{and}\hspace{5mm}
\end{split}
\end{equation*}
\begin{equation*}
\begin{split}
&\ddeg(c_1(N_{F\subset\wt{X}})\cdot[\wt{S}]\cdot[F])=
\ddeg{
{\sum_{\{i,j\}}}}
\big((\rho_i+x_i)(\rho_i+y_i)+(\rho_j+x_j)(\rho_j+y_j)\big)
(-\rho_{\{i,j\}})(\rho_{\{i,j\}}+a_{\{i,j\}})a_{\{i,j\}}=\\
&\ddeg{
{\sum_{\{i,j\}}}}(-2\rho_{\{i,j\}}^4a_{\{i,j\}})=-2\ddeg(c_1c_2(N_{S\subset X})\cdot [S]).
\end{split}
\end{equation*}
\Qed
\end{proof}

For the prime $2$ we also need to control the square of the $1$-st Chern class of $S$.

\begin{sublem}
\label{A23-L2-SL1andhalf}
For $p=2$, we have:
$$
c_1^2(N_{\wt{S}\subset\wt{X}})\cdot [\wt{S}]=\pi^*c_1^2(N_{S\subset X})\cdot[S]\nump 0\hspace{3mm}\text{and}\hspace{3mm} c_1^2(N_{F\subset\wt{X}})\cdot [F]=0.
$$
\end{sublem}

\begin{proof}
Using the fact that the centers of the blow-up $\pi$ were $1$-dimensional, we obtain:
\begin{equation*}
\begin{split}
&c_1^2(N_{\wt{S}\subset\wt{X}})\cdot [\wt{S}]=\sum_i(\rho_i+x_i)(\rho_i+y_i)(x_i+y_i)^2=\sum_ix_iy_i(x_i+y_i)^2=
\pi^*c_1^2(N_{S\subset X})\cdot [S];\\
&c_1^2(N_{F\subset\wt{X}})\cdot [F]=\sum_{\{i,j\}}(-\rho_{\{i,j\}})(\rho_{\{i,j\}}+x_i+y_i+x_j+y_j)(x_i+y_i+x_j+y_j)^2=0.
\end{split}
\end{equation*}
\Qed
\end{proof}

We may assume that our class is represented by the class of $S'=\wt{S}\cup F\cup G$, where $\wt{S}$, $F$ and $G$
are smooth (possibly, disconnected)
and transversal to each other (note, that we don't have triple intersections of $S_i$'s),
and $[G]=(-\rho)\cdot\rho$, for some divisor $\rho$.
Moreover, both $\wt{S}$ and $F$ have trivial $c_1^3$ and $c_1c_2$ characteristic numbers and
$\ddeg(c_1(N_{\wt{S}\subset\wt{X}})\cdot[\wt{S}\cap F])=0$ and $\ddeg(c_1(N_{F\subset\wt{X}})\cdot[\wt{S}\cap F])=0$.
Since components of $\wt{S}$, respectively $F$, are disjoint,
by Statement \ref{LS-codim2-single}, there exists a blow-up $\mu:\ov{X}\row\wt{X}$ such that $\mu^*([\wt{S}])$ and $\mu^*([F])$ are
represented by the classes $[\ov{S}]$ and $[\ov{F}]$ of smooth connected transversal subvarieties, such that
$\ov{S}\cap\ov{F}$ is connected, with trivial $c_1^3$ and $c_1c_2$ characteristic numbers for both $\ov{S}$ and $\ov{F}$ and
with $\ddeg(c_1(N_{\ov{S}\subset\ov{X}})\cdot[\ov{S}\cap\ov{F}])=0$ and $\ddeg(c_1(N_{\ov{F}\subset\ov{X}})\cdot[\ov{S}\cap\ov{F}])=0$.
For $p=2$, we have, in addition, $c_1^2(N_{\ov{S}\subset\ov{X}})\cdot[\ov{S}]\nump 0$ and
$c_1^2(N_{\ov{F}\subset\ov{X}})\cdot[\ov{F}]\nump 0$.
As a next step, we will combine $\ov{S}\cup\ov{F}$ into a single component. We start with the following general statement about co-dimension $2$ subvarieties on a variaty of an arbitrary dimension.

\begin{sublem}
\label{A23-L2-SL2}
Let $Y=\cup_iY_i$ be a divisor with strict normal crossings on $X$ such that $Y_i$'s and all the intersections
$Y_{\{i,j\}}=Y_i\cap Y_j$, for $i\neq j$, are connected.
Let $S=\cup_iS_i$ be a union of smooth transversal components, where $S_i\subset Y_i$ are divisors.
Suppose, that $[S_i\cap S_j]\nump 0$ on $Y_{\{i,j\}}$. Then $[S]=[S']$, where $S'=\cup_iS_i'$, with $S'_i\subset Y_i$ smooth connected
and transversal to each other, and all the intersections $S'_i\cap S'_j$ anisotropic.
This procedure doesn't change the characteristic classes of $N_{S\subset X}$ in $\Ch^*(X)$.
\end{sublem}

\begin{proof}
Adding to $S_i$ a $p$-multiple of a very ample divisor on $Y_i$, we may assume that $S_i$ is given by a very ample divisor on $Y_i$
(this doesn't change the characteristic classes (mod $p$)). Let $S'_i$ be the generic representative of the linear system $|S_i|$.
Then $S'_i$ is connected and, by Statement \ref{alg-geom-Lef}, the restriction
$\Ch_2(Y_i\cap Y_j)\twoheadrightarrow\Ch_0(S'_i\cap S'_j)$ is surjective.
This implies that $S'_i\cap S'_j$ is anisotropic, since $[S_i\cap S_j]\nump 0$ on $Y_i\cap Y_j$. Since $[S'_i\row Y_i]$ is cobordant
to $[S_i\row Y_i]$ in $\Omega^*(Y_i)$, all the characteristic classes are preserved.
Finally, since $k$ is flexible, we may assume that $S'$ is defined over $k$.
\Qed
\end{proof}

Using this result we can make $\ov{S}\cup\ov{F}$ into a 
single smooth component with the needed Chern classes 
numerically trivial. This is done with the help of the following
statement, specific to dimension 5.

\begin{sublem}
\label{A23-L2-SL4}
Let $S=S_1\cup S_2$ with $S_1,S_2$ smooth connected transversal to each other and connected $S_{\{1,2\}}=S_1\cap S_2$.
\begin{itemize}
\item[$(1)$]
Suppose, $\ddeg(c_1(N_{S_i\subset X})\cdot [S_{\{1,2\}}])=0$,
$\ddeg(c_1c_2(N_{S_i\subset X})\cdot [S_i])=0$, for $i=1,2$, $\ddeg(c_1^3(N_{S\subset X})\cdot [S])=0$, and for $p=2$,
$c_1^2(N_{S\subset X})\cdot [S]\nump 0$ on $X$. Then there exists a blow-up $\eps:\hat{X}\row X$ such that $\eps^*([S])$
is represented by the class of a smooth connected subvariety $\hat{S}$ with
$\ddeg(c_1c_2(N_{\hat{S}\subset\hat{X}})\cdot [\hat{S}])=0$, $\ddeg(c_1^3(N_{\hat{S}\subset\hat{X}})\cdot [\hat{S}])=0$, and for $p=2$,
$c_1^2(N_{\hat{S}\subset\hat{X}})\cdot [\hat{S}]\nump 0$ on $\hat{X}$.
\item[$(2)$] If $[S]\nump 0$ and $\ddeg(c_1c_2(N_{S_i\subset X})\cdot [S_i])=0$, then
$\ddeg(c_1(N_{S_i\subset X})\cdot[S_{\{1,2\}}])=0$.
\end{itemize}
\end{sublem}

\begin{proof}
(2) Denote $N_i=N_{S_i\subset X}$. Then $\ddeg(c_1(N_1)\cdot [S_{\{1,2\}}])=\ddeg(c_1(N_1)\cdot [S_1]\cdot [S_2])=
\ddeg(c_1(N_1)\cdot [S_1]\cdot [S])-\ddeg(c_1(N_1)\cdot [S_1]\cdot [S_1])=-\ddeg(c_1c_2(N_1)\cdot [S_1])=0$, as $[S]\nump 0$,
and similarly, $\ddeg(c_1(N_2)\cdot [S_{\{1,2\}}])=0$.

(1)
Let $\pi:\wt{X}\row X$ be the blow-up in the components of $S$, with the components $E_i\stackrel{\pi_i}{\lrow}S_i$
of the special divisor, and $\rho_i=-[E_i]$. Note, that $E_1,E_2$ and $E_{\{1,2\}}=E_1\cap E_2$ are connected.
Then $\pi^*([S_i])=[E_i]\cdot(\rho_i+\pi_i^*c_1(N_i))$.
Consider the classes $[E_1](\rho_1-\rho_2+\pi_1^*c_1(N_1))$ and $[E_2] (\rho_1+\rho_2+\pi_2^*c_1(N_2))$.
We may assume that these classes are represented by the classes of smooth connected subvarieties $S'_1$ and $S'_2$, contained
in $E_1$ and $E_2$, respectively, and transversal to each other and to $E_{\{1,2\}}$. Let $S'=S'_1\cup S'_2$. Then $[S']=\pi^*([S])$.
Using the equation $\rho_i^2+\rho_i\pi^*c_1(N_i)+\pi^*c_2(N_i)=0$, on $E_{\{1,2\}}$ we get:
\begin{equation*}
\begin{split}
&[S'_1\cap S'_2]_{E_{\{1,2\}}}=(\rho_1-\rho_2+\pi_{\{1,2\}}^*c_1(N_1))(\rho_1+\rho_2+\pi_{\{1,2\}}^*c_1(N_2))=\\
&\pi_{\{1,2\}}^*(-c_2(N_1)+c_2(N_2)+c_1(N_1)c_1(N_2))+\rho_1\pi_{\{1,2\}}^*c_1(N_2)+\rho_2\pi_{\{1,2\}}^*c_1(N_1)=\\
&\rho_1\pi_{\{1,2\}}^*c_1(N_2)+\rho_2\pi_{\{1,2\}}^*c_1(N_1),
\end{split}
\end{equation*}
where $\pi_{\{1,2\}}:E_{\{1,2\}}\row S_{\{1,2\}}=S_1\cap S_2$ is the projection. Here $S_{\{1,2\}}$ is a smooth connected curve, with
$\ddeg(c_1(N_i)\cdot [S_{\{1,2\}}])=0$. Because $S_{\{1,2\}}$ is connected, this implies that
$[S'_1\cap S'_2]\nump 0$ on $E_{\{1,2\}}$.

By Sublemma \ref{A23-L2-SL2}, we can substitute $S'_i$ by smooth connected $\wt{S}_i$, transversal to each other, with
anisotropic $\wt{S}_1\cap\wt{S}_2$, without changing characteristic classes in $\Ch^*(\wt{X})$. Let $\eta:\ov{X}\row\wt{X}$
be the blow-up in $\wt{S}_1$, $\wt{S}_2$, with the special divisor $V$.
Here $V$ is smooth outside an anisotropic subscheme (the intersection of components), and so, we can treat it as a single component.
$\eta^*([\wt{S}])$ is represented by the class of a subvariety $\ov{S}$ of $V$ whose characteristic classes are
$\eta^*$ of those of $\wt{S}$, and which is smooth outside an anisotropic subscheme.
Then using Statement \ref{LS-div-conn} (with $V$ treated as a single component) and flexibility of $k$, we can find an irreducible
divisor $Z$ on $\ov{X}$, smooth outside an anisotropic subscheme, and containing $\ov{S}$. Let $\mu:\hat{X}\row\ov{X}$ be the embedded
desingularization of $Z$, and $\hat{Z}$, $S''$ be the proper pre-images of $Z$ and $\ov{S}$.
Since the singularities of $Z$ are anisotropic, the map $\mu^*:\Ch^*_{k/k}(\ov{X})\stackrel{=}{\row}\Ch^*_{k/k}(\hat{X})$
is an isomorphism, and $\mu^*([\ov{S}])=[S'']$. Let $\hat{S}$ be a smooth connected variety representing $[S'']$
on the smooth connected divisor $\hat{Z}$. Since $S''$ is smooth outside an anisotropic subscheme,
the characteristic classes of $\hat{S}$ with values in $\Ch^*_{k/k}(\hat{X})$ coincide with $\mu^*$ of those of $\ov{S}$.

It remains to check that the needed characteristic classes of $\hat{S}$ are trivial. Using the fact that $E_{\{1,2\}}$ is a
$\pp^1\times\pp^1$-bundle over a curve $S_{\{1,2\}}$, and the equation $\rho_i^2+\rho_i\pi_i^*c_1(N_i)+\pi_i^*c_2(N_i)=0$ on $E_i$, we get:
\begin{equation*}
\begin{split}
&\ddeg(c_1^3(N_{\hat{S}\subset\hat{X}})\cdot[\hat{S}])=\ddeg(c_1^3(N_{S'\subset\wt{X}})\cdot[S'])=\\
&\ddeg([E_1](\rho_1-\rho_2+\pi_1^*c_1(N_1))(-\rho_2+\pi_1^*c_1(N_1))^3+
[E_2](\rho_2+\rho_1+\pi_2^*c_1(N_2))(\rho_1+\pi_2^*c_1(N_2))^3)=\\
&\ddeg([E_1]\rho_1(\pi_1^*c_1^3(N_1)+3\rho_2^2\pi_1^*c_1(N_1)-\rho_2^3)+
[E_2]\rho_2(\pi_2^*c_1^3(N_2)+3\rho_1^2\pi_2^*c_1(N_2)+\rho_1^3))=\\
&\ddeg(\pi_1^*(c_1^3(N_1)[S_1])\rho_1)+\ddeg(\pi_2^*(c_1^3(N_2)[S_2])\rho_2)-
\ddeg(\pi_{\{1,2\}}^*((2c_1(N_1)+4c_1(N_2))[S_{\{1,2\}}])\rho_1\rho_2)=\\
&\ddeg(c_1^3(N_{S\subset X})\cdot [S])-2\ddeg(c_1(N_1)\cdot [S_{\{1,2\}}])-4\ddeg(c_1(N_2)\cdot [S_{\{1,2\}}])=0,
\end{split}
\end{equation*}
Similarly,
\begin{equation*}
\begin{split}
&\ddeg(c_1c_2(N_{\hat{S}\subset\hat{X}})\cdot[\hat{S}])=\ddeg(c_1c_2(N_{S'\subset\wt{X}})\cdot[S'])=\\
&\ddeg(-[E_1]\rho_1(\rho_1-\rho_2+\pi_1^*c_1(N_1))^2(-\rho_2+\pi_1^*c_1(N_1))-
[E_2]\rho_2(\rho_2+\rho_1+\pi_2^*c_1(N_2))^2(\rho_1+\pi_2^*c_1(N_2)))=\\
&\ddeg(-[E_1]\rho_1(\rho_2^2(\pi_1^*(3c_1(N_1)+c_1(N_2))+2\rho_1)+(\rho_1+\pi_1^*c_1(N_1))^2\pi_1^*c_1(N_1))-\\
&\hspace{11mm}[E_2]\rho_2(\rho_1^2(\pi_2^*(3c_1(N_2)-c_1(N_1))+2\rho_2)+(\rho_2+\pi_2^*c_1(N_2))^2\pi_2^*c_1(N_2)))=\\
&\ddeg([E_{\{1,2\}}]\pi_{\{1,2\}}^*(2c_1(N_1)+4c_1(N_2))\rho_1\rho_2)+
\ddeg([E_1]\rho_2^2(2\rho_1\pi_1^*c_1(N_1)))+\ddeg([E_2]\rho_1^2(2\rho_2\pi_2^*c_1(N_2)))-\\
&\hspace{11mm}\ddeg([E_1]\pi_1^*c_1(N_1)\rho_1(\rho_1+\pi_1^*c_1(N_1))^2)-\ddeg([E_2]\pi_2^*c_1(N_2)\rho_2(\rho_2+\pi_2^*c_1(N_2))^2)=\\
&\ddeg([E_{\{1,2\}}]\cdot 2\pi_{\{1,2\}}^*c_1(N_2)\rho_1\rho_2)+\ddeg([E_1]\cdot\pi_1^*c_1c_2(N_1)(\rho_1+\pi_1^*c_1(N_1)))+\\
&\hspace{11mm}\ddeg([E_2]\cdot\pi_2^*c_1c_2(N_2)(\rho_2+\pi_2^*c_1(N_2)))=
\ddeg(c_1c_2(N_{S\subset X})[S])+2\cdot\ddeg(c_1(N_2)\cdot[S_{\{1,2\}}])=0.
\end{split}
\end{equation*}
Finally, for $p=2$, the fact that $c_1^2(N_{\hat{S}\subset\hat{X}})\cdot[\hat{S}]\nump 0$
follows from Sublemma \ref{A23-L2-SL5}.
\Qed
\end{proof}

\begin{sublem}
\label{A23-L2-SL5}
In the situation of Sublemma \ref{A23-L2-SL4}, on $\hat{X}$,
$$
c_1^2(N_{\hat{S}\subset\hat{X}})\cdot [\hat{S}]\nump\eps^*(-2[S_1]\cdot[S_2]+c_1^2(N_{S\subset X})\cdot [S]).
$$
\end{sublem}

\begin{proof}
In the notations of the proof of Sublemma \ref{A23-L2-SL4},
denoting $\nu=\eta\circ\mu$, and using
the fact that $E_{\{1,2\}}$ is a
$\pp^1\times\pp^1$-bundle over a curve $S_{\{1,2\}}$,  the equation $\rho_i^2+\rho_i\pi_i^*c_1(N_i)+\pi_i^*c_2(N_i)=0$ on $E_i$, and
\cite[Thm 6.7]{Fu}, we obtain:
\begin{equation*}
\begin{split}
&c_1^2(N_{\hat{S}\subset\hat{X}})\cdot[\hat{S}]=\nu^*(c_1^2(N_{S'\subset\wt{X}})\cdot[S'])=\\
&\nu^*([E_1](\rho_1-\rho_2+\pi_1^*c_1(N_1))(-\rho_2+\pi_1^*c_1(N_1))^2+
[E_2](\rho_2+\rho_1+\pi_2^*c_1(N_2))(\rho_1+\pi_2^*c_1(N_2))^2)=\\
&\nu^*([E_1](\rho_2^2(\pi_1^*(3c_1(N_1)+c_1(N_2))+\rho_1)+\rho_2(-2\pi_1^*c_1(N_1)\rho_1)+\pi_1^*c_1^2(N_1)(\pi_1^*c_1(N_1)+\rho_1))+\\
&\hspace{5mm}[E_2](\rho_1^2(\pi_2^*(3c_1(N_2)-c_1(N_1))+\rho_2)+\rho_1(2\pi_2^*c_1(N_2)\rho_2)+\pi_2^*c_1^2(N_2)(\pi_2^*c_1(N_2)+\rho_2))=\\
&\nu^*((j_{\{1,2\}})_*(\pi_{\{1,2\}}^*(-3c_1(N_1)-3c_1(N_2))\rho_2+\pi_{\{1,2\}}^*(-3c_1(N_2)+3c_1(N_1))\rho_1-2\rho_1\rho_2)+\\
&\hspace{5mm}\pi^*(c_1^2(N_1)[S_1])+\pi^*(c_1^2(N_2)[S_2]))\nump\nu^*\pi^*(c_1^2(N_{S\subset X})\cdot [S]-2[S_1]\cdot [S_2]),
\end{split}
\end{equation*}
since $c_1(N_i)\nump 0$ on $S_{\{1,2\}}$ (note, that $S_{\{1,2\}}$ is connected).
Here $j_{\{1,2\}}:E_{\{1,2\}}\row\wt{X}$ is the closed embedding.
\Qed
\end{proof}

The subvariety $\ov{S}\cup\ov{F}$ satisfies the conditions of
Sublemma \ref{A23-L2-SL4}(1). Thus, we may substitute it by a single connected component and assume that our class is represented
by the class of $T\cup G$, where $[G]=(-\rho)\cdot\rho$, for
some divisor $\rho$, and $T$ is a smooth connected subvariety
with $c_1^3(N_{T\subset X})\cdot [T]\nump 0$, $c_1c_2(N_{T\subset X})\cdot [T]\nump 0$, and for $p=2$, in addition,
$c_1^2(N_{T\subset X})\cdot [T]\nump 0$.
We may assume $T$ and $G$ transversal, with $G$ and $T\cap G$ connected. Clearly, $c_1^3(N_{G\subset X})\cdot[G]\nump 0$,
$c_1c_2(N_{G\subset X})\cdot[G]\nump 0$, and for $p=2$, $c_1^2(N_{G\subset X})\cdot[G]\nump 0$ as well.
Since $[T]+[G]\nump 0$, applying Sublemma \ref{A23-L2-SL4}(2) and (1) again, we may represent our class by $[S]$,
where $S$ is smooth connected with
$c_1^3(N_{S\subset X})\cdot [S]\nump 0$, $c_1c_2(N_{S\subset X})\cdot [S]\nump 0$, and for $p=2$, in addition,
$c_1^2(N_{S\subset X})\cdot [S]\nump 0$. Lemma \ref{A23-L2} is proven.
\Qed
\end{proof}

Now as $S$ is smooth connected with numerically trivial 
zero-dimensional Chern classes, we can move up the dimension
and make the square of the $1$-st Chern class numerically
trivial as well. This is already achieved for the prime $2$. It
remains to treat the odd primes. We will, actually, make the mentioned $c_1^2$ numerically trivial not only on $X$, but 
already on $S$ itself. This will be important for the next step.

\begin{lem}
\label{A23-L3}
We may assume that $S$ is smooth connected with $c_1^3(N_{S\subset X})\cdot [S]\nump 0$,
$c_1c_2(N_{S\subset X})\cdot [S]\nump 0$ and $c_1^2(N_{S\subset X})\cdot [S]\nump 0$ on $X$.
We also may assume that $c_1^2(N_{S\subset X})\nump 0$ on $S$.
\end{lem}

\begin{proof}
By Lemma \ref{A23-L2}, we have all the needed conditions, aside from that on $c_1^2$.
Let us first make $c_1^2(N_{S\subset X})\cdot [S]\nump 0$ on $X$.
For ${\bf\un{p=2}}$ we already have it by Lemma \ref{A23-L2}.

\noindent
${\bf\un{(p\neq 2,3)}}$\hspace{2mm}
Let $\pi:\wt{X}\row X$ be the blow-up in the smooth connected complete intersection
$R=\frac{1}{2}c_1(N_{S\subset X})\cdot\frac{1}{3}c_1(N_{S\subset X})$ (of very ample divisors on $S$).
Let $E$ be the special divisor of $\pi$, and $\rho=-[E]$. Then $\pi^*([S])=[\wt{S}]+[F]$, where $\wt{S}$ is the proper transform
of $S$ and $F$ is a smooth connected (very ample) divisor on $E$, transversal to $\wt{S}$, with connected $\wt{S}\cap F$ and
with $[F]=[E](\rho+\pi_E^*c_1)$, where
$c_l=c_l(N_{S\subset X})$. Also, $c_1(N_{\wt{S}\subset\wt{X}})=2\rho+\pi_S^*c_1$ and $c_1(N_{F\subset\wt{X}})=\pi_E^*c_1$, where
$\pi_S:\wt{S}\row S$ and $\pi_E:E\row R$ are natural projections.
Since $\rho$ satisfies the equation $\rho^2+\frac{5}{6}\pi_S^*c_1\cdot\rho+\frac{1}{6}\pi_S^*c_1^2$ on $\wt{S}$, and $\ddim(R)=1$,
we have:
\begin{equation*}
\begin{split}
&c_1^2(N_{\wt{S}\subset\wt{X}})\cdot[\wt{S}]+c_1^2(N_{F\subset\wt{X}})\cdot[F]=[\wt{S}](2\rho+\pi_S^*c_1)^2+[E](\rho+\pi_E^*c_1)\pi_E^*c_1^2=\\
&[\wt{S}](-4\cdot{\textstyle{\frac{5}{6}}}\pi_S^*c_1\cdot\rho-4\cdot{\textstyle{\frac{1}{6}}}\pi_S^*c_1^2+4\pi_S^*c_1\cdot\rho+\pi_S^*c_1^2)=
[\wt{S}]({\textstyle{\frac{2}{3}}}\pi_S^*c_1\cdot\rho+{\textstyle{\frac{1}{3}}}\pi_S^*c_1^2)\nump
[\wt{S}]\cdot{\textstyle{\frac{1}{3}}}\pi_S^*c_1^2,
\end{split}
\end{equation*}
since $-[\wt{S}]\rho$ is represented by the class of a $\pp^1$-bundle $P=\pp_R(N_{R\subset S})$ over $R$, and $c_1\nump 0$ on $R$,
as $\ddeg(c_1^3\cdot [S])=0$ and $R$ is connected.
By the same reason,
\begin{equation*}
\begin{split}
&[\wt{S}]\cdot[F]=[\wt{S}](-\rho)(\rho+\pi_E^*c_1)=
[\wt{S}]({\textstyle{\frac{5}{6}}}\pi_S^*c_1\cdot\rho+{\textstyle{\frac{1}{6}}}\pi_S^*c_1^2-\pi_S^*c_1\cdot\rho)\nump
[\wt{S}]\cdot{\textstyle{\frac{1}{6}}}\pi_S^*c_1^2.
\end{split}
\end{equation*}
At the same time, since $\rho\pi_S^*c_1\nump 0$ on $\wt{S}$, as we saw above, and $[S]c_1^3\nump 0$, by assumption,
\begin{equation*}
\begin{split}
&c_1^3(N_{\wt{S}\subset\wt{X}})\cdot[\wt{S}]=[\wt{S}](2\rho+\pi_S^*c_1)^3\nump[\wt{S}](8\rho^3+\pi_S^*c_1^3)\nump\\
&[\wt{S}]\cdot 8\rho^3=
[P]\cdot 8({\textstyle{\frac{5}{6}}}\pi_S^*c_1\cdot\rho+{\textstyle{\frac{1}{6}}}\pi_S^*c_1^2)\nump
0,\,\,\,\text{and}\\
&c_1^3(N_{F\subset\wt{X}})\cdot[F]=[E](\rho+\pi_E^*c_1)c_1^3=0,
\end{split}
\end{equation*}
again, since $\ddim(R)=1$ and $c_1\nump 0$ on $R$. Using the same arguments,
\begin{equation*}
\begin{split}
&c_1c_2(N_{\wt{S}\subset\wt{X}})\cdot[\wt{S}]=[\wt{S}](2\rho+\pi_S^*c_1)(\rho^2+\pi_S^*c_1\cdot\rho+\pi_S^*c_2)\nump\\
&\pi_S^*(c_1c_2\cdot[S])+[\wt{S}]\cdot 2\rho^3\nump\pi_S^*(c_1c_2\cdot[S])\nump 0,\,\,\,\text{and}\\
&c_1c_2(N_{F\subset\wt{X}})\cdot[F]=[E](\rho+\pi_E^*c_1)^2(-\rho)\pi_E^*c_1\nump 0.
\end{split}
\end{equation*}

Let $\eps:\hat{X}\row\wt{X}$ be the blow-up from Sublemma \ref{A23-L2-SL4} applied to $\wt{S}\cup F$. Then we get a smooth
connected subvariety $\hat{S}$ on $\hat{X}$ such that $c_1^3(N_{\hat{S}\subset\hat{X}})\cdot [\hat{S}]\nump 0$ and
$c_1c_2(N_{\hat{S}\subset\hat{X}})\cdot [\hat{S}]\nump 0$. Finally, by  Sublemma \ref{A23-L2-SL5},
\begin{equation*}
\begin{split}
&c_1^2(N_{\hat{S}\subset\hat{X}})\cdot[\hat{S}]=\eps^*(c_1^2(N_{\wt{S}\subset\wt{X}})\cdot[\wt{S}]+
c_1^2(N_{F\subset\wt{X}})\cdot[F]-2[\wt{S}]\cdot[F])\nump\\
&\eps^*([\wt{S}]\cdot({\textstyle{\frac{1}{3}}}\pi_S^*c_1^2-2\cdot{\textstyle{\frac{1}{6}}}\pi_S^*c_1^2))=0.
\end{split}
\end{equation*}
The case ${\bf{(p\neq 2,3)}}$ is done.

\noindent
${\bf\un{(p=3)}}$\hspace{2mm}
Let $d_1=c_1^2+c_2$ be the characteristic class of degree $2$ corresponding to the reduced power operation $P^1:\Ch^*\row\Ch^{*+2}$
(modulo $3$). Since $[S]\nump 0$, by Proposition \ref{AB1},
$d_1(N_{S\subset X})\cdot [S]=P^1([S])\nump 0$ as well.  On the other hand, $c_2(N_{S\subset X})\cdot [S]=[S]\cdot [S]\nump 0$.
Hence, $c_1^2(N_{S\subset X})\cdot [S]\nump 0$ too. The case ${\bf{(p=3)}}$ is done.

Thus, we managed to make $c_1^2(N_{S\subset X})\cdot [S]\nump 0$, while keeping $c_1^3(N_{S\subset X})\cdot [S]\nump 0$ and
$c_1c_2(N_{S\subset X})\cdot [S]\nump 0$. Let us now make the complete intersection
$c_1(N_{S\subset X})\cdot c_1(N_{S\subset X})$ on $S$ anisotropic.

By blowing-up $S$, we may assume that $S\subset Y$, where $Y$ is smooth connected divisor on $X$. Note, that the new
$c_1^3,c_1c_2$ and $c_1^2$ characteristic classes
are the pull-backs of the old ones, and so, are still numerically trivial. By Statement \ref{LS-div-conn} and flexibility of $k$, there
exists an irreducible and smooth outside an anisotropic subscheme (of $S$) divisor $Z$ containing $S$, such that the restriction
$\Ch^1(X)\twoheadrightarrow\Ch^1(Z\backslash S)$ is surjective. Let $\pi:\wt{X}\row X$ be an embedded desingularization of $Z$ and $S$,
with proper pre-images $\wt{Z}$ and $\wt{S}$, which are smooth connected subvarieties. Since singularities were anisotropic, the maps
$\pi^*:\Ch^*_{k/k}(X)\stackrel{=}{\row}\Ch^*_{k/k}(\wt{X})$ and
$\pi_Z^*:\Ch^*_{k/k}(Z\backslash S)\stackrel{=}{\row}\Ch^*_{k/k}(\wt{Z}\backslash\wt{S})$ are isomorphisms, and so, the restriction
$\Ch^1_{k/k}(\wt{X})\twoheadrightarrow\Ch^1_{k/k}(\wt{Z}\backslash\wt{S})$ is surjective as well.
Thus, the group $\Ch^1_{k/k}(\wt{Z})$ is generated by the image of $j^*:\Ch^1_{k/k}(\wt{X})\row\Ch^1_{k/k}(\wt{Z})$ and the class
$[\wt{S}]$.
Also, on $\wt{X}$, the classes
$c_1^3(N_{\wt{S}\subset\wt{X}})\cdot [\wt{S}]$, $c_1c_2(N_{\wt{S}\subset\wt{X}})\cdot [\wt{S}]$ and
$c_1^2(N_{\wt{S}\subset\wt{X}})\cdot [\wt{S}]$ are $\nump 0$.
Since $c_1^2(N_{\wt{S}\subset\wt{X}})\cdot [\wt{S}]\nump 0$ on $\wt{X}$, this class
will be orthogonal to the $\op{im}(j^*)$ on $\wt{Z}$. While, on $\wt{Z}$,
$$
\ddeg(c_1^2(N_{\wt{S}\subset\wt{X}})[\wt{S}]\cdot_{\wt{Z}}[\wt{S}])=
\ddeg(c_1^2(N_{\wt{S}\subset\wt{X}})c_1(N_{\wt{S}\subset\wt{Z}})[\wt{S}])=
\ddeg(c_1^3(N_{\wt{S}\subset\wt{X}})\cdot [\wt{S}])=0,
$$
since $c_1^2(N_{\wt{S}\subset\wt{X}})\cdot[\wt{S}]\nump 0$ on $\wt{X}$ and $c_1(N_{\wt{Z}\subset\wt{X}})$ is in the image of $j^*$.
Hence, $c_1^2(N_{\wt{S}\subset\wt{X}})\cdot[\wt{S}]\nump 0$ already on $\wt{Z}$.
Since this class is a complete intersection on $\wt{Z}$, the intersection of generic representatives of the respective
(very ample) linear systems is anisotropic by Statement \ref{alg-geom-Lef}. The respective subvariety $S$ is smooth connected,
and $c_1^3$ and $c_1c_2$ characteristic classes are preserved.
Since $k$ is flexible, we may assume that our varieties are defined over $k$.
Lemma \ref{A23-L3} is proven.
\Qed
\end{proof}

In order to apply Corollary \ref{LS-r2-nump-div} and finish the proof of Proposition \ref{A23}, it remains to terminate numerically the $1$-st Chern class of our $S$.

\begin{lem}
\label{A23-L4}
We may assume that $S$ is smooth connected and $c_1^m(N_{S\subset X})\cdot [S]\nump 0$, for $m\geq 0$.
\end{lem}

\begin{proof}
By Lemma \ref{A23-L3}, we may assume that $S$ is smooth connected variety with the numerically trivial $c_1^3$, $c_1c_2$ and $c_1^2$
characteristic classes. It remains to make $c_1$ numerically trivial.
We need to treat separately the case $p=2$ and that of odd primes.

\noindent
${\mathbf\un{(p\neq 2)}}$
Let $R$ be the generic representative of the (very ample) linear system $|\frac{1}{2}c_1(N_{S\subset X})|$ on $S$.
It is a smooth connected surface and, by Statement \ref{alg-geom-Lef}, we have the surjection $\Ch^1(S)\twoheadrightarrow\Ch^1(R)$.
Since $k$ is flexible, we may assume that it is defined over $k$.
Let $\pi:\wt{X}\row X$ be the blow-up at $R$,
with the special connected divisor $E$ and $\rho=-[E]$. Then $\pi^*([S])=[\wt{S}]+[F]$, where $\wt{S}$ is the proper pre-image of $S$ and
$F$ is a smooth connected divisor on $E$ transversal to $\wt{S}$, with connected $\wt{S}\cap F$ and with
$[F]=[E](\rho+\pi_E^*c_1)$, where $c_l=c_l(N_{S\subset X})$.
We have:
$$
c_1(N_{\wt{S}\subset\wt{X}})\cdot[\wt{S}]+c_1(N_{F\subset\wt{X}})\cdot [F]=[\wt{S}](2\rho+\pi_S^*c_1)+[E](\rho+\pi_E^*c_1)\pi_E^*c_1=
[E](\rho+\pi_E^*c_1)\pi_E^*c_1\nump 0,
$$
since, on $\wt{S}\cong S$, $\rho=-[R]$, and
$c_1\cdot [R]\nump 0$ on $R$, as $c_1^2(N_{S\subset X})\nump 0$ on $S$ and $\Ch^1(S)\twoheadrightarrow\Ch^1(R)$. On the other hand, on $R=\wt{S}\cap E$,
$$
[\wt{S}]\cdot[F]=[\wt{S}](-\rho)(\rho+c_1)=[R]({\textstyle{\frac{1}{2}}}c_1)\nump 0,
$$
and moreover, we may assume the intersection $\wt{S}\cap F$ to be anisotropic.
By Statement \ref{LS-codim2-single}, there exists a blow-up $\mu:\ov{X}\row\wt{X}$,
such that $\mu^*([\wt{S}\cup F])$ is represented by the class
of a smooth connected subvariety $\ov{S}$ and such that the characteristic classes of $\ov{S}$ in $\Ch^*_{k/k}$ are $\mu^*$ of the
respective classes of $\wt{S}\cup F$. In particular, the $c_1$-class of it is numerically trivial.

It remains to check $c_1^2$
and $c_1^3$ characteristic classes of $\wt{S}\cup F$. We have:
$$
c_1^2(N_{\wt{S}\subset\wt{X}})\cdot[\wt{S}]+c_1^2(N_{F\subset\wt{X}})\cdot[F]=[\wt{S}](2\rho+\pi_S^*c_1)^2+[E](\rho+\pi_E^*c_1)\pi_E^*c_1^2=
[E](\rho+\pi_E^*c_1)\pi_E^*c_1^2\nump 0
$$
on $\wt{X}$, since $(2\rho+\pi_S^*c_1)=0$ on $\wt{S}$, and $\ddeg(c_1^3(N_{S\subset X})\cdot[S])=0$ while $R$ is connected.
Similarly,
$$
c_1^3(N_{\wt{S}\subset\wt{X}})\cdot[\wt{S}]+c_1^3(N_{F\subset\wt{X}})\cdot[F]=
[\wt{S}](2\rho+\pi_S^*c_1)^3+[E](\rho+\pi_E^*c_1)\pi_E^*c_1^3=0,
$$
by the same and dimensional reasons.
The case ${\mathbf(p\neq 2)}$ is done.

\noindent
${\mathbf\un{(p=2)}}$
The characteristic class $c_1$ corresponds to the reduced power operation $P^1:\Ch^*\row\Ch^{*+1}$ (modulo $2$).
Since $[S]\nump 0$, by Proposition \ref{AB1}, $c_1(N_{S\subset X})\cdot [S]=P^1([S])\nump 0$ as well.
Lemma \ref{A23-L4} is proven.
\Qed
\end{proof}

Proposition \ref{A23} now follows from Corollary \ref{LS-r2-nump-div} and flexibility of $k$.
\Qed
\end{proof}

As Conjecture \ref{main-conj} was established for all varieties
of dimension $\leq 5$, from the existence of push-forward and
pull-back structure we obtain that isotropic Chow groups form
a quotient of the 3-rd theory of higher type associated to 
$\Ch$.

\begin{proposition}
\label{A24}
Let $k$ be a flexible field. The projection $\Ch^*\twoheadrightarrow\Ch^*_{k/k}$ factors through $\Ch^*_{(3)}$.
\end{proposition}

\begin{proof}
A class $u\in\Ch^*(X)$ is $=0\in\Ch^*_{(3)}$, if it can be presented as $f_*(y\cdot g^*(v))$, where
$X\stackrel{f}{\llow}X\times Q\stackrel{g}{\lrow}Q$ are natural projections, $Q$ is a smooth projective variety of dimension $5$,
$y\in\Ch^*(X\times Q)$ and $v\in\Ch^*(Q)$ is $\nump 0$. Then, by Proposition \ref{A12}, Corollary \ref{A13N}, Proposition \ref{A15},
Proposition \ref{A20} and Proposition \ref{A23},
$v=0\in\Ch_{k/k}^*(Q)$, and so, $u=0\in\Ch_{k/k}^*(X)$.
\Qed
\end{proof}

\section{Thick local categories}
\label{thick}

In this section we extend the definition of local motivic category to arbitrary finite coefficients $\zz/n$ and
introduce the {\it thick} versions of it which have better conservativity properties.

\begin{definition}
\label{higher-order}
Let $n\in\nn$. Let $P$ and $Q$ be smooth varieties of finite type over $k$. We say that $\hii_Q\bor{n}\hii_P$, if $P$ is $n$-isotropic
over every generic point of $Q$, while
$Q$ is $n$-anisotropic over some generic point of $P$.
\end{definition}

Let $E/k$ be some finitely generated extension and $P$ be smooth connected variety with $k(P)=E$.
Let ${\mathbf{Q}}^n$ be the disjoint union of all smooth connected varieties $Q$ of finite type with $\hii_Q\bor{n}\hii_P$ and
$$
\Upsilon_P^n:=\whii_{{\mathbf{Q}}^n}\otimes\hii_{P}.
$$
Generalizing Definition \ref{D-Iso-2}, we can define the {\it local motivic category with $\zz/n$-coefficients}.

\begin{definition}
\label{n-local-mot-cat}
The local motivic category with $\zz/n$-coefficients
$$
\dmELF{E}{k}{\zz/n}:=\Upsilon_P^n\otimes\dmkF{\zz/n}.
$$
\end{definition}

If we are interested in a $p$-localized situation,
we can define the {\it thick local motivic categories}.

\begin{definition}
\label{thick-loc-mot-cat}
The {\it $p$-local motivic category of thickness $r$ and $F$-coefficients}
$$
\dmELFr{E}{k}{F}{r}:=\Upsilon_P^{p^r}\otimes\dmkF{F}.
$$
\end{definition}

In particular, the local category with $\zz/p^r$-coefficients $\dmELF{E}{k}{\zz/p^r}$ is the $p$-local motivic category
of thickness $r$ and $\zz/p^r$-coefficients $\dmELFr{E}{k}{\zz/p^r}{r}$.

Since $\hii_Q\bor{n}\hii_P$ implies that $\hii_Q\bor{m}\hii_P$ for $m|n$, we get natural functors
$$
\dmELF{E}{k}{\zz/n}\row\dmELF{E}{k}{\zz/m}\hspace{5mm}\text{and}\hspace{5mm}
\dmELFr{E}{k}{F}{r}\row\dmELFr{E}{k}{F}{s},
$$
for any $r\geq s$ and $m|n$,
commuting with the natural functors $\ffi_E^{\zz/n}:\dmkF{\zz/n}\row\dmELF{E}{k}{\zz/n}$.

In the usual way, we can introduce {\it local geometric motives}, {\it local Chow motives} and {\it local Chow groups}.
Exactly as in Proposition \ref{D-ChowGroups} we get the description of {\it isotropic Chow groups}:

\begin{proposition}
\label{Int-Chow-groups}
$$
\CH_{k/k}(X;\zz/n)=\CH(X;\zz/n)/(n\text{-anisotropic classes}).
$$
\end{proposition}

Similarly, for the {\it thick local Chow groups}, we have:
$$
\CH_{\{k/k\}^r}(X;F)=\CH(X;F)/(p^r\text{-anisotropic classes}).
$$

Since every $n$-anisotropic class is numerically equivalent to zero modulo $n$, we obtain the surjection:
\begin{equation}
\label{Int-proj-Num}
\CH_{k/k}(X;\zz/n)\twoheadrightarrow\CH_{Num}(X;\zz/n).
\end{equation}

\begin{question}
\label{Int-que-r}
Let $k$ be flexible. Is it true that $\CH_{k/k}(X;\zz/p^r)=\CH_{Num}(X;\zz/p^r)$?
\end{question}

Using the arguments of Proposition \ref{A12} and Corollary \ref{A13N}, we see that the answer is positive for divisors and for zero-cycles.
In particular, the projection $\CH(X;\zz/p^r)\twoheadrightarrow\CH_{k/k}(X;\zz/p^r)$ factors through $\CH_{alg}(X;\zz/p^r)$.
Also, it is easy to see that it is true for
cycles of dimension $1$, provided $p>2$, and so, the mentioned projection factors through $\CH_{(2)}(X;\zz/p^r)$, in this case.

With the increase of $r$ (and fixed $s$), the family of functors
$$
\ffi_E^{r;\zz/p^s}:\dmkF{\zz/p^s}\row\dmELFr{E}{k}{\zz/p^s}{r},
$$
for all f.g. extensions $E/k$, becomes more and more conservative. But the target categories are getting more complicated.
At the same time, the categories $\dmELF{E}{k}{\zz/p^r}$ (that is $r=s$) should be simpler, and the natural strategy to
improve conservativity is to pass to $\zz_{(p)}$-coefficients.

The local motivic category with $\zz_{(p)}$-coefficients $\dmELF{E}{k}{\zz_{(p)}}$
should be the local category of {\it $\infty$ thickness} $\dmELFr{E}{k}{\zz_{(p)}}{\infty}$.
This should be defined as a limit $\operatornamewithlimits{lim}_r\dmELFr{E}{k}{\zz_{(p)}}{r}$ of categories of finite thickness
and will be considered in a separate paper.

\section{Auxiliary results}
\label{A-r}

\subsection{Up to blow-up, Chow ring of a variety is generated by divisors}

The following result is crucial for most of our constructions.

\begin{theorem}
\label{AA1}
Let $X$ be a smooth projective variety and $y\in\CH^r(X)$. Then there exists a blow-up $\pi:\wt{X}\row X$ such that
$\pi^*(y)$ is a $\zz$-polynomial in divisor classes.
\end{theorem}

\begin{proof}
Induction on the $\ddim(X)$. Below we will denote this induction as $\op{Ind}_1$.

\noindent
{\bf\un{(${\mathbf{\op{Ind}_1}}$ base)}} $\ddim(X)=0$. Nothing to check.

\noindent
{\bf\un{(${\mathbf{\op{Ind}_1}}$ step)}} Can assume that $r>0$. Then $y$ has support on some divisor.
By blowing $X$ up we can assume that this divisor has
strict normal crossings, and by the following Lemma, can assume that $y$ has support on a smooth divisor $D$.

\begin{lem}
\label{L1-AA1}
If $y=\sum_i y_i$, and the statement is true for each $y_i$, then it is true for $y$.
\end{lem}

\begin{proof}
Suppose, for each $i$, there exists such blow-up
$\pi_i:\wt{X}_i\row X$ that $\pi_i^*(y_i)$ is a polynomial in divisorial classes. Then, by the results of Hironaka \cite{Hi}, there
exists a blow-up $\pi:\wt{X}\row X$, which covers $\pi_i$, for each $i$. Then, clearly, $\pi^*(y)$ is a polynomial in
divisorial classes.
\Qed
\end{proof}

Let now $j_D:D\row X$ be a smooth divisor and $y$ is supported on $D$.
Let us say that the pair $(D,X)$ has a {\it special structure} with the base $B$ if $D$ is a consecutive projective bundle over $B$ where the canonical
line bundles $O(1)$ of all these fibrations are the restrictions of some line bundles from $X$.
Every pair $(D,X)$ possesses a "trivial" special structure with the base $B=D$.
We will prove our statement by the induction on the $\ddim(B)$. We will denote this induction as $\op{Ind}_2$ below.

\noindent
{\bf\un{(${\mathbf{\op{Ind}_2}}$ base)}} $\ddim(B)=0$. Then $D$ is a disjoint union of consecutive projective bundles whose all canonical bundles $O(1)$
are restrictions of some line bundles from $X$. Then $\CH^*(D)$ is generated by $j_D^*(c_1(L))$, for some line bundles $L$ on $X$.
Since $y=(j_D)_*(\ov{y})$, for some $\ov{y}\in\CH^*(D)$, it is a polynomial in divisorial classes.

\noindent
{\bf\un{(${\mathbf{\op{Ind}_2}}$ step)}} We have: $y=(j_D)_*(\ov{y})$, where $\ov{y}=\sum_l\eps^*(u_l)\cdot r_l$ with $u_l\in\CH^*(B)$, $\eps:D\row B$ is the natural projection, and $r_l$ is a monomial
in $\rho_t$'s, where $\rho_t=c_1(O(1)_t)$ is the restriction of some divisorial class from $X$. By Lemma \ref{L1-AA1}, we can
assume that $\ov{y}=\eps^*(u)$, where $u\in\CH^*(B)$.

Since $\ddim(B)<\ddim(X)$, by the inductive assumption of $\op{Ind}_1$, there exists a blow-up $\tau:\wt{B}\row B$, such that
$\tau^*(u)$ is a polynomial in divisorial classes. We have a cartesian diagram of blow-ups:
$$
\xymatrix{
\wt{D} \ar[r]^{\tau_D} \ar[d]_{\wt{\eps}} &  D \ar[d]^{\eps} \\
\wt{B} \ar[r]^{\tau} &  B
}
$$
Let $\ffi:\wt{X}\row X$ be a blow-up of $X$ in the same centers as $\tau_D$, and $j_{\wt{D}}:\wt{D}\row\wt{X}$.
Then $\ffi^*((j_D)_*\eps^*(u))=(j_{\wt{D}})_*\tau_D^*(\eps^*(u))+\sum_m v_m$, where $v_m$ are supported on the components $\wt{E}_m$
of the special divisor $\wt{E}=\cup_m\wt{E}_m$ of the blow-up $\ffi$ - see \cite[Thm 6.7]{Fu} (or \cite[Prop. 5.27]{so2}).
Let $X_m$ be a variety obtained after $m$ blow-ups with the special divisor
$j_{E_m}:E_m\row X_m$ of the $m$-th blow-up and the projection $\ffi_m:\wt{X}\row X_m$ whose restriction to $\wt{E}_m$ is
the blow-up $\alpha_m:\wt{E}_m\row E_m$. Then (again by \cite[Thm 6.7]{Fu}), the image of
$$
(\ffi_m^*(j_{E_m})_*-(j_{\wt{E}_m})_*\alpha_m^*):\CH^*(E_m)\lrow\CH^*(\wt{X})
$$
has support on $\cup_{n>m}\wt{E}_n$, while the map
$$
\oplus\alpha_m^*:\oplus_m\CH^*(E_m) \lrow \CH^*(\wt{E})
$$
is surjective (since $\cup_{n\geq m}\wt{E}_n\backslash\cup_{n>m}\wt{E}_n$ is an open subvariety of $E_m$). Hence, the map
$$
\oplus\ffi_m^*(j_{E_m})_*:\oplus_m\CH^*(E_m)\lrow\CH^*(\wt{X})
$$
covers the image of $(j_{\wt{E}})_*$.

The pair $(E_m,X_m)$ has a {\it special structure} with the base of dimension smaller than $B$ (namely, the center of the $m$-th blow-up of
$\tau$), where $c_1(O(1))$ of the external projective fibration is the restriction of $[-E_m]$ from $X_m$, and all the other (internal)
canonical bundles are induced by the {\it special structure} on $(D,X)$, and so, are restrictions of some divisorial classes from $X$.
By the inductive assumption of $\op{Ind}_2$, any element in $(j_{E_m})_*\CH^*(E_m)$ is a polynomial in divisorial classes over some
blow-up $\wt{X}_m\row X_m$. By Lemma \ref{L1-AA1}, we obtain that $\sum_m v_m$ is a polynomial in divisorial classes
over some blow-up of $\wt{X}$. Hence, it remains to deal with $(j_{\wt{D}})_*\tau_D^*(\eps^*(u))$.

We know that $\tau_D^*(\eps^*(u))=\wt{\eps}^*\tau^*(u)$ is a polynomial in divisorial classes by construction. Let us denote this
element again as $\ov{y}$ and the divisor (on which it is supported) as $D$.
By Lemma \ref{L1-AA1},
we can assume that $\ov{y}$ is a monomial $x_1\cdot\ldots\cdot x_{r-1}$ in \un{very ample} divisorial classes on $D$.
We will use induction on $r$. Below it will be denoted as the $\op{Ind}_3$.

\noindent
{\bf\un{(${\mathbf{\op{Ind}_3}}$ base)}} When $r=1$ there is nothing to prove as $y=(j_D)_*(\ov{y})=[D]$.

\noindent
{\bf\un{(${\mathbf{\op{Ind}_3}}$ step)}} Consider the chain of co-dimension $1$ regular embeddings
$D\stackrel{i_1}{\llow}Y_1\stackrel{i_2}{\llow}Y_2\stackrel{i_3}{\llow}\ldots$, where $[Y_{k+1}]$ represents the restriction of $x_1$
to $Y_k$. Construct the chain of blow-ups
$\ov{X}\stackrel{\pi_1}{\llow}\ov{X}_1\stackrel{\pi_2}{\llow}\ov{X}_2\stackrel{\pi_3}{\llow}\ldots$ in co-dimension-two centers
$Z_k\row\ov{X}_{k-1}$ inductively as follows.
The special divisor $E_k$ of $\pi_k$ is a consecutive projective line fibration $E_k\stackrel{\ffi_k}{\row}Y_k$ and
$Z_{k+1}=\ffi_k^{-1}(Y_{k+1})$. In particular, $E_k\stackrel{\eps_k}{\row}Z_k$ is a projective line fibration and
$\ffi_k=\eps_1\circ\ldots\circ\eps_k$. Let ${\mathbf{i}}_k=i_1\circ\ldots\circ i_k$ and $\ov{y}_k=\ffi_k^*({\mathbf{i}}_k)^*(\ov{y})$ be the class supported on $E_k$.

Then we have a cartesian diagram
$$
\xymatrix{
E_{k+1} \ar[r]^{i_{E_{k+1}}} \ar @{}[rd]|{\square}& \ov{X}_{k+1} \\
Z_{k+1} \ar[u] \ar[r]_{\alpha} & \ov{E}_k \ar[u]_{j_{\ov{E}_k}}
}
$$
with $\ov{E}_k=E_k$ and $\ov{y}_{k+1}=\eps_{k+1}^*\alpha^*\ov{y}_k$.
By (the Chow group version of) \cite[Prop. 5.27]{so2}, we have:
$$
\pi_{k+1}^*((i_{E_k})_*\ov{y}_k)=(j_{\ov{E}_k})_*\ov{y}_k-(i_{E_{k+1}})_*\ov{y}_{k+1}.
$$
Here, by the inductive assumption of $\op{Ind}_3$, the first summand
$(j_{\ov{E}_k})_*\ov{y}_k=(j_{\ov{E}_k})_*(x_1\cdot x_2\cdot\ldots\cdot x_{r-1})=
[E_{k+1}]\cdot(j_{\ov{E}_k})_*(x_2\cdot\ldots\cdot x_{r-1})$
is expressible as a polynomial in divisorial classes over some
blow-up of $\ov{X}_{k+1}$, so (using Lemma \ref{L1-AA1}) the question about $\ov{y}_k$ supported on $E_k\row\ov{X}_k$ is reduced to
the question about $\ov{y}_{k+1}$ supported on $E_{k+1}\row\ov{X}_{k+1}$. Thus, it is sufficient to show  that
our statement is true for $\ov{y}_k$ supported on $E_k\row\ov{X}_k$, for at least, one $k$.
But for $k=\ddim(X)-r$, the class $({\mathbf{i}}_k)^*(\ov{y})$ is zero by dimensional reasons.
Hence, the ($\op{Ind}_3$ step) is proven. This implies ($\op{Ind}_2$ step) and ($\op{Ind}_1$ step).
The Theorem is proven.
\Qed
\end{proof}

\begin{corollary}
\label{AA2}
Let $X$ be a smooth projective variety and $y\in\CH^r(X)$. Then there exists a blow-up $\pi:\wt{X}\row X$, such that
$\pi^*(y)$ is represented by a linear combination of classes of smooth complete intersections of very ample divisors which are
transversal to each other. In particular, for $r>\ddim(X)/2$, it is represented by the difference of classes of two smooth disjoint
subvarieties.
\end{corollary}

\subsection{General position results}

In this section we present various general position arguments which permit to replace
cycles by the classes of connected subvarieties and, in some cases, reduce the anisotropy of a class $[S]$
to the numerical triviality of some characteristic classes of $S$.

The following simple and well-known "Chow group shadow" of the Lefshetz theorem is one of our key ingredients.

\begin{statement}
\label{alg-geom-Lef}
Let $X$ be a scheme with a map $X\stackrel{f}{\row}\pp^n$ and
$\iota:D_{\eta}\row X$ be the generic hyperplane section of $X$ (over $k((\pp^n)^{\vee})$).
Then the pull back $\iota^*:\CH^*(X)\twoheadrightarrow\CH^*(D_{\eta})$ is surjective.
If $X$ is a smooth variety, so is $D_{\eta}$.
\end{statement}

\begin{proof}
Consider $Y\subset X\times(\pp^n)^{\vee}$ given by $Y=\{(x,H)|f(x)\in H\}$.
Then $Y$ is a projective bundle: $Y=\op{Proj}_X(V)$, where $V=W/O(-1)$ (with $(\pp^n)^{\vee}=\pp(W)$).
Let $Y_{\eta}$ be the generic fiber of the projection
$Y\row(\pp^n)^{\vee}$. This is the "generic" hyperplane section $D_{\eta}$.
By the projective bundle theorem and localization, we have surjections
$\CH^*(X)[\rho]\twoheadrightarrow\CH^*(Y)\twoheadrightarrow\CH^*(Y_{\eta})$, where $\rho=c_1(O(1))$. Since $\rho$ is the pull-back
of the class of a hyperplane on $(\pp^n)^{\vee}$, it is zero in $\CH^*(Y_{\eta})$. Thus, we get the surjection
$\CH^*(X)\twoheadrightarrow\CH^*(Y_{\eta})$. Finally, if $X$ is smooth, so is $Y$ and $Y_{\eta}$.
\Qed
\end{proof}

The next result shows that any closed subscheme of an effective Cartier divisor can be made the base set of the respective linear system (which, in turn, can be made very ample), if we modify the divisor by an 
$n$-multiple of some other divisor, for a given natural $n$.

\begin{statement}
\label{LS-Sbs}
Let $X$ be an irreducible quasi-projective variety and $n\in\nn$.
Let $E$ be an effective Cartier divisor on $X$, and $S\subset E$ be a closed subscheme.
Then there exists a very ample divisor $D$ such that the linear system $|E+n\cdot D|_S$ consists
of very ample divisors and $S$ is the base set of it.
\end{statement}

\begin{proof}
There are very ample divisors $F_1,F_2$ on $X$ such that $[E]=[F_1]-[F_2]$. Then the class $[E+nF_2]=[F_1+(n-1)F_2]=[H]$ is very ample
and defines an embedding $X\hookrightarrow\pp^N$. Let $\ov{S}$ be the closure of the image of $S$ in this embedding.
The coordinate ring of $\ov{S}$ has relations of degree $\leq m$.
In other words, the base set of the linear system $|kH|_S$ is $S$, for any $k\geq m$. Take $k=nm+1$. Then
$[kH]=[(nm+1)(E+nF_2)]=[E+n(mF_1+((n-1)m+1)F_2)]=[E+nD]$, where $D=mF_1+((n-1)m+1)F_2$ is very ample.
\Qed
\end{proof}

The following result will enable us to substitute a multi-component divisor with only $n$-anisotropic singularities, which
contains a given closed subscheme, by a single component one.
The same can be done with a collection of such divisors and
subschemes with the result having all the faces connected, if
the original collection of divisors was with strict normal crossings modulo $n$-anisotropic subvarieties.
This will be our key tool below.

\begin{statement}
\label{LS-div-conn}
Let $n\in\nn$, $X$ be a smooth projective connected variety of dimension $d$ and
$E=\cup_iE_i$ be a divisor on it with strict normal crossings
outside an $n$-anisotropic closed subscheme, with, possibly, reducible components $E_i$, and let
$S_i\subset E_i$ be some closed subschemes such that, for any subset $I$ of the set of indices, $\ddim(S_I)<d-\#(I)$, where
$S_I=\cap_{i\in I}S_i$.
Then, over some purely transcendental f.g. extension of $k$, there is a divisor $Z=\cup_iZ_i$, where $[Z_i]=[E_i]\in\CH^1(X;\zz/n)$,
$S_i\subset Z_i$, for any $I$, the variety $Z_I=\cap_{i\in I}Z_i$ is irreducible, and $Z$ has strict normal crossings outside an
$n$-anisotropic closed subscheme of $S$. Moreover, the restriction $\CH^*(X)\twoheadrightarrow\CH^*(Z_i\backslash S_i)$ is surjective.
\end{statement}

\begin{proof}
By the Statement \ref{LS-Sbs}, for any $i$, there exists a very ample divisor $D_i$, such that the linear system
$\Phi_i=|E_i+n\cdot D_i|_{S_i}$
consists of very ample divisors and has the base set $S_i$. Let $Z_i$ be the generic element of this linear system
(defined over $k(\pp(\Phi_i))$). Clearly, $[Z_i]=[E_i]\in\CH^1(X;\zz/n)$ and $S_i\subset Z_i$.
Let us show that $Z_I=\cap_{i\in I}Z_i$ is irreducible. We will prove by induction on $\#(I)$ that, for any $i\in I$, the restriction
$\CH^0(Z_{I\backslash i})\twoheadrightarrow\CH^0(Z_I)$ is surjective. The {\bf(base)} $I=\emptyset$ is empty.

\noindent
{\bf\un{(step)}} The linear system $\Phi_i$ defines an embedding of $X\backslash S_i$ into a projective space, and for any subscheme
$T\subset X$ defined over some field $L$ and $Z_i$ defined over $L(\pp(\Phi_i))$ as above, we have:
$\ddim(T\cap(Z_i\backslash S_i))\leq\ddim(T)-1$. By the Statement
\ref{alg-geom-Lef}, we have a surjection
$\CH^0(Z_{I\backslash i}\backslash Z_{I\backslash i}\cap S_i)\twoheadrightarrow\CH^0(Z_I\backslash Z_{I\backslash i}\cap S_i)$.
We know that all the components of $Z_I$ have dimensions $\geq d-\#(I)$. On the other hand, the scheme
$Z_{I\backslash i}\cap S_i$ is the union $\cup_{i\in J\subset I}Y_J$, where
$Y_J=(\cap_{j\in J}S_j)\cap(\cap_{j\in I\backslash J}(Z_j\backslash S_j))$. Since $\ddim(Y_J)\leq\ddim(S_J)-\#(I\backslash J)$,
we obtain that $\ddim(Z_{I\backslash i}\cap S_i)\leq\op{max}(\ddim(S_J)-\#(I\backslash J)|i\in J\subset I)<d-\#(I)$. Hence,
$\CH^0(Z_I)=\CH^0(Z_I\backslash Z_{I\backslash i}\cap S_i)$, and we get the surjection
$\CH^0(Z_{I\backslash i})\twoheadrightarrow\CH^0(Z_I)$. Induction step is proven.

Thus, we have the surjection $\CH^0(X)\twoheadrightarrow\CH^0(Z_I)$, and since $X$ is irreducible, so is $Z_I$.

Our system $\Phi_i$ contains $E_i+|n\cdot D_i|$.
The generic representative of $|n\cdot D_i|$ is $n$-anisotropic
(by Statement \ref{alg-geom-Lef} and the arguments from the proof of Proposition \ref{A12}).
Thus, the generic representative $G_i$ of $E_i+|n\cdot D_i|$
has only $n$-anisotropic singularities and, modulo $n$-anisotropic subvarieties, the divisor $G=\cup_iG_i$ has strict normal crossings.
Consequently, the generic representative $Z_i$ of our system
$|E_i+n\cdot D_i|_Y$ will have only $n$-anisotropic singularities too, and the divisor $Z=\cup_i Z_i$ will have strict normal crossings
modulo $n$-anisotropic subvarieties.
Indeed, we have a divisor $W$ on $X\times P$, where $P=\prod_i\pp(\Phi_i)$ parameterizes (combinations of) elements of our linear systems.
The fiber over the generic point of $P$ is $Z$, while the fiber over some special point is $G$.
Let $R$ be a $d.v.r.$ with the fraction field $K$ and residue field $\kappa$, and $W_R$ be some divisor on $X\times\op{Spec}(R)$, with
fibers $W_K$ and $W_{\kappa}$ over the generic and closed point of $\op{Spec}(R)$, respectively.
Having a closed point ${\cal T}$ of $W_K$, consider the closure of it in $W$.
We get a proper morphism $f:T\row\op{Spec}(R)$ of relative dimension
zero, whose fiber over $\op{Spec}(\kappa)$ consists of points $t_l$ with multiplicities $e_l$ - these are the specializations
of ${\cal T}$. The specialization of a singular point is singular, and specialization of a point where components are not transversal has
the same property. At the same time,
$$
[k(T):K]=\sum_l[k(t_l):\kappa]\cdot e_l.
$$
So, if $[k(T):K]$ is not divisible by $n$, then one of $[k(t_l):\kappa]$ is.
This shows that $Z$ should be a divisor with strict normal crossings at every point which is {\it not $n$-anisotropic}.
Since our linear system $\Phi_i$ is very ample on $X\backslash S_i$, the divisor $Z$ has strict normal crossings outside $S$.
By the same reason and by Statement \ref{alg-geom-Lef}, we have the surjection $\CH^*(X)\twoheadrightarrow\CH^*(Z_i\backslash S_i)$.
\Qed
\end{proof}

This permits to represent some cycles of co-dimension $2$ by single components.

Everywhere below we will denote:
$$
\Ch^*_{k/k}(X):=\CH^*(X,\zz/n)/(n-\text{anisotropic classes}),
$$
where $n$ is some natural number (which should be clear from context).

\begin{statement}
\label{LS-codim2-single}
Let $S=\cup_iS_i$ be the union of smooth connected transversal subvarieties of codimension $2$ on a smooth projective variety
and $n\in\nn$.
Suppose, that all the intersections $S_i\cap S_j$ are $n$-anisotropic. Then over some f.g. purely transcendental extension of $k$
there exists a blow-up $\mu:\ov{X}\row X$, such that
$\mu^*([S])\in\Ch^*_{k/k}(\ov{X})$ is represented by the class of a smooth connected variety
$\ov{S}$ whose characteristic classes in $\Ch^*_{k/k}$ are $\mu^*$ of the characteristic classes of $S$.
Moreover, if $S=\cup_iS_i$ and $T=\cup_jT_j$ are two subvarieties as above which are transversal to each other, then there exists
a blow-up $\mu:\ov{X}\row X$ with the properties as above for both $S$ and $T$ and such that $\ov{S}$ and $\ov{T}$ are smooth
connected, transversal to each other, and $\ov{S}\cap\ov{T}$ is connected.
\end{statement}

\begin{proof}
Let $\pi:\wt{X}\row X$ be the blow-up in all the components $S_i$ of $S$, with $E_i\stackrel{\pi_i}{\lrow}S_i$
the components of the special divisor and
$\rho_i=-[E_i]$. Then, by \cite[Prop. 6.7]{Fu}, $\pi^*([S])=[F]$, where $F=\cup_i F_i$, $F_i$ is supported on $E_i$ and
$[F_i]=[E_i](\rho_i+\pi_i^*(c_1(N_{S_i\subset X})))$. Note, that $E=\cup_iE_i$ is a divisor with strict normal crossings, with
all the intersections $E_i\cap E_j$ $n$-anisotropic. Note, that $c_1(N_{F_i\subset\wt{X}})=\pi_i^*(c_1(N_{S_i\subset X}))$
and $c_2(N_{F_i\subset\wt{X}})=\pi_i^*(c_2(N_{S_i\subset X}))$. Hence, the same is true about all other characteristic classes.

Since $E$ is smooth outside an $n$-anisotropic subscheme, by Statement \ref{LS-div-conn}
(where we consider $E$ as a single component), over
some f.g. purely transcendental extension of $k$, there is an irreducible divisor $Z$, containing $F$,
smooth outside an $n$-anisotropic closed subscheme of $F$.
Let $\eps:\ov{X}\row\wt{X}$ be an embedded desingularization of $Z$. Let $\ov{Z}$ and $\ov{F}$ be the proper pre-images of $Z$ and $F$.
Then $\eps^*:\Ch^*_{k/k}(\wt{X})\stackrel{=}{\row}\Ch^*_{k/k}(\ov{X})$ is an isomorphism and
$\eps^*([F])=[\ov{F}]\in\Ch^2_{k/k}(\ov{X})$ is supported on the smooth connected divisor $\ov{Z}$.
By adding an $n$-multiple of a very ample divisor, we can substitute $[\ov{F}]$ by a very ample divisor on $\ov{Z}$.
Let $\ov{S}$ be the generic representative of the linear system $|\ov{F}|$ on $\ov{Z}$. Then $\ov{S}$ is smooth and connected.
Since, modulo $n$-anisotropic subvarieties, $\ov{X}$ coincides with $X$ and $\ov{F}$ with $F$, the characteristic classes of $\ov{S}$ in
$\Ch^*_{k/k}(\ov{X})$ are $\eps^*$ of the respective classes of $F$.

For the pair of subvarieties $S=\sum_iS_i$ and $T=\sum_jT_j$, consider the blow-up $\pi:\wt{X}\row X$ at all the components of $S$ and $T$.
Then the special divisor $E=E_S\cup E_T$ has strict normal crossings, where $E_S$ and $E_T$ are smooth outside $n$-anisotropic subschemes.
Applying Statement \ref{LS-div-conn} to $F_S\cup F_T$ contained in $E_S\cup E_T$ (where we consider $E_S$ and $E_T$ as single components),
we obtain a divisor $Z=Z_S\cup Z_T$ containing $S\cup T$,
with strict normal crossings outside an $n$-anisotropic subscheme,
with irreducible $Z_S,Z_T$ and $Z_S\cap Z_T$. Resolving $n$-anisotropic singularities and non-transversalities of $Z$, as above,
we obtain the needed smooth connected transversal subvarieties $\ov{S}$ and $\ov{T}$ having the needed characteristic numbers.
Since $Z_S\cap Z_T$ is irreducible, the intersection $\ov{S}\cap\ov{T}$ is connected.
\Qed
\end{proof}

The next statement represents an elementary block with the help 
of which we will ``deform'' the chains of 
co-dimension $1$ embeddings of irreducible varieties.

\begin{statement}
\label{LS-gen-step}
Let $n\in\nn$, $X$ be projective irreducible variety, smooth outside an $n$-anisotropic closed subscheme,
and $S\subset Z\subset X$ be embeddings of co-dimension
$1$ of irreducible subvarieties, smooth outside $n$-anisotropic closed subschemes and such that $S$ is not $n$-anisotropic.
Then, over some f.g. purely transcendental extension of $k$,
there exists $Z'$, such that $S\subset Z'\subset X$ has the same properties, $[Z']=[Z]$ in $\Ch^1_{k/k}(X)$ and the restriction
$\Ch^*_{k/k}(X)\twoheadrightarrow\Ch^*_{k/k}(Z'\backslash S)$ is surjective.
\end{statement}

\begin{proof}
This is a particular case of Statement \ref{LS-div-conn}, aside from the fact that we permit $X$ to have anisotropic singularities.
But the same proof works.
Outside some closed $n$-anisotropic subscheme $T$ of $X$, $Z$ is a Cartier divisor and,
by the Statement \ref{LS-Sbs}, there exists a very ample divisor $D$ on $X\backslash T$, such that the linear system
$\Phi=|Z+n\cdot D|_{S\backslash S\cap T}$ on $X\backslash T$ consists of very ample divisors and has the base set $S\backslash T$.
This linear system defines an embedding
of $X\backslash (S\cup T)$ into a projective space.
Let $Z'$ be the closure in $X$ of the generic element of this linear system
(defined over $k(\pp(\Phi))$). Since $S$ is {\it not $n$-anisotropic}, $S\backslash S\cap T$ is non-empty, and so, $Z'$ contains $S$.
Clearly, $[Z']=[Z]\in\Ch^1_{k/k}(X)=\Ch^1_{k/k}(X\backslash T)$.
By Statement \ref{alg-geom-Lef}, we have a surjection $\Ch^*_{k/k}(X)\twoheadrightarrow\Ch^*_{k/k}(Z'\backslash S)$,
and $Z'$ is irreducible. The same arguments as in the proof of Statement \ref{LS-div-conn} show that $Z'$ has only $n$-anisotropic
singularities.
\Qed
\end{proof}

The previous result permits to deform the chains of co-dimension
1 embeddings in such a way that isotorpic Chow groups of a term of the new chain would be covered by those of the previous (ambient) term modulo such groups of the next (smaller) term of the original chain. Later it will enable us, subject to certain conditions, to make numerically trivial classes anisotropic. 

\begin{statement}
\label{LS-gen-constr}
Let $n\in\nn$ and
$X_r\stackrel{j_r}{\row} X_{r-1}\stackrel{j_{r-1}}{\row}\ldots\stackrel{j_2}{\row} X_1\stackrel{j_1}{\row} X_0$
be embeddings of co-dimension $1$ of irreducible subvarieties, smooth outside $n$-anisotropic closed subschemes, with
$X_r$- not $n$-anisotropic.
Then, over some f.g. purely transcendental extension of $k$, it can be complemented to a commutative diagram:
$$
\xymatrix{
X'_r\ar[r]^{j'_r} & X'_{r-1}\ar[r]^{j'_{r-1}} & X'_{r-2}\ar[r]^{j'_{r-2}} & \ldots & \ldots \ar[r]^{j'_2} & X'_1 \ar[r]^{j'_1} & X_0 \\
X_r\ar[r]^{j_r} \ar[ru]^(0.4){g_r} & X_{r-1}\ar[r]^{j_{r-1}} \ar[ru]^(0.4){g_{r-1}} & X_{r-2}\ar[r]^{j_{r-2}} \ar[ru]^(0.4){g_{r-2}}
& \ldots & \ldots \ar[r]^{j_2} \ar[ru]^-{g_2} & X_1 \ar[r]^{j_1} &
X_0 \ar@{=}[u]
}
$$
where all maps are embeddings of co-dimension $1$ of irreducible subvarieties, smooth outside $n$-anisotropic closed subschemes and
for any $i$, the map
$$
((j'_i)^*,(g_{i+1})_*):\Ch^*_{k/k}(X'_{i-1})\oplus\Ch^{*-1}_{k/k}(X_{i+1})\twoheadrightarrow\Ch^*_{k/k}(X'_i)
$$
is surjective (for $i=r$, the map $(j'_r)^*$ is surjective), and $[X'_{i+1}]=[X_{i+1}]\in\Ch^1_{k/k}(X'_i)$.
\end{statement}

\begin{proof}
This follows from the inductive application of the Statement \ref{LS-gen-step} from top to the bottom.
Finally, in the last step, we take $X'_r$ to be the closure of the generic representative of the respective (very ample)
linear system $|X_r+n\cdot D|$ without any base-set. Since $X_r$ is {\it not $n$-anisotropic}, $X'_r$ is non-empty and irreducible.
By Statement \ref{alg-geom-Lef}, we have a surjection $(j'_r)^*:\Ch^*_{k/k}(X'_{r-1})\twoheadrightarrow\Ch^*_{k/k}(X'_r)$
\Qed
\end{proof}

Note, that although, $X'_{r-1}$ is still {\it not $n$-anisotropic}, $X'_r$ may be, in principle, anisotropic.

We also have a "smooth" version of the above result which is
what we will use below.

\begin{statement}
\label{LS-gen-smooth}
Let $X_0$ be a smooth projective connected variety,
and $X_r\stackrel{j_r}{\row} X_{r-1}\stackrel{j_{r-1}}{\row}\ldots\stackrel{j_2}{\row} X_1\stackrel{j_1}{\row} X_0$
be regular embeddings of co-dimension $1$ of connected varieties, with $X_r$-not $n$-anisotropic.
Then, over some f.g. purely transcendental extension of $k$, it can be complemented to a commutative diagram:
$$
\xymatrix{
X'_r\ar[r]^{j'_r} & X'_{r-1}\ar[r]^{j'_{r-1}} & X'_{r-2}\ar[r]^{j'_{r-2}} & \ldots & \ldots \ar[r]^{j'_2} & X'_1 \ar[r]^{j'_1} &
X'_0 \ar[dd]^{\pi_0} \\
\wt{X}_r\ar[d]^{\pi_r} \ar[ru]^{g_r} \ar[r]^{\wt{j}_r} & \wt{X}_{r-1}\ar[d]^{\pi_{r-1}} \ar[ru]^{g_{r-1}} \ar[r]^{\wt{j}_{r-1}} &
\wt{X}_{r-2}\ar[d]^{\pi_{r-2}} \ar[ru]^{g_{r-2}} \ar[r]^{\wt{j}_{r-2}} & \ldots & \ldots \ar[r]^{\wt{j}_2} &
\wt{X}_1 \ar[d]^{\pi_1} \ar[ru]^{g_1} & \\
X_r\ar[r]^{j_r} & X_{r-1}\ar[r]^{j_{r-1}} & X_{r-2}\ar[r]^{j_{r-2}}
& \ldots & \ldots \ar[r]^{j_2} & X_1 \ar[r]^{j_1} &
X_0
}
$$
where upper and lower horizontal maps are regular embeddings of co-dimension $1$ of connected varieties, while
the vertical ones are blow-ups in $n$-anisotropic centers.
In particular, the maps $\pi_i^*:\Ch^*_{k/k}(X_i)\stackrel{=}{\row}\Ch^*_{k/k}(\wt{X}_i)$ are isomorphisms.
Also, for any $i$, the map
$$
((j'_i)^*,(g_{i+1})_*\pi_{i+1}^*):\Ch^*_{k/k}(X'_{i-1})\oplus\Ch^{*-1}_{k/k}(X_{i+1})
\twoheadrightarrow\Ch^*_{k/k}(X'_i)
$$
is surjective (for $i=r$, the map $(j'_r)^*$ is surjective), and $[X'_{i+1}]=(g_{i+1})_*[\wt{X}_{i+1}]\in\Ch^1_{k/k}(X'_i)$.
\end{statement}

\begin{proof}
By Statement \ref{LS-gen-constr}, we get the commutative diagram
$$
\xymatrix{
\ov{X}_r\ar[r]^{\ov{j}_r} & \ov{X}_{r-1}\ar[r]^{\ov{j}_{r-1}} & \ov{X}_{r-2}\ar[r]^{\ov{j}_{r-2}} & \ldots & \ldots \ar[r]^{\ov{j}_2} &
\ov{X}_1 \ar[r]^{\ov{j}_1} &
X_0 \\
X_r\ar[r]^{j_r} \ar[ru]^(0.4){\ov{g}_r} & X_{r-1}\ar[r]^{j_{r-1}} \ar[ru]^(0.4){\ov{g}_{r-1}} &
X_{r-2}\ar[r]^{j_{r-2}} \ar[ru]^(0.4){\ov{g}_{r-2}}
& \ldots & \ldots \ar[r]^{j_2} \ar[ru]^(0.4){\ov{g}_2} & X_1 \ar[r]^{j_1} &
X_0 \ar@{=}[u]
}
$$
where $\ov{X}_i$ are irreducible varieties smooth outside some closed proper $n$-anisotropic subschemes, the maps
$$
((\ov{j}_i)^*,(\ov{g}_{i+1})_*):\Ch^*_{k/k}(\ov{X}_{i-1})\oplus\Ch^{*-1}_{k/k}(X_{i+1})
\twoheadrightarrow\Ch^*_{k/k}(\ov{X}_i)
$$
are surjective, and $[\ov{X}_{i+1}]=[X_{i+1}]\in\Ch^1(\ov{X}_i)$.

Let $\pi_0:X'_0\row X_0$ be the embedded
desingularization of $\ov{X}_r\subset\ov{X}_{r-1}\subset\ov{X}_{r-2}\subset\ldots\subset\ov{X}_1\subset\ov{X}_0$,
and $\eps_i:X'_i\row\ov{X}_i$ be the proper pre-images (with $\eps_0=\pi_0$). Since special divisors are $n$-anisotropic, we have
isomorphisms $\eps_i^*:\Ch^*_{k/k}(\ov{X}_i)\stackrel{=}{\row}\Ch^*_{k/k}(\wt{X}_i)$.
Noting that $X_i$ is {\it not $n$-anisotropic}, by blowing $X_i$ at $n$-anisotropic centers,
we may resolve the indeterminacies of the maps
$X_i\stackrel{\ov{g}_i}{\row}\ov{X}_{i-1}\stackrel{\eps_{i-1}^{-1}}{\dashrightarrow}X'_{i-1}$ and
$X_i\stackrel{j_i}{\row}X_{i-1}\stackrel{\pi_{i-1}^{-1}}{\dashrightarrow}\wt{X}_{i-1}$
and obtain commutative squares
$$
\xymatrix{
\wt{X}_{i-1} \ar[d]_{\pi_{i-1}} & \wt{X}_i \ar[d]_{\pi_i} \ar[r]^{g_i} \ar[l]_{\wt{j}_i}& X'_{i-1} \ar[d]^{\eps_{i-1}} \\
X_{i-1} & X_i \ar[r]^{\ov{g}_i} \ar[l]_{j_i} & \ov{X}_{i-1}
}
$$
and the needed commutative diagram.
Since the maps $\pi_i^*:\Ch^*_{k/k}(X_i)\stackrel{=}{\row}\Ch^*_{k/k}(X'_i)$ are isomorphisms, the maps
$$
((j'_i)^*,(g_{i+1})_*\pi_{i+1}^*):\Ch^*_{k/k}(X'_{i-1})\oplus\Ch^{*-1}_{k/k}(X_{i+1})
\twoheadrightarrow\Ch^*_{k/k}(X'_i)
$$
are surjective (for $i=r$, the map $(j'_r)^*$ is surjective), and $[X'_{i+1}]=(g_{i+1})_*[\wt{X}_{i+1}]\in\Ch^1_{k/k}(X'_i)$.
Finally, since $[\ov{X}_i]=(\ov{g}_i)_*[X_i]\in\Ch^1_{k/k}(\ov{X}_{i-1})$, we obtain that
$[X'_i]=(g_i)_*[\wt{X}_i]\in\Ch^1_{k/k}(X'_{i-1})$.
\Qed
\end{proof}

In the next key statement,
applying the above result repeatedly, we will deform a given
chain of codimension $1$ regular embeddings keeping the classes 
of all the subvarieties (of the chain) 
in $\Ch_{k/k}^*(X)$
unchanged, but making the image of $\Ch_{k/k}^*(X_r)$ (the smallest subvariety) in $\Ch_{k/k}^*(X)$  a 
submodule generated by monomials in the $1$-st
Chern classes of normal bundles of the (original) chain.
After that, to make $X_r$ anisotropic, it will remain only 
to eliminate the mentioned monomials numerically.

Let $\vec{l}=(l_2,\ldots,l_r)$ be a vector of non-negative integers.
We say that $\vec{l}$ is $i$-{\it{good}}, if there exists an $i+1\leq s\leq r+1$ such that
$l_k>0$ for $i+1\leq k<s$, while $l_k=0$, for $k\geq s$. Any 
$i$-good vector is $(i+1)$-good and every vector is $r$-good,
so we get a filtration.

\begin{statement}
\label{LS-porozhd}
Let $X_r\stackrel{j_r}{\row} X_{r-1}\stackrel{j_{r-1}}{\row}\ldots\stackrel{j_2}{\row} X_1\stackrel{j_1}{\row} X_0$
be regular embeddings of co-dimension $1$ of smooth connected varieties.
Then over some f.g. purely transcendental extension there exists a blow-up in $n$-anisotropic centers $\hat{X}_0\row X_0$
and a similar sequence of embeddings
$\hat{X}_r\stackrel{\hat{j}_r}{\row} \hat{X}_{r-1}\stackrel{\hat{j}_{r-1}}{\row}\ldots\stackrel{\hat{j}_2}{\row}
\hat{X}_1\stackrel{\hat{j}_1}{\row} \hat{X}_0$, where $[\hat{X}_r]=[X_r]\in\Ch^*_{k/k}(X_0)$
and the image of the restriction $f_j^*:\Ch^*_{k/k}(\hat{X}_j)\row\Ch^*_{k/k}(\hat{X}_r)$ as a
$\Ch^*_{k/k}(X_0)=\Ch^*_{k/k}(\hat{X}_0)$-module is generated
by monomials $\hat{c}^{\vec{l}}=\prod_{i=2}^rc_1^{l_i}(\hat{N}_i)$, for $j$-good $\vec{l}$,
where $\hat{N}_i=N_{\hat{X}_i\subset\hat{X}_{i-1}}$, and 
the image of the map $(f_0)_*f_j^*:\Ch^*_{k/k}(\hat{X}_j)\row\Ch^{*+r}_{k/k}(\hat{X}_0)$ as a $\Ch^*_{k/k}(X_0)$-module is generated
by elements $c^{\vec{l}}\cdot [X_r]=\prod_{i=2}^rc_1^{l_i}(N_i)\cdot [X_r]$, 
where $N_i=N_{X_i\subset X_{i-1}}$ and $\vec{l}$ runs over all $j$-good vectors.
(here $f_i:\hat{X}_r\row\hat{X}_i$ is the embedding).
In particular, the image of $(f_0)_*:\Ch^*_{k/k}(\hat{X}_r)\row\Ch^{*+r}_{k/k}(\hat{X}_0)$ as a $\Ch^*_{k/k}(X_0)$-module is generated
by elements $c^{\vec{l}}\cdot [X_r]$, where $\vec{l}$ runs through all vectors.
\end{statement}

\begin{proof}
Let us denote the original sequence as
$X^0_r\stackrel{j^0_r}{\row} X^0_{r-1}\stackrel{j^0_{r-1}}{\row}\ldots\stackrel{j^0_2}{\row} X^0_1\stackrel{j^0_1}{\row} X^0_0$.
Either $X^0_r$ is $n$-anisotropic, in which case there is nothing to prove, or we can produce a diagram as in
Statement \ref{LS-gen-smooth}.
We can iterate this process as long as the variety $X^m_r$ is {\it not $n$-anisotropic} and obtain diagrams:
$$
\xymatrix{
X^{m+1}_r\ar[r]^{j^{m+1}_r} & X^{m+1}_{r-1}\ar[r]^{j^{m+1}_{r-1}} & X^{m+1}_{r-2}\ar[r]^{j^{m+1}_{r-2}} &
\ldots & \ldots \ar[r]^{j^{m+1}_2} & X^{m+1}_1 \ar[r]^{j^{m+1}_1} & X^{m+1}_0 \ar[dd]^{\pi^m_0} \\
\wt{X}^m_r\ar[d]^{\pi^m_r} \ar[ru]^{g^{m+1}_r} \ar[r]^{\wt{j}^m_r} & \wt{X}^m_{r-1}\ar[d]^{\pi^m_{r-1}} \ar[ru]^{g^{m+1}_{r-1}} \ar[r]^{\wt{j}^m_{r-1}} &
\wt{X}^m_{r-2}\ar[d]^{\pi^m_{r-2}} \ar[ru]^{g^{m+1}_{r-2}} \ar[r]^{\wt{j}^m_{r-2}} & \ldots & \ldots \ar[r]^{\wt{j}^m_2} &
\wt{X}^m_1 \ar[d]^{\pi^m_1} \ar[ru]^{g^{m+1}_1} & \\
X^m_r\ar[r]^{j^m_r} & X^m_{r-1}\ar[r]^{j^m_{r-1}} & X^m_{r-2}\ar[r]^{j^m_{r-2}}
& \ldots & \ldots \ar[r]^{j^m_2} & X^m_1 \ar[r]^{j^m_1} &
X^m_0
}
$$
These induce maps on isotropic Chow groups:
$$
\xymatrix{
\Ch^*_{k/k}(X^{m+1}_r) & \Ch^*_{k/k}(X^{m+1}_{r-1}) \ar[l]_{\alpha^{m+1}_{r}} &
\Ch^*_{k/k}(X^{m+1}_{r-2}) \ar[l]_-{\alpha^{m+1}_{r-1}} &
\ldots \ar[l]_-{\alpha^{m+1}_{r-2}} & \Ch^*_{k/k}(X^{m+1}_1) \ar[l]_-{\alpha^{m+1}_2} &
\Ch^*_{k/k}(X^{m+1}_0)  \ar[l]_-{\alpha^{m+1}_1} \ar@{=}[d] \\
\Ch^*_{k/k}(X^m_r) \ar[ru]^(0.4){\beta^{m+1}_r}_(0.7){[1]} &
\Ch^*_{k/k}(X^m_{r-1})\ar[l]_{\alpha^m_{r}} \ar[ru]^(0.35){\beta^{m+1}_{r-1}}_(0.7){[1]} &
\Ch^*_{k/k}(X^m_{r-2})\ar[l]_-{\alpha^m_{r-1}} \ar[ru]^(0.35){\beta^{m+1}_{r-2}}_(0.7){[1]}
& \ldots \ar[l]_-{\alpha^m_{r-2}} & \Ch^*_{k/k}(X^m_1) \ar[l]_-{\alpha^m_2} \ar[ru]^(0.4){\beta^{m+1}_1}_(0.7){[1]} &
\Ch^*_{k/k}(X^m_0) \ar[l]_-{\alpha^m_1}
}
$$
where $\alpha^l_i=(j^l_i)^*$, $\beta^l_i=(g^l_i)_*(\pi^{l-1}_i)^*$ and $\beta$-maps shift the codimension by $(+1)$.
The maps
$$
(\alpha^l_i,\beta^l_{i+1}):\Ch^*_{k/k}(X^l_{i-1})\oplus\Ch^*_{k/k}(X^{l-1}_{i+1})\twoheadrightarrow\Ch^*_{k/k}(X^l_i)
$$
are surjective. Either at some stage we will get an $n$-anisotropic $X^m_r$, in which case we are done, or we can iterate the process
$q=\ddim(X)+1$ times. Set $\hat{X}_i=X^q_i$, etc. Then $\Ch^*_{k/k}(\hat{X}_i)$ is generated by the elements of the form
$\omega(x)$, where $\omega$ is a composition of $\alpha$'s and $\beta$'s and $x\in\Ch^*_{k/k}(X_0)$. Here we are using the fact
that the number of $\beta$'s in such a composition can't be more than $\ddim(X_i)$ (as each $\beta$ increases the codimension by $1$),
and so, the chain has to start with $X_0$. We also have maps
$\gamma^l_i=(j^l_i)_*$ and
$\delta^l_i=(\pi^{l-1}_i)_*(g^l_i)^*$ fitting commutative diagrams (recall, that $(\pi^m_i)^*$ and $(\pi^m_i)_*$ are isomorphisms)
$$
\xymatrix{
\Ch^*_{k/k}(X^{m+1}_i) \ar[r]^-{\gamma^{m+1}_{i}} &
\Ch^*_{k/k}(X^{m+1}_{i-1}) \\
\Ch^*_{k/k}(X^m_{i+1}) \ar[r]^-{\gamma^m_{i+1}} \ar[u]^-{\beta^{m+1}_{i+1}} &
\Ch^*_{k/k}(X^m_{i}) \ar[u]^-{\beta^{m+1}_{i}}
}
\hspace{3mm}\text{and}\hspace{3mm}
\xymatrix{
\Ch^*_{k/k}(X^{m+1}_i) \ar[d]_-{\delta^{m+1}_{i+1}} &
\Ch^*_{k/k}(X^{m+1}_{i-1}) \ar[l]_-{\alpha^{m+1}_{i}} \ar[d]_-{\delta^{m+1}_{i}}\\
\Ch^*_{k/k}(X^m_{i+1}) &
\Ch^*_{k/k}(X^m_{i})\ar[l]_-{\alpha^m_{i+1}}
}
$$
Note, that
\begin{equation*}
\begin{split}
&\beta^{m+1}_i\delta^{m+1}_i(u)=u\cdot c_1(N_{\wt{X}^m_i\row X^{m+1}_{i-1}})=u\cdot c_1(N_{X^{m+1}_i\subset X^{m+1}_{i-1}})=
\gamma^{m+1}_i\alpha^{m+1}_i(u)\hspace{5mm}\text{and}\\
&\delta^{m+1}_i\beta^{m+1}_i(v)=v\cdot c_1(N_{\wt{X}^m_i\row X^{m+1}_{i-1}})=v\cdot c_1(N_{X^m_i\subset X^m_{i-1}})=
\alpha^m_i\gamma^m_i(v)
\end{split}
\end{equation*}
Using these relations, one can reduce $\omega(x)$ to the form $\theta(x)$, where $\theta$ is a combination of $\alpha^q_i$'s
and $\gamma^q_j$'s. The restriction of such an element to $\Ch^*_{k/k}(X^q_r)$ is $f_0^*(x)$ times a monomial in
$c_1(N_{X^q_i\subset X^q_{i-1}})=c_1(\hat{N}_i)$'s, where each factor $c_1(\hat{N}_i)$ corresponds to a loop
$\alpha^q_i\gamma^q_i$ in $\theta$. Thus, the image $f^*_j:\Ch^*_{k/k}(\hat{X}_j)\row\Ch^*_{k/k}(\hat{X}_r)$
as a $\Ch^*_{k/k}(X_0)$-module is generated by monomials in $c_1(\hat{N}_i)$'s. Since such a monomial corresponds to a closed path
from $X^q_j$ to itself, these will be exactly $j$-{\it{good}} monomials.
Also, we need to observe that $c_1(\hat{N}_1)=c_1(N_{\hat{X}_1\subset\hat{X}_0})$ is the restriction of a class from $\hat{X}_0$.
Finally, from the same relations,
we get that $(f_0)_*\hat{c}^{\vec{n}}=\gamma^q_1\ldots\gamma^q_r\hat{c}^{\vec{n}}=\gamma^0_1\ldots\gamma^0_rc^{\vec{n}}=
c^{\vec{n}}\cdot [X_r]$. Hence, the image of $(f_0)_*(f_j)^*:\Ch^*_{k/k}(\hat{X}_j)\row
\Ch^{*+r}_{k/k}(X_0)$ as a
$\Ch^*_{k/k}(X_0)$-module is generated by elements $c^{\vec{l}}\cdot [X_r]$, where $\vec{l}$ runs over
all $j$-{\it good} vectors.
\Qed
\end{proof}

With the previous result in hands we get a practical tool
ensuring the anisotropy of Chow group elements.

\begin{corollary}
\label{LS-nump-anis}
Let $X_r\stackrel{j_r}{\row} X_{r-1}\stackrel{j_{r-1}}{\row}\ldots\stackrel{j_2}{\row} X_1\stackrel{j_1}{\row} X_0$
be regular embeddings of co-dimension $1$ of smooth connected varieties with $N_i=N_{X_i\subset X_{i-1}}$.
Suppose, that $c^{\vec{l}}\cdot [X_r]\numeq 0$ on $X_0$, for all monomials in $c_1(N_i)$'s, $i\geq 2$.
Then, over some f.g. purely
transcendental extension, $[X_r]=0\in\Ch^r_{k/k}(X_0)$.
\end{corollary}

\begin{proof}
By Statement \ref{LS-porozhd}, over some f.g. purely transcendental extension of $k$, there exists a blow-up with $n$-anisotropic centers
$\pi:\hat{X}\row X$ and a sequence $\hat{X}_r\stackrel{\hat{j}_r}{\row} \hat{X}_{r-1}\stackrel{\hat{j}_{r-1}}{\row}\ldots\stackrel{\hat{j}_2}{\row}
\hat{X}_1\stackrel{\hat{j}_1}{\row} \hat{X}_0$ of regular embeddings of smooth
connected projective varieties, such that
$[\hat{X}_r]=[X_r]\in\Ch^r_{k/k}(\hat{X}_0)=\Ch^r_{k/k}(X_0)$ and
the image of the map
$(f_0)_*:\Ch^*_{k/k}(\hat{X}_r)\row\Ch^{*'}_{k/k}(\hat{X}_0)$
as a module over $\Ch^*_{k/k}(X_0)$ is generated by $c^{\vec{l}}\cdot [X_r]$, for all 
$r$-good (=all) $\vec{l}$, where $f_0:\hat{X}_r\row\hat{X}_0$.
Since all these classes are $\numeq 0$
on $\hat{X}_0$, the $0$-dimensional component of our image is zero. This means that $\hat{X}_r$ is anisotropic and so, 
$[\hat{X}_r]=0\in\Ch^{r}_{k/k}(\hat{X}_0)$.
\Qed
\end{proof}

In the case $r=2$, we get the following:

\begin{corollary}
\label{LS-r2-nump-div}
Let $S\subset X$ be a regular embedding of codimension $2$ of smooth connected projective varieties.
Suppose, $c_1^m(N_{S\subset X})\cdot [S]\numeq 0$ on $X$, for any $m\geq 0$.
Then, over some f.g. purely transcendental extension, $[S]=0\in\Ch^2_{k/k}(X)$.
\end{corollary}

\begin{proof}
By blowing $X$ at $S$ we may assume that $S$ is contained in a smooth connected divisor $Z$. Note, that the "new" characteristic classes
of $S$ are pull-backs of the "old" ones, and so, are numerically trivial as well. We obtain the triple $S\row Z\row X$.
Our statement now is a particular case of Corollary \ref{LS-nump-anis}, where
it remains to observe that $c_1(N_{S\subset X})=c_1(N_{S\subset Z})+c_1(N_{Z\subset X})$, where the second
summand is defined on $X$.
\Qed
\end{proof}

\bigskip

\begin{itemize}
\item[address:] {\small School of Mathematical Sciences, University of Nottingham, University Park, Nottingham, NG7 2RD, UK}
\item[email:] {\small\ttfamily alexander.vishik@nottingham.ac.uk}
\end{itemize}

\end{document}